\documentclass[oneside]{amsart}
\usepackage[utf8]{inputenc}
\usepackage{import}
\input{preamble.sty}
\usepackage[backend=bibtex,style=alphabetic,sorting=nty,doi=false,isbn=false,url=false]{biblatex}
\title[Separability of products in relatively hyperbolic groups]{Quasiconvexity of virtual joins and separability of products in relatively hyperbolic groups}

\author{Ashot Minasyan, Lawk Mineh}
\address{CGTA, School of Mathematical Sciences,
University of Southampton, Highfield, Southampton, SO17~1BJ, United
Kingdom.}
\address{Mathematical Institute, University of Bonn,
53115 Bonn, Germany.}
\email{aminasyan@gmail.com, lawk@math.uni-bonn.de}

\begin{document}

\begin{abstract}
A relatively hyperbolic group $G$ is said to be QCERF if all finitely generated relatively quasiconvex subgroups are closed in the profinite topology on $G$. 

Assume that $G$ is a QCERF relatively hyperbolic group with double coset separable (eg, virtually polycyclic) peripheral subgroups. Given any two finitely generated relatively quasiconvex subgroups $Q,R \leqslant G$ we prove the existence of finite index subgroups $Q'\leqslant_f Q$ and $R' \leqslant_f R$ such that the join $\langle Q',R'\rangle$ is again relatively quasiconvex in $G$. We then show that, under the minimal necessary hypotheses on the peripheral subgroups, products of finitely generated relatively quasiconvex subgroups are closed in the profinite topology on \(G\). From this we obtain the separability of products of finitely generated subgroups for several classes of groups, including limit groups, Kleinian groups and balanced fundamental groups of finite graphs of free groups with cyclic edge groups.
\end{abstract}

\keywords{Relatively hyperbolic groups, relatively quasiconvex subgroups, virtual joins, double coset separability, product separability, limit groups, Kleinian groups}
\subjclass[2020]{20F67,20F65,20E26,20H10}

\maketitle

\setcounter{tocdepth}{1}
\tableofcontents

\section{Introduction}
Any group can be equipped with the \emph{profinite topology}, whose basic open sets are cosets of finite index subgroups. A subset of a group is said to be \emph{separable} if it is closed in the profinite topology. 
The trivial subgroup of a group $G$ is separable if and only if the profinite topology is Hausdorff; in this case $G$ is said to be \emph{residually finite}. 
If every finitely generated subgroup of $G$ is separable then $G$ is called \emph{LERF} (or \emph{subgroup separable}), and if the product of any two finitely generated subgroups is separable, $G$ is said to be \emph{double coset separable}. 

In this paper we will be interested in various separability properties of relatively hyperbolic groups. 
The notion of a relatively hyperbolic group was originally suggested by Gromov \cite{Gromov1987} as a generalisation of word hyperbolic groups. 
The concept was further developed by Farb \cite{FarbRHG}, Bowditch \cite{BowditchRHG}, Dru\c{t}u-Sapir \cite{DS}, Osin \cite{OsinRHG} and Groves-Manning \cite{GrovesManning}, whose various definitions were later shown to be equivalent by Hruska \cite{HruskaRHCG}. 
Relative hyperbolicity is a relative property of a group \(G\) in the sense that one must specify a collection of \emph{peripheral subgroups} \(\{H_\nu \mid \nu \in \Nu\}\) with respect to which \(G\) is relatively hyperbolic (see Definition~\ref{def:rh_gp}). 
Typical examples of relatively hyperbolic groups include geometrically finite Kleinian groups, fundamental groups of finite volume manifolds of pinched negative curvature, and small cancellation quotients of free products.
Respectively, these groups are hyperbolic relative to their maximal parabolic subgroups, their cusp subgroups and the images of the free factors (see, for example, \cite{OsinRHG}).

\subsection{Quasiconvexity of virtual joins}
Since general finitely generated subgroups of word hyperbolic (relatively hyperbolic) groups can be quite wild and
need not be separable, it is customary to restrict one's attention to quasiconvex (respectively, relatively quasiconvex) subgroups. 

\emph{Quasiconvex subgroups} play a central role in the study of word hyperbolic groups.  
They are precisely the finitely generated quasi-isometrically embedded subgroups, and, hence, they are hyperbolic themselves and are generally well-behaved. 

If $Q$ and $R$ are two quasiconvex subgroups of a hyperbolic group $G$ then the intersection $S=Q \cap R$ is also quasiconvex (see, for example, Short \cite{Short}) but the join $\langle Q, R \rangle$ need not be. 
This can be remedied by considering a \emph{virtual join} of $Q$ and $R$, which is defined as $\langle Q',R' \rangle$, for some finite index subgroups $Q' \leqslant_f Q$ and $R' \leqslant_f R$. 
The existence of a quasiconvex virtual join  $\langle Q',R' \rangle$ 
was proved by Gitik \cite{Gitik_ping-pong} under the assumption that $S=Q \cap R$ is separable in $G$. 
More precisely, Gitik's theorem states that there exist finite index subgroups $Q' \leqslant_f Q$ and $R' \leqslant_f R$ such that $Q'\cap R'=S$ and the virtual join $\langle Q',R' \rangle$ is quasiconvex in $G$; moreover, $\langle Q',R' \rangle$ will be naturally isomorphic to the amalgamated free product $Q'*_S R'$. This theorem was an important ingredient in the proof that double cosets of quasiconvex subgroups are separable in LERF hyperbolic groups (see \cite{Gitik-double_coset_sep,MinGFERF}).

In the setting of relatively hyperbolic groups, the natural sub-objects are the \emph{relatively quasiconvex subgroups}, which are themselves relatively hyperbolic in a way that is compatible with the ambient group. 
Basic examples of relatively quasiconvex subgroups are \emph{maximal parabolic subgroups} (that is, conjugates of the peripheral subgroups), \emph{parabolic subgroups} (subgroups of maximal parabolics) and finitely generated undistorted (equivalently, quasi-isometrically embedded) subgroups (see \cite{HruskaRHCG}).

In \cite{HruskaRHCG}, Hruska proved that the intersection of two relatively quasiconvex subgroups is again relatively quasiconvex. 
However, until now the existence of a relatively quasiconvex virtual join $\langle Q',R'\rangle$, for two relatively quasiconvex subgroups \(Q\) and \(R\) in a relatively hyperbolic group \(G\), such that $S=Q \cap R$ is separable in $G$, was only known in special cases:

\begin{itemize}
    \item Mart\'{i}nez-Pedroza \cite{MPComb} proved it in the case when $R \leqslant P$, for some maximal parabolic subgroup $P$ of $G$, such that $Q \cap P \subseteq R$;
    
    \item Mart\'{i}nez-Pedroza and Sisto \cite{MPS} proved it when $Q$ and $R$ have \emph{compatible parabolics} (that is, for every maximal parabolic subgroup $P$ of $G$ either $Q \cap P \subseteq R \cap P$ or $R \cap P \subseteq Q \cap P$);
    
    \item Yang \cite{Yang} (unpublished; see also McClellan's thesis \cite{McCl}) proved it when $R$ is a \emph{full subgroup} of $G$ (that is, for every maximal parabolic subgroup $P$ in $G$, $R \cap P$ is either finite or has finite index in $P$). 
\end{itemize}

Similarly to Gitik's theorem \cite{Gitik_ping-pong}, in all three cases above the authors establish an isomorphism between the virtual join $\langle Q',R' \rangle$ and the amalgamated free product $Q'*_{S'} R'$, where $S'=Q' \cap R' \leqslant_fS$.

The extra assumptions on $Q$ and $R$ in each of the above results from \cite{MPComb,MPS,Yang,McCl} imply that $Q$ and $R$ have \emph{almost compatible parabolics} (see Definition~\ref{def:almost_compatible} below). 
Unfortunately this is still a significant restriction and a more general result is desirable.
Moreover, in the absence of almost compatibility one cannot expect a virtual join to split as an amalgamated free product of $Q'$ and $R'$. 
Indeed, for example if both $Q$ and $R$ are subgroups of an abelian peripheral subgroup of $G$ then any virtual join $\langle Q',R' \rangle$ would again be abelian.

One of the goals of the present paper is to establish quasiconvexity of virtual joins without making any compatibility assumptions on $Q$ and $R$. 
However we need to impose stronger assumptions on the properties of the profinite topology on $G$ than just separability of $S=Q \cap R$: we will require the finitely generated relatively quasiconvex subgroups to be separable and the peripheral subgroups to be {double coset separable}.

\begin{definition}[QCERF] 
\label{def:QCERF} 
    We will say that a relatively hyperbolic group $G$ is \emph{QCERF} if every finitely generated relatively quasiconvex subgroup in $G$ is separable.
\end{definition}

\begin{theorem}
\label{thm:sep->qc_intro}
     Let \(G\) be a finitely generated relatively hyperbolic group. Suppose that $G$ is QCERF and the peripheral subgroups of $G$ are double coset separable. 
     If \(Q, R \leqslant G\) are finitely generated relatively quasiconvex subgroups and $S=Q \cap R$ then there exist finite index subgroups $Q'\leqslant_f Q$ and $R' \leqslant_f R$, with \(Q' \cap R' = S\), such that the virtual join $\langle Q',R'\rangle$ is relatively quasiconvex in $G$.
     
    More precisely, there exists $L \leqslant_f G$, with $S \subseteq L$, such that for any $L' \leqslant_f L$, satisfying $S \subseteq L'$, we can choose $Q'=Q\cap L' \leqslant _f Q$, and there exists $M \leqslant_f L'$, with $Q' \subseteq M$, such that for any $M' \leqslant_f M$, satisfying $Q' \subseteq M'$, we can choose $R'=R \cap M' \leqslant_f R$.
\end{theorem}

One can observe that the choice of $R' \leqslant_f R$ in the above theorem depends on the choice of $Q' \leqslant_f Q$. In the case when the peripheral subgroups are abelian the situation is easier:

\begin{theorem}
\label{thm:sep->qc_for_ab_parab}
    Let \(G\) be a finitely generated group hyperbolic relative to a finite collection of abelian subgroups. Assume that $G$ is QCERF. 
    If \(Q, R \leqslant G\) are relatively quasiconvex subgroups and  \(S = Q \cap R\) then there exists a finite index subgroup $L \leqslant_f G$, with $S \subseteq L$, such that the virtual join $\langle Q',R'\rangle$ is relatively quasiconvex in $G$, for arbitrary subgroups $Q' \leqslant_f Q \cap L$ and $R' \leqslant_f R \cap L$, satisfying $Q' \cap R'=S$.
\end{theorem}

In fact, one can slightly weaken the assumptions in Theorem~\ref{thm:sep->qc_for_ab_parab} by requiring the peripheral subgroups of $G$ to be virtually abelian instead of abelian: see Corollary~\ref{cor:virt_ab_periph}.

Unlike the previous results from \cite{MPS,Yang}, Theorem~\ref{thm:sep->qc_intro} does not require any (almost) compatibility of parabolics from the subgroups $Q$ and $R$. To work in this general setting, we develop a novel approach which uses the profinite topology on $G$ to carefully select the finite index subgroups $Q' \leqslant_f Q$ and $R' \leqslant_f R$ satisfying certain metric properties (see Subsections~\ref{subsec:3.1}, \ref{subsec:sep_assump} and Section~\ref{sec:sep->metric}). We also give a new and simple criterion for establishing separability of double cosets in amalgamated free products in Section~\ref{sec:dcs_in_amalgams}. 

Theorem~\ref{thm:sep->qc_intro} applies to a wide class of relatively hyperbolic groups, including all limit groups, all Kleinian groups and many groups acting on CAT($0$) cube complexes.
Regarding QCERF-ness, Manning and Mart\'{i}nez-Pedroza \cite{MMPSep} proved that the following two statements are equivalent: 

\begin{itemize}
    \item[(a)] every finitely generated group hyperbolic relative to a finite collection of LERF and slender subgroups is QCERF;
    \item[(b)] all word hyperbolic groups are residually finite.
\end{itemize}
Recall that a group is called \emph{slender} if every subgroup is finitely generated. 
The question of whether statement (b) is true is a well-known open problem. 
If the answer to it is positive then, for example, all finitely generated groups hyperbolic relative to virtually polycyclic subgroups will be QCERF.

Large classes of relatively hyperbolic groups have already been proved to be QCERF. 
One of the first results in this direction is due to Wilton \cite{WiltonLimitGps}, who established QCERF-ness of limit groups.  
The ground-breaking work of Haglund and Wise \cite{Haglund-Wise} and Agol \cite{Agol} implies that any word hyperbolic group acting geometrically on a CAT($0$) cube complex is QCERF. 
One of the consequences of this result is that all finitely generated Kleinian groups are QCERF. 
More recently, Einstein and Groves \cite{Einstein-Groves} and Groves and Manning \cite{Grov-Man-spec} extended this theory to relatively hyperbolic groups acting (weakly) relatively geometrically on CAT($0$) cube complexes. 
Einstein and Ng \cite{Einstein-Ng} used it to show that full relatively quasiconvex subgroups of $C'(1/6)$-small cancellation quotients of free products of residually finite groups are separable. In the case when the free factors are LERF and slender the latter result can be combined with a theorem of Manning and Mart\'{i}nez-Pedroza \cite[Theorem~1.7]{MMPSep} to conclude that such small cancellation free products are QCERF.

By a theorem of Lennox and Wilson \cite{L-W} all virtually polycyclic groups are double coset separable, hence the assumption about peripheral subgroups in Theorem~\ref{thm:sep->qc_intro} is automatically true in many relevant cases. 
However whether this assumption is actually necessary is less obvious. 
It is required in our approach, but it would be interesting to see whether the theorem remains valid without it. 
As expected from the results in \cite{MPS, Yang}, it is not needed if the relatively quasiconvex subgroups $Q$ and $R$ have almost compatible parabolics: see Theorem~\ref{thm:almost_compat->qc_comb} below.

\subsection{Separability of double cosets}\label{subsec:1.2}
In group theory knowing that double cosets of certain subgroups are separable is often quite useful. For example, the separability of double cosets of hyperplane subgroups was used by Haglund and Wise in \cite{Haglund-Wise} to give a criterion for virtual specialness of a compact non-positively curved cube complex. 
Separability of double cosets of abelian subgroups in Kleinian groups was an important ingredient in the theorem of Hamilton, Wilton and Zalesskii \cite{HWZ} that fundamental groups of compact orientable $3$-manifolds are conjugacy separable.

Double coset separability of free groups was first proved by Gitik and Rips \cite{Gitik-Rips}. Shortly after, Niblo \cite{Niblo} came up with a new criterion for separability of double cosets and applied it to show that finitely generated Fuchsian groups and fundamental groups of Seifert-fibred $3$-manifolds are double coset separable.
Separability of double cosets of quasiconvex subgroups in QCERF word hyperbolic groups was proved by the first author in \cite{MinGFERF}. 
Mart\'{i}nez-Pedroza and Sisto \cite{MPS} generalised this to double cosets of relatively quasiconvex subgroups with compatible parabolics in QCERF relatively hyperbolic groups; Yang \cite{Yang} and McClellan \cite{McCl} treated the case when at least one of the factors is full. 
Our proof of Theorem~\ref{thm:sep->qc_intro} almost immediately yields the following.

\begin{corollary}
\label{cor:double_cosets_sep}
    Let \(G\) be a finitely generated group hyperbolic relative to a finite collection of subgroups \(\lbrace H_\nu \, | \, \nu \in \Nu \rbrace\). 
    Suppose that \(G\) is QCERF and \(H_\nu\) is double coset separable, for every \(\nu \in \Nu\).
    Then for all finitely generated relatively quasiconvex subgroups \(Q, R \leqslant G\), the double coset $QR$ is separable in $G$.
\end{corollary}

Clearly the assumptions of Corollary~\ref{cor:double_cosets_sep} are the minimal possible. This result is powerful enough to prove a conjecture of Hsu and Wise from \cite{Hsu-Wise}: see Corollary~\ref{cor:virt_spec}.

In the case when the relatively hyperbolic group $G$ admits a weakly relatively geometric action on a CAT($0$) cube complex Corollary~\ref{cor:double_cosets_sep} was proved by Groves and Manning~\cite{Grov-Man-spec}. 
Groves and Manning's argument uses Dehn fillings to approximate $G$ by QCERF word hyperbolic groups, thus reducing the statement to separability of double cosets in hyperbolic groups from \cite{MinGFERF}. 
Our approach is completely different as we always work within $G$.

In the following definition we will use a preorder $\preccurlyeq$ on the sets of subsets of a group $G$, introduced by the first author in \cite{Min-Some_props_of_subsets}:
\[
    \text{given }U, V \subseteq G \text{ we will write } U \preccurlyeq V \text{ if there exists a finite subset } Y \subseteq G \text{ such that } U \subseteq VY.
\]
If $d_X$ is the word metric on $G$, corresponding to a finite generating set $X$, and $U,V$ are subsets of $G$ then $U \preccurlyeq V$ if and only if $U$ is contained in a finite $d_X$-neighbourhood of $V$. If $U$ and $V$ are subgroups of $G$ then $U \preccurlyeq V$ is equivalent to $|U:(U \cap V)|<\infty$ (see \cite[Lemma~2.1]{Min-Some_props_of_subsets}).

\begin{definition}[Almost compatible parabolics]
\label{def:almost_compatible} 
    Let $Q$ and $R$ be  subgroups of a relatively hyperbolic groups $G$. 
    We will say that $Q$ and $R$ have \emph{almost compatible parabolics} if for every maximal parabolic subgroup $P$ of $G$ either $Q \cap P \preccurlyeq R \cap P$ or $R \cap P \preccurlyeq Q \cap P$.
\end{definition}

Clearly if $G$ is a relatively hyperbolic group and $Q,R$ are subgroups with compatible parabolics then they have almost compatible parabolics. 
The same is true if at least one of $Q$, $R$ is a full subgroup of $G$.

In the case when the relatively quasiconvex subgroups $Q$ and $R$ have almost compatible parabolics the assumption that the peripheral subgroups $H_\nu$ are double coset separable can be dropped from Corollary~\ref{cor:double_cosets_sep}, allowing us to recover the double coset separability results from \cite{MPS,Yang,McCl}.

\begin{corollary} 
\label{cor:almost_comp->sep_dc} 
    Suppose that $G$ is a finitely generated QCERF relatively hyperbolic group. 
    If $Q$ and $R$ are finitely generated relatively quasiconvex subgroups of $G$ with almost compatible parabolics then the double coset $QR$ is separable in $G$.
\end{corollary}

\subsection{Separability of products of quasiconvex subgroups}
The third part of this paper is dedicated to proving separability for more general products $F_1 \dots F_s$, where $s \in \NN$ is arbitrary and $F_1, \dots,F_s$ are relatively quasiconvex subgroups in a relatively hyperbolic group.

\begin{definition}[{RZ}$_s$ and product separability]
    Let \(P\) be a group and let \(s \in \NN\).
    We say that $P$ has property \emph{{RZ}$_s$} if for arbitrary finitely generated subgroups \(E_1, \dots, E_s \leqslant P\) the product \(E_1 \dots E_s\) is separable in \(P\).
    If \(P\) has property {RZ}\(_{s}\) for all \(s \in \NN\), we say that \(P\) is \emph{product separable}. 
\end{definition}

Thus RZ$_1$ means that the group is LERF and RZ$_2$ is equivalent to double coset separability.
The definition of {RZ}$_s$ is due to Coulbois \cite{Coulb}; he named it after Ribes and Zalesskii, who proved in \cite{RibesZal} that free groups are product separable, confirming a conjecture of Pin and Reutenauer from \cite{PinReute}. 
Pin and Reutenauer showed that product separability of free groups implies Rhodes’ type II conjecture from semigroup theory (see  \cite{PinReute,RhodesConj} for the background).

In \cite{MinGFERF}, generalising the result of \cite{RibesZal}, the first author proved that the product of finitely many quasiconvex subgroups is separable in a QCERF word hyperbolic group. 
Moreover, in \cite{Coulb} Coulbois showed that, for every $s \in \NN$, free products of groups with property {RZ}\(_s\) also have property {RZ}\(_s\). 
Taken together, these facts motivate the following theorem.

\begin{theorem}
\label{thm:RZs}
    Let \(G\) be a finitely generated group hyperbolic relative to a finite collection of subgroups \(\{ H_\nu \, | \, \nu \in \Nu \}\), and let $s \in \NN$.
    Suppose that \(G\) is {QCERF} and  \(H_\nu\) has property RZ$_s$, for each \(\nu \in \Nu\).
    If \(F_1, \dots, F_s \leqslant G\) are finitely generated relatively quasiconvex subgroups of \(G\), then the product \(F_1 \dots F_s\) is separable in \(G\).
\end{theorem}

We note that separability of products of full relatively quasiconvex subgroups in a QCERF relatively hyperbolic group was proved by McClellan \cite{McCl}.

Finitely generated virtually abelian groups are product separable. 
Therefore, Theorem~\ref{thm:RZs} applies to finitely generated  QCERF relatively hyperbolic groups with virtually abelian peripheral subgroups. 
Examples of such groups include limit groups, geometrically finite Kleinian groups and $C'(1/6)$-small cancellations quotients of free products of finitely generated virtually abelian groups (see \cite{Oregon-Reyes}). 
We discuss some applications of Theorem~\ref{thm:RZs} in Subsection~\ref{subsec:prod_sep}, and give a brief outline of the proof at the beginning of Part~\ref{part:multicosets}.

\subsection*{Acknowledgements} The authors would like to thank Pavel Zalesskii, Sam Shepherd and Benjamin Steinberg for helpful discussions, and the referee for careful reading of the paper and for valuable corrections.

\section{Applications}
\label{sec:applications}
In this section we list some applications of the main results from the Introduction.

\subsection{Geometrically finite virtual joins}\label{subsec:geom_fin_joins}
A \emph{Kleinian group} is a discrete subgroup of the (orientation-preserving) isometries of the real hyperbolic $3$-space, \(\mathrm{Isom}(\HH^3)\).
Recall that a Kleinian group \(G\) has an induced action on the ideal boundary \(\partial \HH^3\) of hyperbolic space by homeomorphisms, under which the smallest \(G\)-invariant compact subset, \(\Lambda G\), is called its \emph{limit set}. 
A subgroup \(P \leqslant G\) is called \emph{parabolic} if it has a single fixed point \(p\) in \(\partial \HH^3\) and setwise fixes some horosphere centred at \(p\).
We say that \(G\) is \emph{geometrically finite} if every point of \(\Lambda G\) is either a conical limit point or a bounded parabolic point (see \cite{Bowditch_GeomFin} for definitions). 
Examples of geometrically finite Kleinian groups include the fundamental groups of finite volume hyperbolic 3-manifolds.

As noted in the Introduction, geometrically finite  Kleinian groups are relatively hyperbolic with respect to conjugacy class representatives of their maximal parabolic subgroups (which are virtually abelian).
Moreover, geometrically finite subgroups are exactly the relatively quasiconvex subgroups of geometrically finite Kleinian groups \cite[Corollary 1.6]{HruskaRHCG}.

Baker and Cooper \cite{Baker_Cooper} showed, using geometric methods, that if \(G\) is a finitely generated Kleinian group and \(Q\) and \(R\) are geometrically finite subgroups of \(G\) with almost compatible parabolics, then there are finite index subgroups \(Q' \leqslant_f Q\) and \(R' \leqslant_f R\) such that the join \(\langle Q', R' \rangle\) is geometrically finite.
In \cite{MPS} Mart\'{i}nez-Pedroza and Sisto recover this result for geometrically finite Kleinian groups as a special case of their work, using techniques closer to those in the present paper.
Using Theorem~\ref{thm:sep->qc_intro}, we are able to eliminate the hypothesis of compatible parabolic subgroups in these results:

\begin{corollary}\label{cor:geom_fin_combination}
    Let \(G\) be a geometrically finite Kleinian group, and suppose that \(Q, R \leqslant G\) are geometrically finite subgroups of \(G\). 
    Then there are finite index subgroups \(Q' \leqslant_f Q\) and \(R' \leqslant_f R\) such that \(\langle Q', R' \rangle\) is a geometrically finite subgroup of \(G\).
\end{corollary}

\begin{proof}
    The group \(G\) is geometrically finite, so it is finitely generated \cite[Theorem~12.4.9]{Ratcliffe} and hyperbolic relative to a finite collection of finitely generated virtually abelian subgroups \cite{BowditchRHG,HruskaRHCG}.
    Agol proved that all finitely generated Kleinian groups are LERF \cite[Corollary 9.4]{Agol};
    in particular, this means that they are QCERF.
    Therefore \(G\) is a QCERF relatively hyperbolic group with double coset separable peripheral subgroups. By Hruska's result \cite[Corollary~1.6]{HruskaRHCG}, a subgroup of $G$ is geometrically finite if and only if it is relatively quasiconvex.
    We may now apply Theorem~\ref{thm:sep->qc_intro} to obtain the desired conclusion.
\end{proof}

\subsection{Product separability}
\label{subsec:prod_sep}
Recall that a group $G$ is product separable if the product of finitely many finitely generated subgroups is closed in the profinite topology on $G$. 
Until now, few examples of groups were known to be product separable: free abelian groups, free groups \cite{RibesZal}, groups of the form $F \times \mathbb{Z}$, where $F$ is free \cite{You}, and locally quasiconvex LERF hyperbolic groups \cite{MinGFERF} (eg, surface groups). 
Additionally, the class of product separable groups is closed under taking subgroups, finite index supergroups and free products \cite{Coulb}. 
However, this class is not closed under direct products (eg, the direct product of two non-abelian free groups is not even LERF \cite{Al-Gre}). 
It also does not contain some polycyclic groups: in \cite{L-W} Lennox and Wilson proved that the integral Heisenberg group $H_3(\mathbb{Z})$, which is polycyclic (in fact, finitely generated nilpotent of class $2$), is not product separable as it does not have property RZ$_3$.

We use Theorem~\ref{thm:RZs} to establish product separability for many more groups.

\begin{theorem}
\label{thm:prod_sep} 
    The following groups are product separable:
    \begin{itemize}
        \item[(i)] limit groups;
        \item[(ii)] finitely generated Kleinian groups;
        \item[(iii)] fundamental groups of finite graphs of free groups with cyclic edge groups, as long as they are balanced.
    \end{itemize}
\end{theorem}

Recall that a group \(G\) is called a \emph{limit group} if it is finitely generated and fully residually free (that is, for every finite subset \(A \subset G\), there is a free group $F$ and a homomorphism \(\varphi \colon G \to F\) that is injective when restricted to \(A\)).
Limit groups played an important role in the solutions of Tarski's problems about the first order theory of free groups by Sela \cite{Sela} and Kharlampovich-Myasnikov \cite{K-M}.

Following Wise, we say that a group $G$ is \emph{balanced} if for every infinite order element $g \in G$ the conjugacy between $g^m$ and $g^n$ implies that $n=\pm m$. 
In \cite{Wise-balanced}, Wise proved that the fundamental group $G$ of a finite graph of free groups with cyclic edge groups is LERF if and only if it is balanced if and only if $G$ does not contain any \emph{non-Euclidean} Baumslag-Solitar subgroups $BS(m,n)=\langle a,t\mid t a^m t^{-1}=a^n \rangle$, with $m, n \in \mathbb{Z}\setminus\{0\}$ and $n \neq \pm m$.

Part (iii) of Theorem~\ref{thm:prod_sep} generalises a result of Coulbois \cite[Theorem~5.18]{Coulbois-thesis}, who proved that the free amalgamated product of two free groups along a cyclic subgroup is product separable. 
Theorem~\ref{thm:prod_sep}(iii) confirms (in a strong way) a conjecture of Hsu and Wise \cite[Conjecture~15.5]{Hsu-Wise}, which states that a balanced group splitting as a finite graph of free groups with cyclic edge groups is double coset separable. 

\begin{corollary}\label{cor:virt_spec}
Suppose that $G$ splits as a fundamental group of a finite graph of finitely generated free groups with cyclic edge groups. If  $G$ is balanced then it  is virtually compact special; in other words, $G$ has a finite index subgroup which is isomorphic to the fundamental group of a compact non-positively curved special cube complex (in the sense of Haglund and Wise \cite{Haglund-Wise}).
\end{corollary}

\begin{proof} Hsu and Wise \cite[Theorem~10.4]{Hsu-Wise} proved that $G$ admits a proper cocompact action on a CAT($0$) cube complex $\mathcal X$. By Theorem~\ref{thm:prod_sep}, $G$ is double coset separable, hence, by a result of Haglund and Wise \cite[Theorem~9.19]{Haglund-Wise}, $G$ has a finite index subgroup $K$ such that  $K \backslash \mathcal{X}$ is a special cube complex.
\end{proof}

After the completion of this paper the authors learned of a recent result of Shepherd and Woodhouse \cite[Theorem~1.2]{Shep-Wood}, which gives an alternative proof of Corollary~\ref{cor:virt_spec}, using different methods. 

One of the original motivations for considering product separability of groups came from semigroups and automata theory. 
Pin and Reutenauer \cite{PinReute} used this property to characterise the profinitely closed rational subsets of free groups.

Recall that for a monoid \(M\), the \emph{rational subsets} \(\Rat(M) \subseteq 2^M\) form the smallest collection of subsets of \(M\) satisfying the following conditions:

\begin{enumerate}
    \item \(\emptyset \in \Rat(M)\) and, for each \(m \in M\), \(\{m\} \in \Rat(M)\);
    \item if \(A, B \in \Rat(M)\), then \(AB \in \Rat(M)\) and \(A \cup B \in \Rat(M)\);
    \item if \(A \in \Rat(M)\), then \(A^* \in \Rat(M)\), where \(A^*\) is the submonoid of \(M\) generated by \(A\).
\end{enumerate}
We refer the reader to \cite{PinReute} for an account of the basic theory of rational subsets. 

In a group $G$ it makes sense to consider the subgroup closure instead of the $*$-closure. Thus we define the set $\Rat^0(G) \subseteq 2^G$ as the smallest collection of subsets of $G$ containing all finite subsets, closed under finite unions, products and subgroup closure.  
It is easy to see that $\Rat^0(G)$ consists of all subsets of the form $gF_1 \dots F_s$, where $s \in \NN_0$, $g \in G$ and $F_1,\dots,F_s$ are finitely generated subgroups of $G$ (\cite[Proposition~2.2]{PinReute}). 
Evidently $\Rat^0(G) \subseteq \Rat(G)$; moreover, it is not difficult to show that $\Rat^0(G) = \Rat(G)$ if and only if $G$ is torsion.

The following theorem was proved by Pin and Reutenauer \cite[Corollary~2.5]{PinReute} in the case of free groups (see also \cite[Section 12.3]{Ribes-book} for a slightly different argument), however the proof is readily seen to remain valid in all product separable groups.

\begin{theorem}[Pin-Reutenauer] 
\label{thm:P-R} 
    If $G$ is a product separable group then $\Rat^0(G)$ is precisely the class of all separable rational subsets of $G$. 
\end{theorem}

\begin{corollary}
\label{cor:sep_rat_subsets}
    If $G$ is a group from one of the classes (i)--(iii), described in Theorem~\ref{thm:prod_sep}, then the set of  separable rational subsets of $G$ coincides with $\Rat^0(G)$.
\end{corollary}


\section{Plan of the paper}
\subsection{The metric quasiconvexity theorem}\label{subsec:3.1}
Let $G$ be a relatively hyperbolic group generated by a finite set $X$, and let $Q, R$ be relatively quasiconvex subgroups of $G$. 
The technical heart of this paper is Theorem~\ref{thm:metric_qc} below, which, given some relatively quasiconvex subgroups $Q' \leqslant Q$ and $R' \leqslant R$, provides sufficient metric conditions for the relative quasiconvexity of the join $\langle Q',R' \rangle$.

\begin{definition}[$\minx$]
    Let \(G\) be a group with finite generating set \(X\), and let \(Y \subseteq G\). Then  we denote the number \(\min \lbrace \abs{g}_X \, | \, g \in Y \rbrace\) by \(\minx(Y)\), with the usual convention that minimum over the empty set is $+\infty$.
\end{definition}

Let $S=Q \cap R$ and $A \ge 0$ be some constant. We will be interested in finding  subgroups $Q'\leqslant Q$ and $R' \leqslant R$ satisfying the following properties:

\medskip
\begin{itemize}
    \descitem{P1} if $Q'$ and $R'$ are relatively quasiconvex in $G$ then so is the subgroup \(\langle Q', R' \rangle\);
    \descitem{P2} \(\minx\Bigl(\langle Q', R' \rangle \setminus S\Bigr) \geq A\);
    \descitem{P3} \(\minx \Bigl(Q \langle Q', R' \rangle R \setminus QR \Bigr) \geq A\).
\end{itemize}
\medskip

\begin{remark} \phantom{a}
\begin{itemize}
    \item Observe that quasiconvexity of $Q'$ and $R'$ is only required in property \descref{P1}.

    \item Property \descref{P2} says that all ``short'' elements of $\langle Q',R'\rangle$ belong to $S$.

    \item Property \descref{P3} is the key ingredient for proving that the double coset $QR$ is separable in $G$ in Corollary~\ref{cor:double_cosets_sep}.
\end{itemize}
\end{remark}

Let us now describe the metric conditions used to establish the above properties.
Given a finite collection $\mathcal P$ of maximal parabolic subgroups of $G$, constants $B,C \ge 0$ and subgroups $Q' \leqslant Q$, $R' \leqslant R$, we will consider the following conditions:

\begin{itemize}
    \descitem{C1} \(Q' \cap R' = S\);
    \descitem{C2} \(\minx(Q \langle Q', R'\rangle Q \setminus Q ) \geq B\) and \(\minx(R \langle Q', R'\rangle R \setminus R ) \geq B\);
    \descitem{C3} \(\minx \Bigl( (PQ' \cup PR') \setminus PS\Bigr) \geq C\), for each $P \in \mathcal{P}$.
\end{itemize}

Moreover, if not all of the subgroups in $\mathcal{P}$ are abelian then we will need two more conditions (here for subgroups $H,P \leqslant G$, we use $H_P$ to denote the intersection  $H \cap P \leqslant P$):

\begin{itemize}
    \descitem{C4} \(Q_P \cap \langle Q_P', R_P' \rangle = Q'_P\) and \(R_P \cap \langle Q_P', R_P' \rangle = R_P'\), for every $P \in \mathcal{P}$;
    \descitem{C5}  \(\minx \Bigl(q \langle Q_P', R_P' \rangle R_P \setminus qQ'_P R_P\Bigr) \geq C \), for each $P \in \mathcal{P}$ and all $q \in Q_P$.
\end{itemize}

\begin{remark}
\label{rem:ab_periph->C4_and_C5}
    If the peripheral subgroups of $G$ are abelian then condition \descref{C4} follows from \descref{C1} and condition \descref{C5} is trivially true.
\end{remark}

Indeed, if $P$ is abelian, then, in the notation of \descref{C4}, $\langle Q_P',R_P' \rangle=Q_P'R_P'$, hence 
\[
    Q'_P \subseteq Q_P \cap \langle Q_P', R_P' \rangle = Q_P \cap Q_P'R_P'= Q_P'(Q_P \cap R'_P) \subseteq Q_P'S_P = Q_P',
\]
where the last equality used  that $S_P=S \cap P \subseteq Q_P'$ by \descref{C1}. 
The second equality of \descref{C4} can be proved in the same fashion.

Similarly, if $q \in Q_P$ then
$q \langle Q_P', R_P' \rangle R_P=qQ_P' R_P' R_P=q Q_P' R_P$, so that
\[\minx \Bigl(q \langle Q_P', R_P' \rangle R_P \setminus qQ'_P R_P\Bigr)=\minx(\emptyset)=+\infty,\]
thus \descref{C5} holds.

\begin{remark}
\label{rem:sep->metric} 
    In this paper we will be primarily interested in the existence of finite index subgroups $Q' \leqslant_f Q $ and $R' \leqslant_f R$ satisfying the above conditions. This may be easier to interpret through the lens of the profinite topology on \(G\) (see Section~\ref{sec:sep->metric}):
    \begin{itemize}
        \item conditions \descref{C1} and \descref{C4} can be ensured by choosing any finite index subgroup $M \leqslant_f G$ with $ S\subseteq M$, and setting $Q'=Q \cap M$, $R'=R \cap M$;

        \item the existence of finite index subgroups $Q' \leqslant_f Q $ and $R' \leqslant_f R$ satisfying condition \descref{C2} can be deduced from separability of $Q$ and $R$ in $G$;

        \item the existence of finite index subgroups $Q' \leqslant_f Q $ and $R' \leqslant_f R$ satisfying condition \descref{C3} can be deduced from separability of the double coset $PS$ in $G$;

        \item if $Q_P' \leqslant_f Q_P $ is already chosen then $R_P' \leqslant_f R_P$, satisfying \descref{C5}, can be constructed with the help of separability of the double coset $Q_P' R_P$ in $P$. 
        Indeed, if $Q_P=\bigcup_{j=1}^n a_j Q_P'$, then the inequality in \descref{C5} can be re-written as \(\minx \Bigl(a_j \langle Q_P', R_P' \rangle Q_p'R_P \setminus a_jQ'_P R_P\Bigr) \geq C \), for every $j=1,\dots,n$. 
        Thus our approach to establishing \descref{C5} will be to choose $R' \leqslant_f R$ after $Q'\leqslant_f Q$ has already been constructed (in other words, $R'$ will depend on $Q'$).
    \end{itemize}
\end{remark}

\begin{theorem}[Metric quasiconvexity theorem]
\label{thm:metric_qc}
    Let \(G\) be relatively hyperbolic group generated by a finite set \(X\). Suppose that \(Q, R \leqslant G\) are relatively quasiconvex subgroups and denote \(S = Q \cap R\). 
    There exists a finite collection $\mathcal{P}$ of maximal parabolic subgroups of $G$ such that for any \(A \geq 0\) there are constants \(B,C \geq 0\)   satisfying the following.

    Suppose that \(Q' \leqslant Q\), \(R' \leqslant R\) are  subgroups of \(G\) satisfying conditions \descref{C1}--\descref{C5}.
    Then these subgroups enjoy properties  \descref{P1}--\descref{P3} above.
\end{theorem}

Rough sketches of the proofs of Theorems~\ref{thm:metric_qc} and \ref{thm:sep->qc_intro} are given in the beginning of Part \ref{part:metric_qc_double_cosets} of the paper.

\subsection{The separability assumptions}\label{subsec:sep_assump} As the reader may notice, our main results in the Introduction assume that the underlying relatively hyperbolic group $G$ is QCERF and the peripheral subgroups of $G$ are double coset separable. Indeed, the essence of our method is in finding (sufficiently many) finite index subgroups $Q' \leqslant_f Q$ and $R'\leqslant_f R$ satisfying conditions \descref{C1}--\descref{C5} by using properties of the profinite topology. However, a careful analysis of the arguments reveals that instead of the full QCERF assumption it is possible to require the separability only of certain finitely generated relatively quasiconvex subgroups related to $Q$ and $R$. For example, the proof of Theorem~\ref{thm:sep->qc_intro} relies on the separability conditions~\ref{cond:1}--\ref{cond:3} from Theorem~\ref{thm:sep->qc_comb}, which are established in Section~\ref{sec:dc_when_one_is_parab} using  the separability of relatively quasiconvex subgroups $Q$, $R$, $K$, $\langle K, T\rangle$ and $\langle K, V \rangle$, where $K \leqslant_f P\in \mathcal{P}$, $ T\leqslant_f Q$, $V  \leqslant_f R$ and $\mathcal{P}=\mathcal{P}_1$ is a finite collection of maximal parabolic subgroups of $G$ that depends on $Q$ and $R$ (see Notation~\ref{not:constants_and_P_1_in_part_1}). The exact requirements for double coset separability of the peripheral subgroups are easier to trace: it suffices to look at condition \ref{cond:4} of Theorem~\ref{thm:sep->qc_comb}.

\subsection{Section outline}
This paper is structured as follows.
There are three parts: Part~\ref{part:background} contains background material and useful preliminary results (Sections~\ref{sec:prelim}-\ref{sec:RH_gps}), Part~\ref{part:metric_qc_double_cosets} is dedicated to the proof of the metric quasiconvexity theorem and the double coset separability results that follow from them (Sections~\ref{sec:path_reps}-\ref{sec:double_coset_sep}), and Part~\ref{part:multicosets} is essentially dedicated to the proof and applications of Theorem~\ref{thm:RZs} (Sections~\ref{sec:multicoset_defs}-\ref{sec:ex_prod_sep}).

Section~\ref{sec:prelim} covers generalities and Section~\ref{sec:RH_gps} covers definitions and results specific to relatively hyperbolic groups.
In Section~\ref{sec:path_reps} we introduce the terminology of \emph{path representatives}, their associated \emph{types}, and make some observations about path representatives that have minimal type.
Sections~\ref{sec:adj_backtracking} and \ref{sec:multitracking} are devoted to controlling certain instances of backtracking in minimal type path representatives.
In Section~\ref{sec:quasigeods} we describe the "shortcutting" of a broken line, and establish its quasigeodesicity under some technical assumptions.
Section~\ref{sec:metric_qc} contains the proof of Theorem~\ref{thm:metric_qc}.
In Sections~\ref{sec:sep->metric} and \ref{sec:dc_when_one_is_parab} we show how finite index subgroups \(Q' \leqslant_f Q\) and \(R' \leqslant_f R\) satisfying conditions \descref{C1}-\descref{C5} can be obtained using separability, with the help of a new criterion for separability of double cosets in amalgamated products from Section~\ref{sec:dcs_in_amalgams}.
Section~\ref{sec:sep->qc} contains proofs of Theorems~\ref{thm:sep->qc_intro} and \ref{thm:sep->qc_for_ab_parab}, while Section~\ref{sec:double_coset_sep} contains the proof of Corollary~\ref{cor:almost_comp->sep_dc}.

In Section~\ref{sec:multicoset_defs} we generalise the content of Section~\ref{sec:path_reps} to the setting of products of subgroups, as well as introducing new metric conditions \descref{C2-m} and \descref{C5-m}.
Sections~\ref{sec:mcs_multitracking1} and \ref{sec:mcs_multitracking2} are product analogues to Section~\ref{sec:multitracking}; similarly, Section~\ref{sec:mcs_sep->metric} generalises Section~\ref{sec:sep->metric}.
Finally, Section~\ref{sec:RZs_proof} contains the proof of Theorem~\ref{thm:RZs}, and Section~\ref{sec:ex_prod_sep} establishes new examples of product separable groups, proving Theorem~\ref{thm:prod_sep}.


\part{Background}
\label{part:background}
In this part we will present the definitions and basic results that will be necessary for the rest of the paper.

\section{Preliminaries}
\label{sec:prelim}

\subsection{Notation}
We write \(\NN\) for the set of natural numbers \(\{1, 2, 3 \dots\}\), and \(\NN_0\) for \(\NN \cup \{0\}\).

Let \(G\) be a group.
If \(H\) is a finite index (respectively, finite index normal) subgroup of \(G\), then we write \(H \leqslant_f G\) (respectively, \(H \lhd_f G\)). 
For a subgroup $T \leqslant G$ and elements $a,b \in G$ we will write $T^a=aTa^{-1} \leqslant G$ and $b^a=aba^{-1} \in G$.

By a generating set $\mathcal{A}$ of $G$ we will mean a set $\mathcal{A}$ together with a map $\mathcal{A} \to G$ such that the image of $\mathcal{A}$ under this map generates $G$.

If \(\mathcal{A}\) is a generating set for \(G\), then we denote by \(\Gamma(G,\mathcal{A})\) the (left) Cayley graph of \(G\) with respect to \(\mathcal{A}\). 
The standard edge path length metric on \(\Gamma(G,\mathcal{A})\) will be denoted $d_{\mathcal{A}}(\cdot,\cdot)$. 
After identifying \(G\) with the vertex set of \(\Gamma(G,\mathcal{A})\), this metric induces the \emph{word metric} associated to $\mathcal{A}$: \(d_{\mathcal{A}}(g,h) = \abs{g^{-1}h}_{\mathcal{A}}\) for all $g,h \in G$, where \(\abs{g}_{\mathcal{A}}\) denotes the length of the shortest word in \(\mathcal{A}^{\pm 1}\) representing \(g\) in $G$.

Abusing the notation, we will identify the combinatorial Cayley graph $\Gamma(G,\mathcal{A})$ with its geometric realisation. 
The latter is a geodesic metric space and, given two points $x,y$ in this space, we will use $[x,y]$ to denote a geodesic path from $x$ to $y$ in $\Gamma(G,\mathcal{A})$. 
In general $\Gamma(G,\mathcal{A})$ need not be uniquely geodesic, so there will usually be a choice for $[x,y]$, which will either be specified or will be clear from the context (eg, if $x$ and $y$ already belong to some geodesic path under discussion, then $[x,y]$ will be chosen as the subpath of that path).

If \(Y \subseteq G\) is a subset of \(G\) and \(K \geq 0\), we denote by
\[
    N_{\mathcal{A}}(Y,K) = \{ g \in G \, | \, d_{\mathcal{A}}(g,Y) \leq K \}
\]
the \(K\)-neighbourhood of \(Y\) with respect to \(d_{\mathcal{A}}\).
Note that when \(\mathcal{A}\) is a finite generating set, the metric \(d_{\mathcal{A}}\) is proper. 
However, in this paper we will also be working with infinite generating sets:  see Section~\ref{sec:RH_gps} below, where generating sets of the form $\mathcal{A}=X \cup \mathcal{H}$ are considered.

The following general fact will be used quite often.
\begin{lemma}
\label{lem:nbhdintersection}
    Let \(G\) be a group generated by a finite set $\mathcal{A}$. If \(A, B \leqslant G\) are subgroups of \(G\) then for every \(K \geq 0\) there is a constant \(K' = K'(A,B,K) \geq 0\) such that for any $x \in G$ we have
    \[        N_{\mathcal{A}}(xA,K) \cap N_{\mathcal{A}}(xB,K) \subseteq N_{\mathcal{A}}(x(A \cap B),K').    \]
\end{lemma}
\begin{proof} After applying the left translation by $x^{-1}$, which preserves the metric $d_{\mathcal{A}}$, we can assume that $x=1$. Now the statement follows, for example, from \cite[Proposition~ 9.4]{HruskaRHCG}.
\end{proof}

Suppose that $\gamma$ is a combinatorial path (edge path) in $\Gamma(G,\mathcal{A})$.
We will denote the initial and terminal endpoints of \(\gamma\) by \(\gamma_-\) and \(\gamma_+\) respectively.
We will write \(\ell(\gamma)\) for the length (that is, the number of edges) of \(\gamma\). 
We will also use $\gamma^{-1}$ to denote the inverse of $\gamma$, which is the path starting at $\gamma_+$, ending at $\gamma_-$ and traversing $\gamma$ in the reverse direction.
If \(\gamma_1, \dots, \gamma_n\) are combinatorial paths with \((\gamma_i)_+ = (\gamma_{i+1})_-\), for each \(i \in \{1, \dots, n-1\}\), we will denote their concatenation by \(\gamma_1 \dots \gamma_n\).

Since $\Gamma(G,\mathcal{A})$ is a labelled graph, every combinatorial path $\gamma$ comes with a label $\Lab(\gamma)$, which is a word over the alphabet $\mathcal{A}^{\pm 1}$.
We denote by \(\elem{\gamma} \in G\) the element represented by \(\Lab(\gamma)\) in $G$.
Finally, we write \(\abs{\gamma}_{\mathcal{A}} = |\elem{\gamma}|_{\mathcal{A}}=d_{\mathcal{A}}(\gamma_-,\gamma_+)\). 
Note that $\Lab(\gamma^{-1})$ is the formal inverse of $\Lab(\gamma)$, so that and $|\gamma^{-1}|_{\mathcal{A}}=|\gamma|_{\mathcal{A}}$ and $\widetilde{\gamma^{-1}}={\elem{\gamma}}^{-1}$.


\subsection{Quasigeodesic paths}
In this section we assume that $\Gamma$ is a graph equipped with the standard path length metric $d(\cdot,\cdot)$.

\begin{definition}[Quasigeodesic]
\label{def:quasigeodesic}
    Let $\lambda \ge 1$ and $c \ge 0$ be some numbers and let $p$ be an edge path in  $\Gamma$. 
    Recall that $p$ is said to be \emph{$(\lambda,c)$-quasigeodesic} if for every combinatorial subpath $q$ of $p$ we have
    \[
        \ell(q) \le \lambda d (q_-,q_+) +c.
    \]
\end{definition}

\begin{lemma}
\label{lem:qgeod_with_attachments_is_qeod} 
    Suppose that $s=rpt$ is a concatenation of three combinatorial paths $r$, $p$ and $t$ in $\Gamma$ such that $\ell(r)\le D$ and $\ell(t) \le D$, for some $D \ge 0$, and $p$ is $(\lambda,c)$-quasigeodesic, for some $\lambda \ge 1$ and $c \ge 0$. 
    Then the path $s$ is $(\lambda,c')$-quasigeodesic, where $c'=c+2(\lambda+1)D$.
\end{lemma}

\begin{proof} 
    Consider an arbitrary combinatorial subpath $q$ of $s$. We need to show that
    \begin{equation}
    \label{eq:need_to_show_for_q}
        \ell(q) \le \lambda d (q_-,q_+) +c+2(\lambda+1)D.
    \end{equation}

    If $q$ is contained in $r$ or in $t$ then the desired inequality follows from the assumptions that $\ell(r) \le D$ and $\ell(t) \le D$.
    Therefore we can further suppose that $q_-$ is a vertex of $rp$ and $q_+$ is a vertex of $pt$. The bounds on the lengths of $r$ and $t$ imply that there is a combinatorial subpath $a$ of $p$ such that there are at most $D$ edges of $s$ between $q_-$ and $a_-$ and between $a_+$ and $q_+$. Thus
    $d(q_-,a_-) \le D$, $d(q_+,a_+) \le D$ and $\ell(q) \le \ell(a)+2D$

    The assumption that $p$ is $(\lambda,c)$-quasigeodesic implies that
    \begin{equation}\label{eq:ell_of_q}
        \ell(q) \le \ell(a)+2D \le \lambda d(a_-,a_+)+c+2D.
    \end{equation}
    The triangle inequality gives
    $ d(a_-,a_+)\le d(q_-,q_+)+2D$,
    which, combined with \eqref{eq:ell_of_q}, shows that \eqref{eq:need_to_show_for_q} holds, as required.
\end{proof}

\begin{lemma}
\label{lem:perturbed_quasigeodesic}
    Let \(\lambda\geq 1, c \geq 0\) and $K \in \NN$.
    Suppose that \(p\) is a combinatorial path in $\Gamma$ and let \(p'\) be a path obtained by replacing some edges of \(p\) with combinatorial paths of length at most \(K\).
    If \(p\) is \((\lambda,c)\)-quasigeodesic then \(p'\) is  \((K\lambda,2K^2\lambda+Kc+2K)\)-quasigeodesic.
\end{lemma}

\begin{proof}
    Let $q$ be any combinatorial subpath of $p'$ and write \(q_-=x\) and \(q_+=y\). We need to show that
    \begin{equation}\label{eq:ineq_for_length_of_q}
    \ell(q)\le  K \lambda d(x,y) + 2K^2\lambda + Kc + 2K.
    \end{equation}
    
    If $q$ does not contain any vertices of $p$ then $\ell(q) \le K$ and \eqref{eq:ineq_for_length_of_q} holds. Otherwise, let $z$ and $w$ be the first and the last vertices of $q$ that lie on $p$ respectively, and let $r$ be the subpath of $p$ starting at $z$ and ending at $w$. The assumptions imply that \(d(x,z) \leq K\), \(d(y,w) \leq K\) and
    \begin{equation}
    \label{eq:bd_on_len_q}
        \ell(q) \leq K \ell(r) + 2K.
    \end{equation}
    Using the quasigeodesicity of \(p\) and the triangle inequality, we obtain
    \begin{equation*}
        \ell(r) \leq \lambda d(z,w) + c \leq \lambda d(x,y) + 2K\lambda + c,
    \end{equation*}
    which, combined with (\ref{eq:bd_on_len_q}), gives (\ref{eq:ineq_for_length_of_q}).
\end{proof}

\subsection{Hyperbolic metric spaces}

In this subsection \((\Gamma,d)\) denotes a geodesic metric space.

\begin{definition}[Gromov product]
    Let \(x, y, z \in \Gamma\) be points.
    The \emph{Gromov product} of \(x\) and \(y\) with respect to \(z\) is
    \[
        \langle x, y \rangle_z = \frac{1}{2}\Big( d(x,z) + d(y,z) - d(x,y) \Big).
    \]
\end{definition}

It is easy to see that the Gromov products satisfy the following equations:
\[
    d(x,y)=\langle y,z \rangle_x+\langle x, z \rangle_y,~
    d(y,z)=\langle x, z \rangle_y+\langle x, y \rangle_z \text{ and } d(z,x)=\langle x, y \rangle_z+\langle y, z \rangle_x.
\]

The following elementary property of Gromov products is an immediate consequence of the triangle inequality.

\begin{remark}
\label{rem:Gr_prod_ineq}
    Suppose that $x,y,z$ are points in $\Gamma$, $u$ is a point on any geodesic segment $[x,z]$, from $x$ to $z$, and $v$ is a point on any geodesic segment $[z,y]$, from $z$ to $y$. Then \[\langle u,v \rangle_z \le \langle x,y \rangle_z.\]
\end{remark}

\begin{definition}[$\delta$-thin triangle]
    Let \(\Delta\) be a geodesic triangle in \(\Gamma\) with vertices \(x, y,\) and \(z\), and let \(\delta \geq 0\).
    Denote by \(T_\Delta\) the (possibly degenerate) tripod with edges of length \(\langle x, y \rangle_z, \langle y, z \rangle_x\), and \(\langle z, x \rangle_y\) respectively.
    There is an map from \(\{x,y,z\}\) to the extremal vertices of \(T_\Delta\), which extends uniquely to a map \(\phi \colon \Delta \to T_\Delta\), whose restriction to each side of \(\Delta\) is an isometry.
    If the diameter in \(\Gamma\) of \(\phi^{-1}(\{t\})\) is at most \(\delta\), for all \(t \in T_\Delta\), then \(\Delta\) is said to be \emph{\(\delta\)-thin}.
\end{definition}

\begin{definition}[Hyperbolic space]
    The space \(\Gamma\) is said to be a \emph{hyperbolic metric space} if there is a constant \(\delta \geq 0\) such that every geodesic triangle in \(\Gamma\) is \(\delta\)-thin.
\end{definition}

The above definition of \(\delta\)-hyperbolicity is not the most commonly used in the literature, though it is well-known to be equivalent to other definitions after possibly increasing \(\delta\): see, for example, \cite[III.H.1.17]{Bridson_Haefliger}.
For technical reasons we will always assume that \(\delta\) is chosen to be sufficiently large so that all the definitions in this reference are satisfied.

In the remainder of this subsection we assume that \(\Gamma\) is a \(\delta\)-hyperbolic graph, for some $\delta \ge 0$, and $d(\cdot,\cdot)$ is the standard path length metric on $\Gamma$.

\begin{definition}[Broken line] 
\label{def:broken_line}
    A \emph{broken line} in $\Gamma$ is a path $p$ which comes with a fixed decomposition as a concatenation of combinatorial geodesic paths $p_1,\dots,p_n$ in $\Gamma$, so that $p=p_1p_2 \dots p_n$. 
    The paths $p_1, \dots, p_n$ will be called the \emph{segments} of the broken line $p$, and the vertices $p_-=(p_1)_-$, $(p_1)_+=(p_2)_-$, $\dots,$ $(p_{n-1})_+=(p_n)_-$ and $(p_{n+1})_+=p_+$ will be called the \emph{nodes} of $p$.
\end{definition}

The following statement is a special case of Lemma~4.2 from \cite{AshotResHom}, applied to the situation when each $p_i$ is geodesic (so, in the notation of that lemma, we can take $\overline{\lambda}=1$, $\overline{c}=0$ and $\nu=\delta$).
Note that due to a slightly different definition of quasigeodesicity used in \cite{AshotResHom}, a $(\lambda,c)$-quasigeodesic in the sense of \cite{AshotResHom} is $(1/\lambda,c/\lambda)$-quasigeodesic in the sense of Definition~\ref{def:quasigeodesic} above, and vice-versa.

\begin{lemma}
\label{lem:concat_of_geodesics-original}
    Let $c_0, c_1$ and $c_2$ be constants such that \(c_0 \geq 14\delta\), $c_1 = 12( c_0+\delta)+1$ and $c_2=10 (\delta+c_1)$.

   Suppose that $p=p_1 \dots p_n$ is a broken line in $\Gamma$, where $p_i$ is a geodesic with $(p_i)_-=x_{i-1} $, $(p_i)_+=x_i $, $i=1,\dots,n$.
   If \(d(x_{i-1}, x_{i}) \geq c_1 \) for \(i = 1, \dots, n\), and \(\langle x_{i-1}, x_{i+1} \rangle_{x_i} \leq c_0\) for each \(i = 1, \dots, n-1\), then the path \(p\) is \((4, c_2)\)-quasigeodesic.
\end{lemma}

We will need an extension of the above lemma which allows the first and the last geodesic segments $p_1$ and $p_n$ to be short.

\begin{lemma}
\label{lem:concat}
    For any constant $c_0$, satisfying $c_0 \ge 14 \delta$, let $c_1=c_1(c_0) = 12( c_0+\delta)+1$ and $c_3=c_3(c_0)=10 (\delta+2c_1)$.

    Suppose that $p=p_1 \dots p_n$ is a broken line in $\Gamma$, where $p_i$ is a geodesic with $(p_i)_-=x_{i-1}$, $(p_i)_+=x_i$, $i=1,\dots,n$.
    If \(d(x_{i-1}, x_{i}) \geq c_1 \) for \(i = 2, \dots, n-1\), and \(\langle x_{i-1}, x_{i+1} \rangle_{x_i} \leq c_0\) for each \(i = 1, \dots, n-1\), then the path \(p\) is \((4,c_3)\)-quasigeodesic.
\end{lemma}

\begin{proof}
    This follows easily by combining Lemma~\ref{lem:concat_of_geodesics-original} with Lemma~\ref{lem:qgeod_with_attachments_is_qeod}. 
    Indeed, there are four possibilities depending on whether or not $d(x_0,x_1) \ge c_1$ and $d(x_{n-1},x_n) \ge c_1$. 
    Since all of these cases are similar, let us concentrate on the situation when  $d(x_0,x_1) <c_1$ and $d(x_{n-1},x_n) \ge c_1$. 
    Then the path $q=p_2p_3 \dots p_n$ is $(4,c_2)$-quasigeodesic by  Lemma~\ref{lem:concat_of_geodesics-original}, where $c_2=10 (\delta+c_1)$. 
    Since $\ell(p_1)=d(x_0,x_1) < c_1$, we can apply Lemma~\ref{lem:qgeod_with_attachments_is_qeod} to deduce that the path $p=p_1\dots p_n=p_1 q$ is $(4,c_3)$-quasigeodesic, where $c_3=c_2+10c_1=10 (\delta+2c_1)$ as required.
\end{proof}

\subsection{Profinite topology and separable subsets}

Let \(G\) be a group.
The \emph{profinite topology} on \(G\) is the topology \(\pt(G)\) whose basis consists of left cosets to finite index subgroups of \(G\).

A subset \(Z \subseteq G\) is called \emph{separable} (in \(G\)) if it is closed in \(\pt(G)\).
Evidently finite unions and arbitrary intersections of separable subsets are separable.
It is easy to see that a subset \(Z \subseteq G\) is separable if and only if for every \(g \in G\setminus Z\), there is a finite group \(Q\) and a homomorphism \(\varphi \colon G \to Q\) such that \(\varphi(g) \notin \varphi(Z)\) in $Q$. A subgroup \(H \leq G\) is separable if and only if it is the intersection of  the finite index subgroups of $G$ containing it.

The following observation stems from the fact that the group operations of taking an inverse and multiplying by a fixed element are homeomorphisms with respect to the profinite topology.

\begin{remark}
\label{rem:sep_props}
    Let $Z$ be a separable subset of a group $G$. 
    Then for every $g \in G$ the subsets $Z^{-1}$, $gZ$ and $Zg$ are also separable.
\end{remark}

\begin{lemma}
\label{lem:induced_top} 
    Suppose that $A$ is a subgroup of a group $G$.
    \begin{itemize}
        \item[(a)] Every subset of $A$ which is closed in $\pt(G)$ is also closed in $\pt(A)$.
        \item[(b)] If every finite index subgroup of $A$ is separable in $G$ then every closed subset of $\pt(A)$ is closed in $\pt(G)$.
    \end{itemize}
\end{lemma}

\begin{proof} 
    Claim (a) immediately follows from the observation that the intersection of $A$ with any basic closed subset from $\pt(G)$ is either empty or is a basic closed subset of $\pt(A)$.

    If each finite index subgroup of $A$ is separable in $G$ then, in view of Remark~\ref{rem:sep_props}, every basic closed set in $\pt(A)$ is closed in the profinite topology of $G$. 
    Claim (b) of the lemma now follows from the fact that any closed subset of $A$ is the intersection of basic closed sets. 
\end{proof}

\begin{lemma}
\label{lem:fi_dc}
    Let $G$ be a group with subgroups $A,B$. 
    Suppose that $A' \leqslant_f A$, $B' \leqslant_f B$ and $A'B'$ is separable in $G$. 
    Then $AB$ is separable in $G$.
\end{lemma}

\begin{proof}
    Let $A=\bigsqcup_{i=1}^m a_iA'$ and   $B=\bigsqcup_{j=1}^n B'b_j$. 
    Then \[AB=\bigcup_{i=1}^m \bigcup_{j=1}^n a_iA'B'b_j,\] which is separable in $G$ by Remark~\ref{rem:sep_props}.
\end{proof}

The next two lemmas use the notation introduced in Subsections~\ref{subsec:1.2} and \ref{subsec:3.1}.

\begin{lemma}
\label{lem:V_sep_and_U_leq_V->UV-sep}
    Let $A,B$ be subgroups of a group $G$ such that $A \preccurlyeq B$. If $B$ is separable in $G$ then so are the double cosets $AB$ and $BA$.
\end{lemma}

\begin{proof} 
    By \cite[Lemma 2.1]{Min-Some_props_of_subsets} $A \cap B$ has finite index in $A$, so $A=\bigsqcup_{i=1}^m a_i(A \cap B)$, for some $a_1,\dots,a_m \in A$. 
    It follows that $AB=\bigcup_{i=1}^m a_iB$, so it is separable by Remark~\ref{rem:sep_props}.
    The same remark also implies that $BA=(AB)^{-1}$ is separable in $G$.
\end{proof}

The main use of the profinite topology in this paper stems from the following elementary facts.
\begin{lemma}
\label{lem:sep->large_minx}
    Let $G$ be a group generated by a finite set $X$, and let $P \leqslant G$ be a subgroup. Suppose that $Z$ is  a separable subset of $P$.
    \begin{enumerate}[label=(\alph*)]
        \item If a finite subset $U \subseteq P$ is  disjoint from $Z$ then there is a normal finite index subgroup $N \lhd_f P$ such that $U \cap ZN=\emptyset$. Thus the image of $U$ in the quotient $P/N$ will be disjoint from the image of $Z$.
        \item For every constant $C \ge 0$ there is a finite index normal subgroup $N \lhd_f P$ such that \[\minx(ZN \setminus Z) \ge C.\]
        \item For any finite subset $A \subseteq P$ and any $C \ge 0$ there exists $N \lhd_f P$ such that \[\minx(aZN \setminus aZ) \ge C,~\text{ for all } a \in A.\]
    \end{enumerate}
\end{lemma}

\begin{proof}   
    (a) Let $U=\{u_1,\dots,u_m\} \subseteq P$. Since $u_i \notin Z$ and $Z$ is separable in $P$, there exists $N_i \lhd_f P$ such that $u_iN_i \cap Z=\emptyset$, for each $i=1,\dots,m$.
    We set $N=\bigcap_{i=1}^m N_i \lhd_f P$, so that $u_iN \cap Z=\emptyset$.
    That is, $u_i \notin ZN$ for all \(i = 1, \dots m\). Therefore $U \cap ZN=\emptyset$ and (a) has been proved.

    Claim (b) follows by applying claim (a) to the finite subset $U=\{g \in P\setminus Z \mid |g|_X <C \}$ of $P$.

    To prove (c), suppose that $A=\{a_1,\dots,a_k\} \subseteq P$. By Remark~\ref{rem:sep_props}, $a_jZ$ is separable in $P$, for every $j=1,\dots,k$, so, according to part (b),  there exists $N_j \lhd_f P$ such that 
    \[
        \minx(a_jZN_j \setminus a_jZ) \ge C, \text{ for each }j=1,\dots,k .
    \] 
    It is easy to see that the normal subgroup $N= \bigcap_{j=1}^k N_j \lhd_f P$ enjoys the required property.
\end{proof}

The following statement is well-known; we include a proof for completeness.

\begin{lemma}
\label{lem:Lemma_0}
    Let \(G\) be a group with subgroups \(K \leqslant_f H \leqslant G\). If \(K\) is separable in \(G\), then there is  \(L \leqslant_f G\) such that \(L \cap H = K\)
\end{lemma}

\begin{proof}
    Since \(K\) is of finite index in \(H\), we can write \(H = K \cup Kh_1 \cup \dots \cup Kh_m\) for some \(h_1, \dots h_m \in H \setminus K\).
    The subgroup \(K\) is separable in \(G\), meaning that it is closed in \(\pt(G)\).
    Following Remark~\ref{rem:sep_props}, the union \(Kh_1 \cup \dots \cup Kh_m\) is also closed in \(\pt(G)\).
    Thus the subset \((G \setminus H) \cup K = G \setminus (Kh_1 \cup \dots \cup Kh_m)\) is open in \(\pt(G)\) and contains the identity.
    It follows from the definition of the profinite topology that there is a finite index normal subgroup \(N \lhd_f G\) with \(N \subseteq (G \setminus H) \cup K\).
    Observe that \(Kh_i \cap N = \emptyset\), for every \(i= 1, \dots, m\), so \(N \cap H \leqslant K\).
    Now set \(L = KN\leqslant_f G\). Then \(L \cap H = KN \cap H = K(N \cap H) = K\), as required.
\end{proof}


\section{Relatively hyperbolic groups}
\label{sec:RH_gps}
In this section we define relatively hyperbolic groups and collect various properties that will be used throughout the paper.

\subsection{Definition}
We will define relatively hyperbolic groups following the approach of Osin (for full details, see \cite{OsinRHG}).

\begin{definition}[Relative generating set, relative presentation] 
\label{def:rel_gen_set}
    Let \(G\) be a group, \(X \subseteq G\) a subset and \(\lbrace H_\nu \, | \, \nu \in \Nu \rbrace\) a collection of subgroups of \(G\). 
    The group \(G\) is said to be \emph{generated by \(X\) relative to \(\lbrace H_\nu \, | \, \nu \in \Nu \rbrace\)} if it is generated by \(X \sqcup \mathcal{H}\), where \(\mathcal{H}= \bigsqcup_{\nu \in \Nu} (H_\nu \setminus\{1\})\) (with the obvious map $X \sqcup \mathcal{H} \to G$). 
        If this is the case, then there is a surjection
    \[
        F = F(X) \ast (\ast_{\nu \in \Nu} H_\nu) \to G,
    \]
    where $F(X)$ denotes the free group on $X$.
    Suppose that the kernel of this map is the normal closure of a subset $\mathcal{R} \subseteq F$. Then $G$ can equipped with the \emph{relative presentation}
\begin{equation} \label{eq:rel_pres}
\langle X, H_\nu, \nu \in \mathcal{N} \mid \mathcal{R} \rangle.    
\end{equation} 

If \(X\) is a finite set, then \(G\) is said to be \emph{finitely generated relative to \(\lbrace H_\nu \, | \, \nu \in \Nu \rbrace\)}. If \(\mathcal{R}\) is also finite, \(G\) is said to be \emph{finitely presented relative to \(\lbrace H_\nu \, | \, \nu \in \Nu \rbrace\)} and the presentation above is a \emph{finite relative presentation}.
\end{definition}

With the above notation, we call the Cayley graph \(\ga\) the \emph{relative Cayley graph} of \(G\) with respect to  \(X\) and \(\lbrace H_\nu \, | \, \nu \in \Nu \rbrace\).
Note that when \(X\) is itself a generating set of \(G\), \(d_{X\cup\mathcal{H}}(g,h) \leq d_X(g,h)\), for all \(g,h \in G\).

\begin{definition}[Relative Dehn function]
    Suppose that \(G\) has a finite relative presentation \eqref{eq:rel_pres} with respect to a collection of subgroups \(\lbrace H_\nu \, | \, \nu \in \Nu \rbrace\).
    If \(w\) is a word in the free group \(F(X\sqcup\mathcal{H})\), representing the identity in \(G\), then it is equal in \(F\) to a product of conjugates
    \[
        w \stackrel{F}{=} \prod_{i=1}^n a_i r_i a_i^{-1},
    \]
    where \(a_i \in F\) and \(r_i \in \mathcal{R}\), for each \(i\).
    The \emph{relative area} of the word \(w\) with respect to the relative presentation, \(Area^{rel}(w)\), is the least number \(n\) among products of conjugates as above that are equal to \(w\) in \(F\).

    A \emph{relative isoperimetric function} of the above presentation is a function \(f \colon \NN \to \NN\) such that \(Area^{rel}(w)\) is at most \( f(\abs{w})\), for every freely reduced word \(w\) in \(F(X\sqcup\mathcal{H})\) representing the identity in \(G\).
    If an isoperimetric function exists for the presentation, the smallest such function is called the \emph{relative Dehn function} of the presentation.
\end{definition}

\begin{definition} [Relatively hyperbolic  group]
\label{def:rh_gp}
    Let \(G\) be a group and let \(\lbrace H_\nu \, | \, \nu \in \Nu \rbrace\) be a collection of subgroups of \(G\). 
    If \(G\) admits a finite relative presentation with respect to this collection of subgroups which has a well-defined linear relative Dehn function, it is called \emph{hyperbolic relative to} \(\lbrace H_\nu \, | \, \nu \in \Nu \rbrace\).
    When it is clear what the relevant collection of subgroups is, we refer to \(G\) simply as a \emph{relatively hyperbolic group}.
    The groups \(\lbrace H_\nu \, | \, \nu \in \Nu \rbrace\) are called the \emph{peripheral subgroups} of the relatively hyperbolic group $G$, and their conjugates in $G$ are called \emph{maximal parabolic subgroups}. 
    Any subgroup of a maximal parabolic subgroup is said to be \emph{parabolic}.
\end{definition}

\begin{lemma}[{\cite[Corollary 2.54]{OsinRHG}}]
\label{lem:Cayley_graph-hyperbolic}
    Suppose that $G$ is a group generated by a finite set $X$ and hyperbolic relative to a collection of subgroups \(\lbrace H_\nu \mid \nu \in \Nu \rbrace\), and let \(\mathcal{H} = \bigsqcup_{\nu \in \Nu} (H_\nu \setminus \{1\})\). 
    Then the Cayley graph $\ga$ is $\delta$-hyperbolic, for some $\delta \ge 0$.
\end{lemma}

In the remainder of this section (namely, in Subsections \ref{subsec:quasigeod_in_rh_gps}--\ref{subsec:qc_subset_in_rh_gps}, we will assume that \(G\) is a group generated by a finite subset $X$ and hyperbolic relative to a finite collection of subgroups \(\{ H_\nu \, | \, \nu \in \Nu \}\). As usual, we will let \(\mathcal{H} = \bigsqcup_{\nu \in \Nu} (H_\nu \setminus \{1\})\).

\subsection{Geodesics and quasigeodesics in relatively hyperbolic groups}\label{subsec:quasigeod_in_rh_gps}
\begin{definition}[Path components]
    Let \(p\) be a combinatorial path in \(\Gamma(G,X\cup\mathcal{H})\).
    A non-trivial combinatorial subpath of \(p\) whose label consists entirely of elements of \(H_\nu \setminus \{1\}\), for some \(\nu \in \Nu\), is called an \emph{\(H_\nu\)-subpath} of \(p\).

    An \(H_\nu\)-subpath is called an \emph{\(H_\nu\)-component} if it is not contained in any strictly longer \(H_\nu\)-subpath.
    We will call a subpath of \(p\) an \(\mathcal{H}\)-subpath (respectively, an \emph{\(\mathcal{H}\)-component}) if it is an \(H_\nu\)-subpath (respectively, an \(H_\nu\)-component), for some \(\nu \in \Nu\).
\end{definition}

\begin{definition}[Connected and isolated components]
    Let \(p\) and $q$ be edge paths in \(\Gamma(G,X\cup\mathcal{H})\) and suppose that \(s\) and \(t\) are \(H_\nu\)-subpaths of \(p\) and $q$ respectively, for some $\nu \in \Nu$.
    We say that \(s\) and \(t\) are \emph{connected} if $s_-$ and $t_-$ belong to the same left coset of $H_\nu$ in $G$. The latter means that for all vertices $u$ of $s$ and $v$ of $t$ either $u=v$ or there is an edge \(e\) in $\ga$ with $\Lab(e) \in H_\nu \setminus\{1\}$ and \(e_- = u, e_+ = v\).

    If \(s\) is an $H_\nu$-component of a path $p$ and $s$ is not connected to any other \(H_\nu\)-component of \(p\) then we say that \(s\) is \emph{isolated} in \(p\).
\end{definition}

\begin{definition}[Phase vertex]
    A vertex \(v\) of a combinatorial path \(p\) in \(\Gamma(G,X\cup\mathcal{H})\) is called \emph{non-phase} if it is an interior vertex of an \(\mathcal{H}\)-component of \(p\) (that is, if it lies in an \(\mathcal{H}\)-component which it is not an endpoint of).
    Otherwise \(v\) is called \emph{phase}.
\end{definition}

\begin{definition}[Backtracking]
    If all \(\mathcal{H}\)-components of a combinatorial path \(p\) are isolated, then \(p\) is said to be \emph{without backtracking}.
    Otherwise we say that \(p\) \emph{has backtracking}.
\end{definition}

\begin{remark}
\label{rem:comp_of_geod_is_an_edge}
    If \(p\) is a geodesic edge path in \(\Gamma(G,X\cup\mathcal{H})\) then every $\mathcal{H}$-component of $p$ will consist of a single edge, labelled by an element from $\mathcal{H}$. Therefore every vertex of $p$ will be phase.
    Moreover, it is easy to see that \(p\) will be without backtracking.
\end{remark}

The following is a basic observation about the lengths of paths in the relative Cayley graph whose \(\mathcal{H}\)-components are uniformly short.

\begin{lemma}
\label{lem:rel_geods_with_short_comps}
    Let \(p\) be a path in \(\Gamma(G,X\cup\mathcal{H})\) and suppose there is a constant \(\Theta \geq 1\) such that for any \(\mathcal{H}\)-component \(h\) of \(p\), we have \(|h|_X \leq \Theta\). Then \(|p|_X \leq \Theta \ell(p)\).
\end{lemma}

\begin{proof}
    We can write \(p\) as a concatenation \(p = a_0 h_1 a_1 \dots a_{n-1} h_n a_n\), where \(h_1, \dots, h_n\) are the \(\mathcal{H}\)-com\-po\-nents of \(p\) and \(a_0, \dots, a_n\) are subpaths of \(p\) all whose edges are labelled by elements of \(X^{\pm1}\).

    It follows from the triangle inequality that
    \begin{equation*}
       |p|_X= d_X(p_-, p_+) \leq \sum_{i = 0}^n d_X((a_i)_-, (a_i)_+) + \sum_{i = 1}^n d_X((h_i)_-, (h_i)_+).
    \end{equation*}
    Since each edge of \(a_i\) is labelled by an element of \(X^{\pm1}\), we have that \(d_X((a_i)_-, (a_i)_+) \leq \ell(a_i)\), for all \(i = 0, \dots, n\).
    Moreover, \(d_X((h_i)_-, (h_i)_+)=|h_i|_X \leq \Theta \ell(h_i)\), for each \(i = 1, \dots, n\), by the hypothesis of the lemma, as \(\ell(h_i) \ge 1\).

    Combining the above three inequalities with the fact that \(\Theta \geq 1\), we obtain
    \[         
        |p|_X \leq \sum_{i=0}^n \ell(a_i) + \sum_{i=1}^n \Theta \ell(h_i) \leq \Theta \Big(\sum_{i=0}^n \ell(a_i) + \sum_{i=1}^n \ell(h_i)\Big) = \Theta \ell(p).
    \]
\end{proof}

\begin{lemma}[{\cite[Lemma 3.1]{OsinRHG}}]
\label{lem:osinisoperimetric}
    There is  a constant \(M \geq 1\) such that if \(h_1, \dots, h_n\) are isolated \(\mathcal{H}\)-components of a cycle \(q\) in \(\Gamma(G,X\cup\mathcal{H})\), then
    \[
        \sum_{i=1}^n \abs{h_i}_X \leq M\ell(q).
    \]
\end{lemma}

\begin{lemma}
\label{lem:qgds_with_long_comps}
    For any \(\lambda \geq 1\), \(c \geq 0\) and \(A \geq 0\) there is a constant \(\eta = \eta(\lambda,c,A) \geq 0\) such that the following is true.

    Suppose that \(p\) is a \((\lambda,c)\)-quasigeodesic path in $\ga$ possessing an isolated \(\mathcal{H}\)-component \(h\) such that \(\abs{h}_X \geq \eta\). Then \(\abs{p}_X \geq A\).
\end{lemma}

\begin{proof} 
    Let $M \ge 1$ be the constant from Lemma~\ref{lem:osinisoperimetric}, and set
    \begin{equation} 
    \label{eq:defn_of_zeta}
        \eta = M(1 + \lambda)A + Mc.
    \end{equation}

    Let \(q\) be a path in \(\Gamma(G,X\cup\mathcal{H})\), labelled by a word over \(X^{\pm 1}\), with endpoints \(q_- = p_-\) and \(q_+ = p_+\), such that \(\ell(q) = \abs{p}_X\).

    Consider the cycle \(r = pq^{-1}\) in \(\Gamma(G,X\cup\mathcal{H})\), formed by concatenating \(p\) and the inverse of \(q\).
    By the quasigeodesicity of \(p\), \(\ell(p) \leq \lambda \abs{p}_{X\cup\mathcal{H}} + c \leq \lambda \abs{p}_X + c\).
    Now \(\ell(r) = \ell(p) + \ell(q)\), therefore
    \begin{equation}
    \label{eq:len_of_cycle_r}
        \ell(r) \leq (1+ \lambda) \abs{p}_X + c.
    \end{equation}

    Since \(h\) is isolated in \(p\) it must also be an isolated $\mathcal H$-component of  the cycle \(r\) (because all edges of $q$ are labelled by letters from $X^{\pm 1}$). Hence $\abs{h}_X \leq M \ell(r)$
    by Lemma~\ref{lem:osinisoperimetric}, so \eqref{eq:len_of_cycle_r} implies that
    \begin{equation}
    \label{eq:bound_on_h_in_r}
        |p|_X \ge \frac{1}{1+\lambda} (\ell(r)-c) \ge \frac{1}{M(1+\lambda)}(|h|_X-Mc).
    \end{equation}
    Combining the above inequality with \eqref{eq:defn_of_zeta} and the assumption that \(\abs{h}_X \geq \eta\), we obtain
    the desired bound $\abs{p}_X \ge A$.
\end{proof}

\begin{proposition}[{\cite[Proposition 3.2]{OsinFilling}}]
\label{prop:osinpolygon}
    There is a constant \(L \ge 0\) such that if \(\Delta\) is a geodesic triangle in \(\Gamma(G,X\cup\mathcal{H})\) and some side $p$ is an isolated $\mathcal{H}$-component of $\Delta$  then $ \abs{p}_X \leq L$.
\end{proposition}

\begin{lemma}
\label{lem:isol_comp_in_triangles_are_short}
    There is a constant $L \ge 0$ such that if $p_1$ and $p_2$ are geodesic paths in $\ga$ with $(p_1)_+=(p_2)_-$, and $s$ and $t$ are connected $H_\nu$-components of $p_1$, $p_2$ respectively, for some $\nu \in \Nu$, then $d_X(s_+,t_-) \le L$.
\end{lemma}

\begin{proof} 
    Let $L \ge 0$ be the constant provided by Proposition~\ref{prop:osinpolygon}.

    Since the component $s$ of $p_1$ is connected to the component $t$ of $p_2$, we know that $h=(s_+)^{-1}t_- \in H_\nu$. 
    If $h=1$ then $s_+=t_-$ and there is nothing to prove, otherwise $s_+$ and $t_-$ are endpoints of an edge $e$ labelled by $h$ in $\ga$.

    Consider the geodesic triangle $\Delta$ with vertices $s_+$, $(p_1)_+$ and $t_-$, where the sides $[s_+,(p_1)_+] $ and $[(p_1)_+,t_-]$ are chosen to be subpaths of $p_1$ and $p_2$ respectively, and the side $[s_+,t_-]$ is the edge $e$.

    If $v \in [s_+,(p_1)_+]$ is a vertex belonging to the left coset $s_+H_\nu$ then  $\dxh(s_-,v)=1$ and $s_+ \in [s_-,v]$ in $p_1$. 
    Since $\dxh(s_-,s_+)=1$ and $p_1$ is geodesic, we can conclude that $v=s_+$. 
    Similarly, the only vertex of $[(p_1)_+,t_-]$ which belongs to the left coset $t_-H_\nu=s_+H_\nu$ is $t_-$. 
    It follows that the edge $e$ is an isolated $H_\nu$-component of $\Delta$. Hence  $d_X(s_+,t_-) \le L$ by Proposition~\ref{prop:osinpolygon}.
\end{proof}

\begin{proposition}[{\cite[Theorem 3.26]{OsinRHG}}]
\label{prop:osinslimtriangles}
    Let \(\Delta\) be a combinatorial geodesic triangle in \(\Gamma(G,X\cup\mathcal{H})\) with sides \(p\), \(q\) and \(r\). 
    There is a constant \(\sigma = \sigma(G,\mathcal{H},X)  \in \NN_0\) such that for any vertex \(u \in p\), there is a vertex \(v \in q \cup r\) with \(d_X(u,v) \leq \sigma\).
\end{proposition}

\begin{definition}[$k$-similar paths]
    Let \(p\) and \(q\) be paths in \(\Gamma(G,X\cup\mathcal{H})\), and let \(k \geq 0\).
    The paths \(p\) and \(q\) are said to be \emph{\(k\)-similar} if \(d_X(p_-,q_-) \leq k\) and \(d_X(p_+,q_+) \leq k\).
\end{definition}

\begin{proposition}[{\cite[Proposition 3.15, Lemma 3.21 and Theorem~3.23]{OsinRHG}}]
\label{prop:osinbcp}
    For any \(\lambda \geq 1\), \(c, k \geq 0\) there is a constant \(\kappa = \kappa(\lambda,c,k) \geq 0\) such that if \(p\) and \(q\) are $k$-similar \((\lambda,c)\)-quasigeodesics in \(\Gamma(G,X\cup\mathcal{H})\) and \(p\) is without backtracking, then
    \begin{enumerate}
        \item for every phase vertex \(u\) of \(p\), there is a phase vertex \(v\) of \(q\) with \(d_X(u,v) \leq \kappa\);
        \item every \(\mathcal{H}\)-component \(s\) of \(p\), with \(\abs{s}_X \geq \kappa\), is connected to an \(\mathcal{H}\)-component of \(q\).
    \end{enumerate}
    Moreover, if \(q\) is also without backtracking then
    \begin{enumerate}[resume]
        \item if \(s\) and \(t\) are connected \(\mathcal{H}\)-components of \(p\) and \(q\) respectively, then
        \[
            \max \{ d_X(s_-,t_-), d_X(s_+,t_+) \} \leq \kappa .
        \]
    \end{enumerate}
\end{proposition}

\subsection{Quasigeodesicity of paths with long components}
One of the tools for proving Theorem~\ref{thm:metric_qc} will be the next result of Mart\'{i}nez-Pedroza from \cite{MPComb}.

\begin{proposition}[{\cite[Proposition 3.1]{MPComb}}]
\label{prop:mpquasigeodesic-original}
     There are constants \(\zeta_0 \ge 0\) and \(\lambda_0 \geq 1\) such that the following holds. If  \(q=r_0s_1 \dots r_n s_{n+1}\) is a concatenation of geodesic paths \(r_0, s_1, \dots, r_{n}, s_{n+1}\) in \(\Gamma(G,X\cup\mathcal{H})\) such that
     \begin{enumerate}
         \item \(s_i\) is an \(\mathcal{H}\)-component of \(q\), for each \(i = 1, \dots, n+1\),
         \item \(\abs{s_i}_X \geq \zeta_0\), for every $i=1,\dots,n+1$,
         \item \(s_i\) is not connected to \(s_{i+1}\), for every \(i = 1, \dots, n\),
     \end{enumerate}
     then \(q\) is \((\lambda_0, 0)\)-quasigeodesic in \(\Gamma(G,X\cup\mathcal{H})\) without backtracking.
 \end{proposition}

We will actually need a slightly more general version of Proposition~\ref{prop:mpquasigeodesic-original}, as follows.

\begin{proposition}
\label{prop:mpquasigeodesic}
    There exist constants \(\lambda \geq 1\) and $c \ge 0$ such that for every \(\rho \ge 0\) there is \(\zeta_1 > 0\) such that the following holds. 
    Suppose that \(p=a_0b_1a_1 \dots b_na_n\) is a concatenation of geodesic paths \(a_0, b_1, \dots, b_n, a_n\) in \(\Gamma(G,X\cup\mathcal{H})\) such that
    \begin{enumerate}
        \item \(b_i\) is an \(\mathcal{H}\)-subpath of \(p\), for each \(i = 1, \dots, n\),
        \item \(\abs{b_i}_X \geq \zeta_1\), for each \(i=1,\dots,n\);
        \item \(b_i\) is not connected to \(b_{i+1}\), for every $i=1,\dots,n-1$;
        \item if $b_i$ is connected to a component $h$ of $a_i$ or $a_{i-1}$ then
        $|h|_X \le \rho$, $i=1,\dots,n$.
    \end{enumerate}
     Then \(p\) is a \((\lambda, c)\)-quasigeodesic  without backtracking.
\end{proposition}

\begin{proof}
    The argument below employs the following trick: for each $i=1,\dots, n$, we replace the $\mathcal{H}$-component of $p$ containing $b_i$ by a single edge $s_i$, and then embed the resulting path $p'$ into a larger path $q$ to which Proposition~\ref{prop:mpquasigeodesic-original} can be applied. Since a subpath of a $(\lambda,c)$-quasigeodesic path without backtracking is again $(\lambda,c)$-quasigeodesic and  without backtracking, this will complete the proof. In order to construct the path $q$ we add an extra infinite peripheral subgroup $Z$ by  embedding $G$ into a larger relatively hyperbolic group $G_1$.

    Let us consider the free product $G_1=G*Z$, where $Z=\langle z \rangle$ is an infinite cyclic group. Since $G$ is hyperbolic relative to the family $\{H_\nu \mid \nu \in \Nu\}$, the group $G_1$ is hyperbolic relative to the union $\{H_\nu \mid \nu \in \Nu\} \cup \{Z\}$ (this can be fairly easily deduced from the definition or from many existing combination theorems for relatively hyperbolic groups, eg, \cite[Corollary~1.5]{Osin-comb}).

    Note that $G$ embeds in $G_1$ and $G_1$ is generated by the finite set $X'=X \sqcup \{z\}$. Let $\mathcal{H'}= \mathcal{H} \sqcup Z\setminus\{1\}$, so that the Cayley graph $\ga$ is naturally a subgraph of the Cayley graph $\Gamma(G_1,X' \cup \mathcal{H}')$. Therefore we can think of $p$ as a path in $\Gamma(G_1,X' \cup \mathcal{H}')$.

    The normal form theorem for free products (\cite[Theorem~IV.1.2]{LS}) implies that the embedding of $G$ into $G_1$ is isometric with respect to both proper and relative metrics, more precisely
    \begin{equation}
    \label{eq:isom}
        d_X(g,h)=d_{X'}(g,h) ~\text{ and }~ \dxh(g,h)=d_{X' \cup \mathcal{H}'} (g,h), \text{ for all } g,h \in G.
    \end{equation}
    An alternative way to see this is to use the retraction $r:G_1 \to G$, such that $r(x)=x$ for all $x \in X$ and $r(z)=1$. Then $r(X')=X \cup \{1\}$, $r(H_\nu)=H_\nu$, for all $\nu \in \Nu$, and $r(Z)=\{1\}$.

    Let $\zeta_0 \ge 0$ and $\lambda_0 \ge 1$ be the constants provided by Proposition~\ref{prop:mpquasigeodesic-original} applied to the group $G_1$, its finite generating set $X'$ and its Cayley graph $\Gamma(G_1,X' \cup \mathcal{H}')$. Set $\zeta_1=\zeta_0+2\rho+1 >0$.

    For each $i=1,\dots,n$, let $t_i$ denote the $H_{\nu_i}$-component of $p$ containing the edge $b_i$, $\nu_i \in \mathcal{N}$. Note that $t_1,\dots,t_n$ are pairwise distinct by condition (3), in particular no two of them share a common edge.
    In view of Remark~\ref{rem:comp_of_geod_is_an_edge}, for every $i=1,\dots,n$ we can represent $t_i$ as a concatenation $t_i=h_{i-1}b_if_i$, where
    \begin{itemize}
        \item $h_{i-1}$ is either the last edge and an $H_{\nu_{i}}$-component of $a_{i-1}$ if $a_{i-1}$ ends with an $H_{\nu_{i}}$-component, or $h_{i-1}$ is the trivial path, consisting of the vertex $(a_{i-1})_+$, if $a_{i-1}$ does not end with an $H_{\nu_{i}}$-component;
        \item $f_i$ is the first edge and an $H_{\nu_i}$-component of $a_i$ if $a_{i}$ starts with an $H_{\nu_{i}}$-component, or $f_i$ is the trivial path, consisting of the vertex $(a_i)_-$, if $a_{i}$ does not start with an $H_{\nu_{i}}$-component.
    \end{itemize}
    Note that for each $i=1,\dots, n$ we have $|h_{i-1}|_X \le \rho$ and $|f_i|_X \le \rho$, by condition (4). By (2) and the triangle inequality we get
    \begin{equation}
    \label{eq:length_of_t_i}
        |t_i|_X \ge |b_i|_X-2\rho \ge \zeta_0+1,~ \text{ for } i=1,\dots,n.
    \end{equation}

    Therefore $p$ decomposes as a concatenation \[p=r_0t_1r_1 \dots t_n r_n,\] where $r_i$ is a subpath of $a_i$, $i=0,\dots,n$, so that $a_0=r_0h_0$, $a_1=f_1r_1h_1$, $\dots$, $a_n=f_nr_n$.

    By \eqref{eq:length_of_t_i} the endpoints of the $H_{\nu_i}$-component $t_i$ of $p$ must be distinct, hence there is an edge $s_i$ joining them in $\ga$, such that $\Lab(s_i) \in H_{\nu_i}\setminus\{1\}$, $i=1,\dots,n$. Now,  \eqref{eq:length_of_t_i} and \eqref{eq:isom} imply that
    \begin{equation*}
    \label{eq:length_of_s_i}
        |s_i|_{X'}=|t_i|_{X'}=|t_i|_X \ge \zeta_0,~ \text{ for } i=1,\dots,n.
    \end{equation*}

    Choose $k \in \NN$ so that $|z^k|_{X'} \ge \zeta_0$ and let $s_{n+1}$ be the edge in  $\Gamma(G_1,X' \cup \mathcal{H}')$, starting at $p_+=(r_n)_+$ and labelled by $z^k$. Observe that $|s_{n+1}|_{X'}=|z^k|_{X'} \ge \zeta_0$.

    Consider the path $q$ in $\Gamma(G_1,X' \cup \mathcal{H}')$, defined as the concatenation $q=r_0s_1 \dots r_n s_{n+1}$. By \eqref{eq:isom} the paths $r_0,\dots,r_n$ are still geodesic in  $\Gamma(G_1,X' \cup \mathcal{H}')$, and $s_1,\dots,s_{n+1}$ are  $\mathcal{H}'$-components of $q$, by construction. Finally, $s_i$ is not connected to $s_{i+1}$, for $i=1,\dots,n-1$, because elements of $G$ that belong to different $H_\nu$-cosets continue to do so in $G_1$, and $s_n$ is not connected to $s_{n+1}$ because $H_{\nu_n}$ and $Z$ are distinct peripheral subgroups of $G_1$. Therefore all of the assumptions of Proposition~\ref{prop:mpquasigeodesic-original} are satisfied, which allows us to conclude that the path $q$ is $(\lambda_0,0)$-quasigeodesic without backtracking in  $\Gamma(G_1,X' \cup \mathcal{H}')$.

    Consequently, the path $p'=r_0s_1r_1 \dots s_n r_n$ is $(\lambda_0,0)$-quasigeodesic without backtracking in  $\Gamma(G_1,X' \cup \mathcal{H}')$, as a subpath of $q$. Since $p'$ only contains vertices and edges from $\ga$, we see that $p'$ is also  $(\lambda_0,0)$-quasigeodesic without backtracking in  $\ga$.

    Now, the original path $p$ can be obtained by replacing the edges $s_1,\dots, s_n$ of $p'$ by paths $t_1,\dots,t_n$, each of which has length at most $3$. Hence, by Lemma~\ref{lem:perturbed_quasigeodesic}, $p$ is $(3\lambda_0,18\lambda_0+6)$-quasigeodesic. Since $p'$ is without backtracking and every $\mathcal{H}$-component of $p$ is connected to an $\mathcal H$-component of $p'$ (and vice-versa), by construction, the path $p$ must also be without backtracking.

    Thus we have shown that the path $p$ is $(\lambda,c)$-quasigeodesic without backtracking in $\ga$, where $\lambda=3\lambda_0$ and $c=18\lambda_0+6$.
\end{proof}

\subsection{Quasiconvex subsets in relatively hyperbolic groups}\label{subsec:qc_subset_in_rh_gps}
In this paper we shall use the definition of a relatively quasiconvex subgroup given by Osin in \cite{OsinRHG}. 
For convenience we state it in the case of arbitrary subsets rather than just subgroups.

\begin{definition}[Relatively quasiconvex subset]
\label{def:rel_qc}
    A subset $Q \subseteq G$ is said to be \emph{relatively quasiconvex} (with respect to $\{H_\nu \mid \nu \in \Nu\}$) if there exists $\varepsilon \ge 0$ such that for every geodesic path $q$ in $\Gamma(G,X \cup \mathcal{H})$, with $q_-,q_+ \in Q$, and every vertex $v$ of $q$ we have $d_X(v,Q) \le \varepsilon$.

    Any number $\varepsilon \ge 0$ as above will be called a \emph{quasiconvexity constant} of $Q$.
\end{definition}

Osin proved that relative quasiconvexity of a subset is independent of the choice of a finite generating set $X$ of $G$: see \cite[Proposition 4.10]{OsinRHG} -- the proof there is stated for relatively quasiconvex subgroups but actually works more generally for relatively quasiconvex subsets.

We outline some basic properties of quasiconvex subsets and subgroups of $G$ in the next two lemmas.

\begin{lemma}
\label{lem:props_of_qc_subsets}
    Let $Q$ be a relatively quasiconvex subset of $G$. Then
    \begin{enumerate}[label=(\alph*)]
        \item the subset $gQ$ is relatively quasiconvex, for every $g \in G$;
        \item if $T \subseteq G$ lies at a finite $d_X$-Hausdorff distance from $Q$ then $T$ is relatively quasiconvex.
    \end{enumerate}
\end{lemma}

\begin{proof} 
    Claim (a) follows immediately from the fact that left multiplication by $g$ induces an isometry of $G$ with respect to both the proper metric $d_X$ and the relative metric $d_{X \cup \mathcal{H}}$.

    To prove claim (b), suppose that $\varepsilon \ge 0$ is a quasiconvexity constant of $Q$ and  the $d_X$-Hausdorff distance between $Q$ and $T$ is less than $k\in \NN$.
    Consider any geodesic path $t$ in $\Gamma(G, X\cup\mathcal{H})$ with $t_-,t_+ \in T$, and take any vertex $v$ of $t$. 
    Then there are $x,y \in Q$ such that $d_X(x,t_-)\le k$ and $d_X(y,t_+) \leq k$. Let $q$ be any geodesic connecting $x$ with $y$. 
    Then $q$ is $k$-similar to $t$, hence there is a vertex $u$ of $q$ such that $d_X(v,u) \leq \kappa$, where $\kappa = \kappa(1,0,k) \ge 0$ is the global constant given by Proposition~\ref{prop:osinbcp} applied to \(k\)-similar geodesics.  
    By the relative quasiconvexity of $Q$, there exists $w \in Q$ such that $d_X(u,w) \leq \varepsilon$. Moreover, \(d_X(w,T) \leq k\) by assumption. 
    Therefore $d_X(v,T) \leq \kappa+\varepsilon+k$, thus $T$ is relatively quasiconvex in $G$.
\end{proof}

\begin{lemma}
\label{lem:props_of_qc_sbgps}
    Suppose that $Q \leqslant G$ is a relatively quasiconvex subgroup. 
    Then for all $g \in G$ and $Q' \leqslant_f Q$ the subgroups $gQg^{-1}$ and $Q'$ are relatively quasiconvex in $G$.
\end{lemma}

\begin{proof} 
    By claim (a) of Lemma~\ref{lem:props_of_qc_subsets}, the coset $gQ$ is relatively quasiconvex and the $d_X$-Hausdorff distance between this coset and $gQg^{-1}$ is at most $|g|_X$, hence $gQg^{-1}$ is relatively quasiconvex in $G$ by claim (b) of the same lemma.

    Suppose that $Q=\bigcup_{i=1}^m Q' h_i$, where $h_i \in Q$, $i=1,\dots,m$. 
    Then the $d_X$-Hausdorff distance between $Q$ and $Q'$ is bounded above by $\max\{|h_i|_X \mid 1\le i \le m\}$, so $Q'$ is relatively quasiconvex by Lemma~\ref{lem:props_of_qc_subsets}(b).
\end{proof}

\begin{corollary} 
\label{cor:parab->qc} 
    Any parabolic subgroup of $G$ is relatively quasiconvex.
\end{corollary}

\begin{proof} 
    Let $H=gQg^{-1}$ be a parabolic subgroup, where $g \in G$ and $Q \leqslant H_\nu$, for some $\nu \in \Nu$. 
    The subgroup $Q$ is relatively quasiconvex in $G$ (with quasiconvexity constant $0$), because any geodesic connecting two elements of $Q$ consists of a single edge in $\ga$. 
    Therefore $H$ is relatively quasiconvex by Lemma~\ref{lem:props_of_qc_sbgps}.
\end{proof}

\begin{lemma}
\label{lem:fg_qc_int_parab_is_fg} 
    Let $P$ be a maximal parabolic subgroup of $G$ and let $Q$ be a finitely generated relatively quasiconvex subgroup of $G$. 
    Then the subgroups $P$ and $Q \cap P$ are finitely generated.
\end{lemma}

\begin{proof} 
    The fact that each $H_\nu$ is finitely generated, provided $G$ is finitely generated, was proved by Osin in \cite[Theorem 1.1]{OsinRHG}.

    Now, Hruska \cite[Theorem 9.1]{HruskaRHCG} proved that every quasiconvex subgroup $Q$ of $G$ is itself relatively hyperbolic and maximal parabolic subgroups of $Q$ are precisely the infinite intersections of $Q$ with maximal parabolic subgroups of $G$. 
    In other words, if $P \leqslant G$ is maximal parabolic, then $Q \cap P$ is either finite or a maximal parabolic subgroup of $Q$. 
    Combined with Osin's result \cite[Theorem 1.1]{OsinRHG} mentioned above we can conclude that if $Q$ is finitely generated then so is $Q \cap P$, as required.
\end{proof}

The following property of quasiconvex subgroups will be useful.

\begin{lemma}
\label{lem:shortening} 
    Let $Q, R \leqslant G$ be relatively quasiconvex subgroups of $G$. For every $\zeta \ge 0$ there exists a constant \(\mu=\mu(\zeta) \geq 0\) such that the following holds.

    Suppose $x \in G$, \(a \in Q \), \(b \in R\) are some elements, $[x,xa]$ and $[x,xb]$ are geodesic paths in $\ga$, and \(u \in [x,xa]\), \(v \in [x,xb]\) are vertices such that \(d_X(u,v) \leq \zeta\).
    Then there is an element \(z \in x(Q \cap R)\) such that \(d_X(u,z) \leq \mu\) and \(d_X(v,z) \leq \mu\).
\end{lemma}

\begin{proof}
    Denote by \(\varepsilon \geq 0\) a quasiconvexity constant of the subgroups \(Q\) and \(R\).
    After applying the left translation by $x^{-1}$, which is an isometry with respect to both metrics $d_X$ and $\dxh$, we can assume that $x=1$. 
    Let $K'=K'(Q,R,\varepsilon+\zeta)$ be the constant given by Lemma~\ref{lem:nbhdintersection}.

    Since $x=1 \in Q \cap R$, $xa=a \in Q$ and $xb=b \in R$, by the relative quasiconvexity of $Q$ and $R$ we know that $u \in N_X(Q,\varepsilon)$ and $v \in N_X(R,\varepsilon)$. 
    By the assumptions $d_X(u,v) \le \zeta$, it follows that $u \in N_X(Q,\varepsilon+\zeta) \cap N_X(R,\varepsilon+\zeta)$, hence $u \in N_X(Q \cap R,K')$ by Lemma~\ref{lem:nbhdintersection}.

    Thus there exists $z \in Q \cap R$ such that $d_X(u,z) \le K'$, and, hence, $d_X(v,z) \le K'+\zeta$ by the triangle inequality. 
    Therefore the statement of the lemma holds for $\mu=K'+\zeta$.
\end{proof}

The next combination theorem was proved by Mart\'{i}nez-Pedroza.
\begin{theorem}[{\cite[Theorem~1.1]{MPComb}}]
\label{thm:M-P_comb} 
    Let $G$ be a relatively hyperbolic group generated by a finite set $X$. 
    Suppose that $Q$ is a relatively quasiconvex subgroup of $G$, $P$ is a maximal parabolic subgroup of $G$ and $D=Q \cap P$. 
    There is a constant $C \ge 0$ such that the following holds. 
    If $H \leqslant P$ is any subgroup satisfying
    \begin{enumerate}
        \item $H \cap Q=D$, and
        \item $\minx( H \setminus D) \ge C$,
    \end{enumerate}
    then the subgroup $A=\langle H,Q \rangle$ is relatively quasiconvex in $G$ and is naturally isomorphic to the amalgamated free product $H*_{D} Q$.

    Moreover, for every maximal parabolic subgroup $T$ of $G$, there exists $u \in A$ such that  \[\text{either } A \cap T \subseteq uQu^{-1}~\text{ or }~ A \cap T \subseteq uHu^{-1}.\]
\end{theorem}

\part{Quasiconvexity of virtual joins}
\label{part:metric_qc_double_cosets}
This part of the paper is mostly devoted to the proofs of Theorems \ref{thm:metric_qc} and \ref{thm:sep->qc_intro}. Let us start by giving  brief outlines of the arguments.

Suppose \(G\) is a group generated by finite set \(X\) and hyperbolic relative to a collection of subgroups \(\{H_\nu \, | \, \nu \in \Nu\}\). Denote \(\mathcal{H} = \bigsqcup_{\nu \in \Nu} H_\nu \setminus \{1\}\) and take any \(A \geq 0\).
Consider two finitely generated relatively quasiconvex subgroups \(Q, R \leqslant G \). Set \(S = Q \cap R\) and suppose that \(Q' \leqslant Q\) and \(R' \leqslant R\) are subgroups satisfying conditions \descref{C1}-\descref{C5} from Subsection~\ref{subsec:3.1}, with some finite collection of maximal parabolic subgroups $\mathcal{P}$ of $G$ (which is independent of $A$) and parameters \(B\) and \(C\) that are sufficiently large with respect to \(A\).

Every element \(g \in \langle Q', R' \rangle\) can be written as a product of elements of \(Q'\) and \(R'\), which gives rise to a broken geodesic line in \(\Gamma(G, X \cup \mathcal{H})\) (not necessarily uniquely), whose label represents \(g\) in $G$.
We choose a path \(p\) from the collection of such broken lines, representing \(g\), that is minimal in a certain sense.
The path \(p\) may fail to be uniformly quasigeodesic, as it may travel through $H_\nu$-cosets for an arbitrarily long time.
We do, however, have some metric control over such instances of backtracking, using the fact that \(Q'\) and \(R'\) satisfy conditions \descref{C1}-\descref{C5} and the minimality of \(p\).

We construct a new path from \(p\), which we call the \emph{shortcutting} of \(p\), that turns out to be uniformly quasigeodesic.
Informally speaking, the shortcutting of \(p\) is obtained by replacing each maximal instance of backtracking in consecutive geodesic segments of \(p\) with a single edge, then connecting these edges in sequence by geodesics.
The resulting path can be seen to satisfy the hypotheses of Proposition~\ref{prop:mpquasigeodesic}.
It follows that the shortcutting of \(p\) is uniformly quasigeodesic, and hence \(\langle Q', R' \rangle\) is relatively quasiconvex.
Properties \descref{P2} and \descref{P3} also follow from this quasigeodesicity, giving us Theorem~\ref{thm:metric_qc}.

Now suppose that \(G\) is QCERF and its peripheral subgroups are double coset separable. 
In Theorem~\ref{thm:sep->qc_comb} we use the separability assumptions on \(G\) and \(\{H_\nu \, | \, \nu \in \Nu\}\) to deduce the existence of a finite index subgroup \(M \leqslant_f G\) such that \(Q' = Q \cap M \leqslant_f Q, R' = R \cap M \leqslant_f R\) satisfy conditions \descref{C1}-\descref{C5} with constants \(B\) and \(C\) large enough to apply Theorem~\ref{thm:metric_qc} (as suggested in Remark~\ref{rem:sep->metric}). 
Conditions \descref{C1} and \descref{C4} are essentially automatic.
Conditions \descref{C2}, \descref{C3} and \descref{C5} can be assured to hold for the subgroups \(Q'\) and \(R'\) using Lemma~\ref{lem:sep->large_minx} by the QCERF condition on \(G\), separability of double cosets \(PS\) (where \(P\) is one of finitely many maximal parabolic subgroups) and double coset separability of the peripheral subgroups, respectively.

The remaining technical difficulty is in showing that the double cosets of the form \(PS\) as above are separable in \(G\).
To this end, we prove a general result about lifting separability of certain double cosets in amalgamated free products.
This is then combined with a result of Mart\'{i}nez-Pedroza (Theorem~\ref{thm:M-P_comb}), allowing us to deduce  Theorem~\ref{thm:sep->qc_intro} from Theorem~\ref{thm:metric_qc}.


\section{Path representatives}
\label{sec:path_reps}
Let us set the notation that will be used in the next few sections.

\begin{convention} 
\label{conv:main}
    We fix a group  $G$, generated by a finite set $X$, which is hyperbolic relative to a finite family of subgroups \(\lbrace H_\nu \, | \, \nu \in \Nu \rbrace\). We let $\mathcal{H}=\bigsqcup_{\nu \in \Nu} (H_\nu\setminus\{1\})$. 
    It follows that the Cayley graph $\ga$ is $\delta$-hyperbolic, for some $\delta \in \NN$ (see Lemma~\ref{lem:Cayley_graph-hyperbolic}).

    Furthermore, we assume that  \(Q, R \leqslant G\) are fixed relatively quasiconvex subgroups of \(G\), with a quasiconvexity constant \(\varepsilon \ge 0\), and denote $S=Q \cap R$.
\end{convention}

In this section \(Q'\) and \(R'\) will denote some subgroups of $Q$ and $R$ respectively. 
We will introduce path representatives of elements in $\langle Q',R' \rangle$ and will order such representatives by their types. 
This will be crucial in our proof of Theorem~\ref{thm:metric_qc}.

\begin{definition}[Path representative, I]
\label{def:path_reps}
    Consider an arbitrary element \(g \in \langle Q', R' \rangle\).
    Let \(p=p_1 \dots p_n\) be a broken line in \(\Gamma(G,X\cup\mathcal{H})\) with geodesic segments \(p_1, \dots, p_n\),  such that $\elem{p}=g$ and \(\elem{p_i} \in Q' \cup R'\) for each \(i \in \{1,\dots,n\}\).
    We will call \(p\) a \emph{path representative} of \(g\).
\end{definition}

To choose an optimal path representative we define their types.

\begin{definition}[Type of a path representative, I] 
\label{def:type_of_path_rep}
    Suppose that $p=p_1\dots p_n$ is a broken line in $\ga$.
    For each $i=1,\dots,n$, let $T_i$ denote the set of all \(\mathcal{H}\)-components of $p_i$, and let $T= \bigcup_{i=1}^n T_i$.  
    We define the \emph{type} $\tau(p)$ of $p$ to be the triple 
    \[
        \tau(p)=\Big(n,\ell(p),\sum_{t \in T} |t|_X \Big) \in {\NN_0}^3,
    \]
    where $\ell(p)=\sum_{i=1}^n \ell(p_i)$ is the length of $p$.
\end{definition}

\begin{definition}[Minimal type]
    Given \(g \in \langle Q', R' \rangle\), the set $\mathcal S$ of all path representatives of $g$ is non-empty. 
    Therefore the subset $\tau(\mathcal{S})=\{\tau(p) \mid p \in \mathcal{S}\} \subseteq {\NN_0}^3$, where ${\NN_0}^3$ is equipped with the lexicographic order, will have a unique minimal element.

    We will say that $p=p_1 \dots p_n$ is a \emph{path representative of $g$ of minimal type} if $\tau(p)$ is the minimal element of $\tau(\mathcal S)$.
\end{definition}

\begin{remark}
\label{rem:alt}
    Note that if \(p_1\) and \(p_2\) are paths with \((p_1)_+ = (p_2)_-\) whose labels both represent elements of \(Q'\) (or, respectively, both \(R'\)), then the label of any geodesic \([(p_1)_-,(p_2)_+]\) also represents an element of \(Q'\) (respectively, \(R'\)).
    Hence in a path representative of $g \in \langle Q',R' \rangle$ of minimal type, the labels of the consecutive segments necessarily alternate between representing elements of \(Q' \setminus (Q' \cap R')\) and \(R' \setminus (Q' \cap R')\), whenever \(g\) is not itself an element of \(Q' \cap R'\).
\end{remark}

The minimality of the type of a path representative is thus a numerical condition on the total lengths of the paths $p_i$ and the total lengths of their components. 
In the next few sections we will study local properties induced by this global condition. 
The first such property is stated in the next lemma.

\begin{notation}
    Let \(x,y,z \in G\).
    We will write \(\langle x, y \rangle^{rel}_z=\frac12 (\dxh(x,z)+\dxh(y,z)-\dxh (x,y))\) to denote the Gromov product of \(x\) and \(y\) with respect to \(z\) in the relative metric \(d_{X\cup\mathcal{H}}\).
\end{notation}

\begin{lemma}[Gromov products are bounded]
\label{lem:bddinnprod}
    There is a constant \( C_0 \geq 0\) such that the following holds.

    Let $Q' \leqslant Q$ and $R' \leqslant R$ be subgroups satisfying condition \descref{C1}. If \(p=p_1 \dots p_n\) is a minimal type path representative of an element \(g \in \langle Q', R' \rangle\) and $f_0, \dots, f_n \in G$ are the nodes of $p$ (that is,  $f_{i-1}=(p_i)_-$, for $i=1,\dots,n$, and $f_n=(p_n)_+$)  then
    \(\langle f_{i-1}, f_{i+1} \rangle_{f_i}^{rel} \leq C_0\) for each \(i= 1, \dots, n-1\).
\end{lemma}

\begin{proof} 
    Let \(\sigma \in \NN_0\) be the constant from Proposition~\ref{prop:osinslimtriangles} and let \(\mu=\mu(\sigma) \geq 0\) be given by Lemma~\ref{lem:shortening}.
    Set \(C_0 = \mu + \delta + 2 \sigma +2\), and assume that \(p=p_1 \dots p_n\) is a path representative of \(g \in \langle Q', R' \rangle\) of minimal type.

    Take any $i \in \{1,\dots,n-1\}$. Choose vertices \(u \in p_i \) and \(v \in p_{i+1}\) so that
    \(d_{X\cup\mathcal{H}}(f_i,u) = d_{X\cup\mathcal{H}}(f_i,v) = \lfloor\langle f_{i-1}, f_{i+1} \rangle_{f_i}^{rel} \rfloor \). As \(\Gamma(G,X\cup\mathcal{H})\) is \(\delta\)-hyperbolic, we must have \(d_{X\cup\mathcal{H}}(u,v) \leq \delta\).

    If \(\langle f_{i-1}, f_{i+1} \rangle_{f_i}^{rel} < C_0\) then we are done, so suppose otherwise.
    Then $\dxh(u,f_i) \ge \delta+\sigma+1 \in \NN$, so
    there is a vertex \(u_1\) on the subpath $[u,f_i]$ of $p_i$ such that \[ d_{X\cup\mathcal{H}}(u_1,u) = \delta + \sigma +1.\]

    Applying Proposition~\ref{prop:osinslimtriangles} to the geodesic triangle \(\Delta\) with sides $[u,f_i]$, $[f_i,v]$ and $[u,v]$ (here we choose $[f_i,v]$ to be a subpath of $p_{i+1}$), we can find some vertex \(v_1 \in [u,v] \cup [f_i,v]\) with \(d_X(v_1,u_1) \leq \sigma\) .
    If \(v_1 \in [u,v]\), then, by the triangle inequality, \[d_{X\cup\mathcal{H}}(u_1,u) \le \dxh(u_1,v_1)+\dxh(u,v)\leq \sigma + \delta,\] which would contradict the choice of \(u_1\).
    Therefore it must be that \(v_1 \in [f_i,v]\) (see Figure~\ref{fig:bdd_inn_prod}).
    \begin{figure}[ht]
        \centering
        \includegraphics{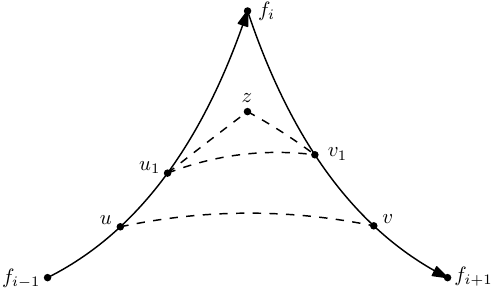}
        \caption{We obtain a different path representative for \(g\) by replacing \(p_i\) and \(p_{i+1}\) with geodesics from \(f_{i-1}\) to \(z\) to \(f_{i+1}\). }
        \label{fig:bdd_inn_prod}
    \end{figure}

    Since the path representative $p$ has minimal type, in view of Remark~\ref{rem:alt} we must have either $\elem{p_i} \in Q'$ and $\elem{p_{i+1}} \in R'$ or $\elem{p_i} \in R'$ and $\elem{p_{i+1}} \in Q'$. Without loss of generality let us assume the former. 
    We can apply Lemma~\ref{lem:shortening} to find \(z \in f_i(Q \cap R)\) with \(d_X(u_1,z) \leq \mu\) and \(d_X(v_1,z) \leq \mu\). 
    Let $p_i'$ be a geodesic path in $\ga$ joining $f_{i-1}=(p_i)_-$ with $z$ and let $p_{i+1}'$ be a geodesic path joining $z$ with $f_{i+1}=(p_{i+1})_+$. 
    Observe that $f_{i-1} \in f_iQ'$ and $Q \cap R \subseteq Q'$ by \descref{C1}, whence
    \[ 
        \elem{p'_i}=f_{i-1}^{-1}z \in Q' f_i^{-1} f_i (Q \cap R)=Q'.
    \]
    Similarly, $\elem{p_{i+1}'} \in R'$. It follows that the path $p'=p_1 \dots p_{i-1} p_i' p_{i+1}' p_{i+2} \dots p_n$ is also a path representative of the same element $g \in \langle Q',R' \rangle$.

    Since $p$ has minimal type, by the assumption, it must be that $\ell(p_i)+\ell(p_{i+1}) \le \ell(p_i')+\ell(p_{i+1}')$, which can be re-written as
    \begin{equation}
    \label{eq:ineq_on_dist-1}
        \dxh(f_{i-1},f_i)+\dxh(f_i,f_{i+1}) \le \dxh(f_{i-1},z)+\dxh(z,f_{i+1}).
    \end{equation}

    Since $u_1 \in p_i$, we have $\dxh(f_{i-1},f_i)=\dxh(f_{i-1},u_1)+\dxh(u_1,f_i)$. 
    On the other hand, \[\dxh(f_{i-1},z) \le \dxh(f_{i-1},u_1)+\dxh(u_1,z) \le \dxh(f_{i-1},u_1)+\mu,\] by the triangle inequality. 
    Similarly, 
    \[
        \dxh(f_{i},f_{i+1})=\dxh(f_{i},v_1)+\dxh(v_1,f_{i+1}) \text{ and } \dxh(z,f_{i+1}) \le \dxh(v_1,f_{i+1})+\mu.
    \]
    Combining the above inequalities with \eqref{eq:ineq_on_dist-1}, we obtain
    \begin{equation}
    \label{eq:Gr_prod_bound}
        \dxh(u_1,f_i)+\dxh(f_i,v_1) \le 2 \mu.
    \end{equation}

    Now, by construction, we have
    \begin{equation}
    \label{eq:d(u1,fi)}
        \dxh(u_1,f_i) = \dxh(u,f_i)-\dxh(u_1,u)=\lfloor\langle f_{i-1}, f_{i+1} \rangle_{f_i}^{rel}\rfloor-(\delta+\sigma+1).
    \end{equation}
    On the other hand, since $\dxh(v_1,u_1) \le \sigma$, we achieve
    \begin{equation}
    \label{eq:d(fi,v1)}
        \dxh(f_i,v_1) \ge \dxh(u_1,f_i)-\dxh(v_1,u_1) \ge \lfloor\langle f_{i-1}, f_{i+1} \rangle_{f_i}^{rel}\rfloor-(\delta+2\sigma+1).
    \end{equation}
    After combining \eqref{eq:d(u1,fi)}, \eqref{eq:d(fi,v1)} and \eqref{eq:Gr_prod_bound}, we obtain
    \[
        2\lfloor\langle f_{i-1}, f_{i+1} \rangle_{f_i}^{rel}\rfloor-(2\delta+3\sigma+2) \le 2 \mu.
    \]
    Therefore, we can conclude that $\langle f_{i-1}, f_{i+1} \rangle_{f_i}^{rel} \le \mu+ \delta+2 \sigma+2=C_0$, as required.
\end{proof}


\section{Adjacent backtracking in path representatives of minimal type}
\label{sec:adj_backtracking}
In this section we continue working under Convention~\ref{conv:main}. 
Our goal here is to study the possible backtracking within two 
adjacent segments in a minimal type path representative.

\begin{lemma}
\label{lem:one_comp_in_cusp_is_bounded}
    For all non-negative numbers $\zeta$ and $\xi$ there exists $\tau=\tau(\zeta,\xi) \ge 0$ such that the following holds.

    Suppose that $Q' \leqslant Q$ and $R' \leqslant R$ are subgroups satisfying \descref{C1}, $g \in \langle Q',R' \rangle$ and $p=p_1\dots p_n$ is a path representative of $g$ of  minimal type.  
    If for some \(i \in \{1,\dots,n-1\}\) $s$ and $t$ are connected $\mathcal{H}$-components of $p_i$ and $p_{i+1}$ respectively, such that $d_X(s_-,t_+) \le \zeta$ and $d_X(s_+,(p_i)_+) \le \xi$, then $|s|_X \le \tau$ and $|t|_X \le \tau$.
\end{lemma}

\begin{proof}
    Let $\mu=\mu(\zeta) \ge 0$ be the constant from Lemma~\ref{lem:shortening}. Since $|X|<\infty$ and $|\Nu|<\infty$ we can define the constant $k \ge 0$ as follows:
    \begin{equation}
    \label{eq:def_of_k}
        k=\max\{K'(Q \cap R, c H_\nu c^{-1},\xi+\mu) \mid \nu \in \Nu,~c \in G,~|c|_X \le \xi\},
    \end{equation}
    where for each $c \in G$ and $\nu  \in \Nu$ the constant $K'(Q \cap R, c H_\nu c^{-1},\xi+\mu)$ is given by Lemma~\ref{lem:nbhdintersection}. 
    Let $L \ge 0$ be the constant from Lemma~\ref{lem:isol_comp_in_triangles_are_short} and set $\tau=2k+2\xi+ \zeta+L \ge 0$.

    Let $p=p_1\dots p_n$ be a path representative of some $g \in \langle Q',R' \rangle$ of minimal type. Suppose that $s$ and $t$ are connected $H_\nu$-components of $p_i$ and $p_{i+1}$ respectively, for some $i \in \{1,\dots,n-1\}$ and $\nu \in \Nu$, such that $d_X(s_-,t_+) \le \zeta$ and $d_X(s_+,(p_i)_+) \le \xi$.

    Note that, by Lemma~\ref{lem:isol_comp_in_triangles_are_short},
    \begin{equation}
    \label{eq:dist_from_s+_to_t-}
        d_X(s_+,t_-) \le L.
    \end{equation}

    Denote $x=(p_i)_+=(p_{i+1})_- \in G$, $a=x^{-1} s_+ \in G$ and $b=x^{-1}t_- \in G$: see Figure~\ref{fig:bounded_comps}.

    \begin{figure}[ht]
        \centering
        \includegraphics{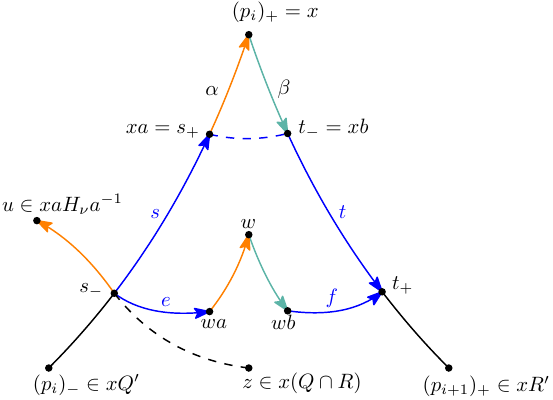}
        \caption{Illustration of Lemma~\ref{lem:one_comp_in_cusp_is_bounded}.} 
        \label{fig:bounded_comps}
    \end{figure}

    Note that
    \begin{equation}
    \label{eq:a_and_b_in_the_same_H-coset}
        aH_\nu=bH_\nu,~\text{ hence } aH_\nu a^{-1}=bH_\nu b^{-1},
    \end{equation}
    because the $H_\nu$-components $s$ and $t$ are connected. 
    Using the lemma hypotheses and (\ref{eq:dist_from_s+_to_t-}) we also have
    \begin{equation}
    \label{eq:mod_a_and_b}
        |a|_X=d_X(x,s_+) \le \xi ~\text{ and }~ |b|_X \le d_X(x,s_+)+d_X(s_+,t_-) \le \xi+L.
    \end{equation}

    In view of Remark~\ref{rem:alt}, without loss of generality we can assume that $\Lab(p_i)$ represents an element of $Q'$ and $\Lab(p_{i+1})$ represents an element of $R'$ in $G$ (the other case can be treated similarly).
    Applying Lemma~\ref{lem:shortening}, we can find $z \in x(Q \cap R)$ such that $d_X(s_-,z) \le \mu$.

    Consider the element $u=s_- a^{-1}= x a \elem{s}^{-1} a^{-1} \in x aH_\nu a^{-1}$, and observe that $d_X(s_-,u)=|a^{-1}|_X \le \xi$.  
    On the other hand, $d_X(s_-,x(Q \cap R)) \le d_X(s_-,z) \le \mu$, whence
    \[
        s_- \in N_X\Bigl(x(Q \cap R),\xi+\mu\Bigr) \cap N_X\Bigl(x aH_\nu a^{-1},\xi+\mu\Bigr).  
    \]

    Therefore, according to Lemma~\ref{lem:nbhdintersection}, there exists $w \in x(Q \cap R \cap aH_\nu a^{-1})$ such that
    \begin{equation}
    \label{eq:d_X(s_-,w)}
        d_X(s_-,w) \le k,
    \end{equation}
    where $k \ge 0$ is the constant defined in \eqref{eq:def_of_k}.

    Let $\alpha$ be the subpath of $p_i$ from $s_+=xa$ to $(p_i)_+=x$. Choose the geodesic path $[wa,w]$ as the translate $wx^{-1} \alpha$. 
    Observe that \(s_- \in xaH_\nu\) and \(wa \in xaH_\nu a^{-1} a = xaH_\nu\) lie in the same \(H_\nu\)-coset.
    Thus $\dxh(s_-,wa) \le 1$; 
    if \(s_- = wa\) we let $e$ be the trivial path in $\ga$ consisting of the single vertex $s_-$, and otherwise we let $e$ be the edge of $\ga$ labelled by an element of $H_\nu \setminus \{1\}$ that joins $s_-$ to $wa$. 
    Define the path $q$ in $\ga$ as the concatenation
    \begin{equation}
    \label{eq:def_of_q}
        q=[(p_i)_-,s_-]\, e\, [wa,w],
    \end{equation}
    where $[(p_i)_-,s_-]$ is chosen as the initial segment of $p_i$.

    Since $\ell(e) \le 1=\dxh(s_-,s_+)$, we can bound the length of the path $q$ from above as follows:
    \begin{equation}
    \begin{split}
    \label{eq:length_of_q}
        \ell(q) & =\dxh((p_i)_-,s_-)+\ell(e)+\dxh(wa,w) \\ 
                & \le \dxh((p_i)_-,s_-)+\dxh(s_-,s_+)+\dxh(xa,x)=\ell(p_i).
    \end{split}
    \end{equation}

    Now we construct a similar path from $w$ to $(p_{i+1})_+$. 
    Let $\beta$ be the subpath of $p_{i+1}$ from $(p_{i+1})_-=x$ to $t_-=xb$. 
    Choose the geodesic path $[w,wb]$ as the translate $wx^{-1} \beta$. 
    Recall that $t_+\in xb H_\nu$ and note that the inclusion $w \in xaH_\nu a^{-1}$, together with \eqref{eq:a_and_b_in_the_same_H-coset}, imply that $wb \in xbH_\nu$ also.
    If $t_+=wb$ then let $f$ be the trivial path in $\ga$ consisting of the single vertex $t_+$, otherwise let \(f\) be the edge in \(\ga\) joining the vertices $wb$ and $t_+$ with $\Lab(f) \in H_\nu \setminus\{1\}$. 
    We now define the path $r$ in $\ga$ as the concatenation
    \begin{equation}
    \label{eq:def_of_r}
        r=[w,wb]\, f \, [t_+,(p_{i+1})_+],
    \end{equation}
    where $[t_+,(p_{i+1})_+]$ is chosen as the ending segment of $p_{i+1}$. 
    Similarly to the case of $q$ we can estimate that
    \begin{equation}
    \label{eq:length_of_r}
        \ell(r) \le \ell(p_{i+1}).
    \end{equation}

    Note that since $q_-=(p_i)_-=x \elem{p_i}^{-1} \in xQ'$, $q_+=w \in x(Q \cap R)$ and $Q \cap R \subseteq Q'$, we have $\elem{q} \in Q'$. Similarly, $\elem{r} \in R'$.

    Let $p_i'$ be a geodesic path from $q_-=(p_i)_-$ to $q_+=w$, and let $p_{i+1}'$ be a geodesic path from $w=r_-$ to $(p_{i+1})_+=r_+$. 
    Since $\elem{p_i'}=\elem{q} \in Q'$ and $\elem{p_{i+1}'}=\elem{r} \in R'$, the broken line $p'=p_1 \dots p_{i-1} p_i' p_{i+1}' p_{i+2} \dots p_n$ is a path representative of the same element $g \in G$.

    If at least one of the paths $q$, $r$ is not geodesic in $\ga$, then, in view of \eqref{eq:length_of_q} and \eqref{eq:length_of_r} we have
    \[
        \ell(p_i')+\ell(p_{i+1}') < \ell(q)+\ell(r) \le \ell(p_i)+\ell(p_{i+1}),
    \]
    hence $\ell(p)=\sum_{i=1}^n \ell(p_i)>\ell(p')$, contradicting the minimality of the type of $p$.

    Hence both $q$ and $r$ must be geodesic in $\ga$ , so we can further assume that $p_i'=q$ and $p_{i+1}'=r$. 
    Moreover, the inequality $\ell(p) \le \ell(p')$ must hold by the minimality of the type of $p$. 
    Therefore $\ell(p_i)+\ell(p_{i+1}) \le \ell(q)+\ell(r)$, which, in view of \eqref{eq:length_of_q} and \eqref{eq:length_of_r}, implies that $\ell(q)=\ell(p_i)$, $\ell(r)=\ell(p_{i+1})$ and $\ell(p)=\ell(p')$. 
    In particular, $e$ and $f$ are actual edges of $\ga$ (and not trivial paths).

    The definition \eqref{eq:def_of_q} of $q$ implies that $\Lab(q)$ can differ from $\Lab(p_i)$ in at most one letter, which is the label of the $H_\nu$-component $e$ in $\Lab(q)$ and the label of the $H_\nu$-component $s$ in $\Lab(p_i)$. 
    Indeed, 
    \[
        \Lab(p_i)=\Lab([(p_i)_-,s_-]) \Lab(s) \Lab(\alpha) \text{ and } \Lab(q)=\Lab([(p_i)_-,s_-]) \Lab(e) \Lab(\alpha),
    \] 
    where we used the fact that $[wa,w]$ is the left translate of $\alpha$, by definition, and hence it has the same label as $\alpha$.

    Similarly, (\ref{eq:def_of_r}) implies $\Lab(r)$ can differ from $\Lab(p_i)$ in at most one letter which is the label of $f$ in $r$ and the label of $t$ in $p_{i+1}$. 
    The minimality of the type of $p$ therefore implies that
    \begin{equation}
    \label{eq:sum_of_lengths_of_s_and_t}
        |s|_X+|t|_X \le |e|_X+|f|_X.
    \end{equation}

    Now, using the triangle inequality, \eqref{eq:d_X(s_-,w)} and \eqref{eq:mod_a_and_b} we obtain
    \begin{equation}
    \label{eq:mod_e}
        |e|_X=d_X(s_-,wa) \le d_X(s_-,w)+d_X(w,wa) \le k+|a|_X \le k+\xi .
    \end{equation}

    To estimate $|f|_X$ we also use the inequality $d_X(s_-,t_+) \le \zeta$:
    \begin{equation}
    \label{eq:mod_f}
    \begin{split}
        |f|_X=d_X(t_+,wb) & \le d_X(t_+,w)+|b|_X \\
            & \le d_X(t_+,s_-)+d_X(s_-,w)+\xi+L \le \zeta+ k+\xi+L .
    \end{split}
    \end{equation}

 Combining \eqref{eq:sum_of_lengths_of_s_and_t}--\eqref{eq:mod_f} together, we achieve
\[\max\{|s|_X,|t|_X\}\le |e|_X+|f|_X \le 2k + 2 \xi + \zeta + L=\tau.\]
    This inequality completes the proof of the lemma. 
\end{proof}

The following auxiliary definition will only be used in the remainder of this section.

\begin{definition}
\label{def:tau_j}
    Let $C_0 \ge 0$ be the constant provided by Lemma~\ref{lem:bddinnprod}, let $L \ge 0$ be the constant given by Lemma~\ref{lem:isol_comp_in_triangles_are_short} and let $\kappa=\kappa(1,0,L) \ge 0$ be the constant from Proposition~\ref{prop:osinbcp}.

    Define the sequences $(\zeta_j)_{j \in \NN}$, $(\xi_j)_{j \in \NN}$ and $(\tau_j)_{j \in \NN}$ of non-negative real numbers as follows.

    Set $\zeta_1=\kappa$, $\xi_1=C_0+1$ and $\tau_1=\max\{\kappa, \tau(\zeta_1,\xi_1)\}$, where $\tau(\zeta_1,\xi_1)$ is given by Lemma~\ref{lem:one_comp_in_cusp_is_bounded}.

    Now suppose that $j>1$ and the first $j-1$ members of the three sequences have already been defined. Then we set
    \[
        \zeta_j=\kappa,~\xi_j=C_0+1+\sum_{k=1}^{j-1} \tau_k \text{ and } \tau_j=\max\{\kappa,\tau(\zeta_j,\xi_j)\},
    \]
     where $\tau(\zeta_j,\xi_j)$ is given by Lemma~\ref{lem:one_comp_in_cusp_is_bounded}.
\end{definition}

\begin{lemma}
\label{lem:shortspikes}
    There exists a constant $C_1 \ge 0$ such that the following is true.

    Let $Q' \leqslant Q$ and $R' \leqslant R$ be subgroups satisfying \descref{C1} and let  $p=p_1\dots p_n$ be a minimal type path representative for an element $g \in \langle Q',R' \rangle$. 
    Suppose that, for some \(i \in \{1,\dots,n-1\}\), $q$ and $r$ are connected $\mathcal{H}$-components of $p_i$ and $p_{i+1}$ respectively. 
    Then \(d_X(q_+, (p_i)_+) \leq C_1\) and \(d_X((p_i)_+,r_-) \leq C_1\).
\end{lemma}

\begin{proof} 
    Denote $x=(p_{i})_+=(p_{i+1})_- \in G$. 
    First, let us show that
    \begin{equation}
    \label{eq:dxh(q_+,(p_i)_+}
        \dxh(q_+,x) \le C_0+1,
    \end{equation}
    where $C_0 \ge 0$ is the global constant provided by Lemma~\ref{lem:bddinnprod}. 
    Indeed, the latter lemma states that $\langle (p_i)_-,(p_{i+1})_+ \rangle_{x}^{rel} \le C_0$. 
    Since $q_+$ and $r_-$ are points on the geodesics $p_i$ and $p_{i+1}$, Remark~\ref{rem:Gr_prod_ineq} implies that
    \[
        \langle q_+,r_-\rangle_{x}^{rel} \le \langle (p_i)_-,(p_{i+1})_+ \rangle_{x}^{rel} \le C_0.
    \] 
    Consequently,
    \begin{align*}
        C_0 \ge \langle q_+,r_-\rangle_{x}^{rel} & = \frac12 \Bigl( \dxh(x,q_+)+\dxh(x,r_-)-\dxh(q_+,r_-)\Bigr) \\ 
            & \ge \frac12\Bigl( 2\dxh(x,q_+)-2\dxh(q_+,r_-)\Bigr) \ge \dxh(x,q_+) -1,
    \end{align*}
    where the last inequality used the fact that $\dxh(q_+,r_-) \le 1$, which is true because $q$ and $r$ are connected $\mathcal{H}$-components. 
    This establishes the inequality \eqref{eq:dxh(q_+,(p_i)_+}.

    Let $\alpha$ denote the subpath of $p_i$ starting at $q_+$ and ending at $x$, and let $\beta$ denote the subpath of $p_{i+1}$ starting at $x$ and ending at $r_-$. Let $s_1,\dots,s_l$, $l \in \NN_0$, be the set of all $\mathcal{H}$-components of $\alpha$ listed in the reverse order of their occurrence.
    That is, $s_1$ is the last $\mathcal{H}$-component of $\alpha$ (closest to $\alpha_+=x$) and $s_l$ is the first $\mathcal{H}$-component of $\alpha$ (closest to $\alpha_-=q_+$). 
    Note that, by \eqref{eq:dxh(q_+,(p_i)_+},
    \begin{equation} 
    \label{eq:bound_on_l}
        l \le \ell(\alpha) = d_{X\cup\mathcal{H}}(x,q_+) \le C_0+1.
    \end{equation}

    Let $L \ge 0$ be the constant given by Lemma~\ref{lem:isol_comp_in_triangles_are_short}, then
    \begin{equation}
    \label{eq:d_X(q_+,r_-)}
        d_X(\alpha_-,\beta_+) =d_X(q_+,r_-) \le L.
    \end{equation}
    It follows that the geodesic paths $\alpha$ and $\beta^{-1}$ are $L$-similar in $\ga$. 
    Let $\kappa=\kappa(1,0,L) \ge 0$ be the constant provided by Proposition~\ref{prop:osinbcp}.

    We will now prove the following.

    \begin{claim} 
        For each $j=1,\dots,l$ we have
        \begin{equation}
        \label{eq:tau_j}
            |s_{j}|_X \leq \tau_j,
        \end{equation}
        where $\tau_j \ge 0$ is given by  Definition~\ref{def:tau_j}.
    \end{claim}

    We will establish the claim by induction on $j$. For the base of induction, $j=1$, note that if  $|s_1|_X < \kappa$ then the inequality $|s_1|_X \le \tau_1$ will be true by definition of $\tau_1$.    
    Thus we can suppose that $|s_1|_X \ge \kappa$. 
    In this case, by Proposition~\ref{prop:osinbcp}, $s_1$ must be connected to some $\mathcal{H}$-component of $\beta^{-1}$. Claim (3) of the same proposition implies that there is an $\mathcal H$-component $t_1$ of $\beta$, such that $s_1$ is connected to $t_1$ and $d_X((s_1)_-,(t_1)_+) \le \kappa$. 
    Note that, by construction, $s_1$ and $t_1$ are also connected $\mathcal{H}$-components of $p_i$ and $p_{i+1}$ respectively.

    Observe that the subpath of $\alpha$ from $(s_1)_+$ to $x$ is labelled by letters from $X^{\pm 1}$ because it has no $\mathcal H$-components. 
    Therefore $d_X((s_1)_+,x) \le \ell(\alpha) \le C_0+1$.  
    Consequently, we can apply Lemma~\ref{lem:one_comp_in_cusp_is_bounded} to deduce that  $|s_1|_X\le \tau(\zeta_1,\xi_1)$, where $\zeta_1=\kappa$ and $\xi_1=C_0+1$.

    Thus we have shown that $|s_1|_X\le \tau_1$, where $\tau_1=\max\{\kappa, \tau(\zeta_1,\xi_1)\}$, and the base of induction has been established.

    Now, suppose that $j>1$ and inequality \eqref{eq:tau_j} has been proved for all strictly smaller values of $j$. 
    If $|s_j|_X< \kappa$ then are done, because $\tau_j \ge \kappa$ by definition. 
    So we can assume that $|s_j|_X \ge \kappa$. As before, we can use  Proposition~\ref{prop:osinbcp}, to find an $\mathcal{H}$-component $t_j$ of $\beta$ such that $s_j$ is connected to $t_j$ and $d_X((s_j)_-,(t_j)_+) \le \kappa$.

    By construction, $s_1, \dots, s_{j-1}$ is the list of all $\mathcal{H}$-components of the subpath $[(s_j)_+,x]$ of $\alpha$, hence
    \[
        d_X((s_j)_+,x) \le \ell(\alpha)+\sum_{k=1}^{j-1} |s_k|_X \le C_0+1+\sum_{k=1}^{j-1} \tau_k, 
    \] 
    where the second inequality used \eqref{eq:bound_on_l} and the induction hypothesis. 
    This allows us to apply Lemma~\ref{lem:one_comp_in_cusp_is_bounded} again, and conclude that $|s_j|_X \le \tau(\zeta_j,\xi_j)$, where $\zeta_j=\kappa$ and $\xi_j=C_0+1+\sum_{k=1}^{j-1} \tau_k$.

    Thus, $|s_j|_X \le \max\{\kappa,\tau(\zeta_j,\xi_j)\}=\tau_j$, as required.
    Hence the claim has been proved by induction on $j$.

    We are finally ready to prove the main statement of the lemma. Since $s_1,\dots,s_l$ is the list of all $\mathcal H$-components of $\alpha$, we can combine the inequalities \eqref{eq:bound_on_l} and \eqref{eq:tau_j} to achieve
    \[
        d_X(q_+,(p_i)_+)=|\alpha|_X \le \ell(\alpha)+\sum_{j=1}^l |s_j|_X \le C_0+1+\sum_{j=1}^l \tau_j \le C_0+1+\sum_{j=1}^{\lfloor C_0+1 \rfloor} \tau_j .
    \]

    On the other hand, by the triangle inequality and \eqref{eq:d_X(q_+,r_-)}, we have
    \[
        d_X((p_i)_+,r_-)\le L+d_X(q_+,(p_i)_+) \le  L+C_0+1+\sum_{j=1}^{\lfloor C_0+1 \rfloor} \tau_j. 
    \]

    We have shown that the constant $\displaystyle C_1=L+C_0+1+\sum_{j=1}^{\lfloor C_0+1 \rfloor} \tau_j > 0$ is an upper bound for $d_X(q_+,(p_i)_+)$ and $d_X((p_i)_+,r_-)$, thus the lemma is proved.
\end{proof}

\begin{definition}[Consecutive, adjacent and multiple backtracking]
    Let \(p=p_1 \dots p_n\) be a broken line in \(\Gamma(G,X\cup\mathcal{H})\).
    Suppose that for some $i,j$, with $1 \le i <j \le n$, and $\nu \in \Nu$ there exist pairwise connected $H_\nu$-components $h_i,h_{i+1},\dots, h_j$ of the paths $p_i,p_{i+1}, \dots, p_j$, respectively.
    Then we will say that $p$ has \emph{consecutive backtracking} along the components $h_i,\dots,h_j$ of \(p_i, \dots, p_j\).
    Moreover, if \(j=i+ 1\), we will call it an instance of \emph{adjacent backtracking}, while if \(j>i+1\) will use the term \emph{multiple backtracking}.
\end{definition}

The next lemma shows that, among path representatives of minimal type, instances of adjacent backtracking where at least one of the components is sufficiently long with respect to the proper metric \(d_X\) must have initial and terminal vertices far apart in $d_X$.

\begin{lemma}[Adjacent backtracking is long]
\label{lem:longadjbacktracking}
    For any \(\zeta \geq 0\) there is \(\Theta_0 = \Theta_0(\zeta) \in \NN\) such that the following holds.

    Let $Q' \leqslant Q$ and $R' \leqslant R$ be subgroups satisfying \descref{C1} and let \(p=p_1 \dots p_n\) be a minimal type path representative for an element \(g \in \langle Q', R' \rangle\). Suppose that for some $i \in \{1,\dots,n-1\}$ the paths \(p_i\) and \(p_{i+1}\) have connected \(\mathcal H\)-components \(q\) and \(r\) respectively, satisfying
    \[
        \max\{ |q|_X,|r|_X\} \geq \Theta_0.
    \]
    Then \(d_X(q_-,r_+) \geq \zeta\).
\end{lemma}

\begin{proof} 
    For any $\zeta \ge 0$ we can define $\Theta_0=\lfloor\tau(\zeta,C_1)\rfloor+1$, where $C_1$ is the constant from Lemma~\ref{lem:shortspikes} and $\tau(\zeta,C_1)$ is provided by Lemma~\ref{lem:one_comp_in_cusp_is_bounded}.

    It follows that if \(d_X(q_-,r_+) < \zeta\) then $|q|_X < \Theta_0$ and $|r|_X < \Theta_0$, which is the contrapositive of the required statement.
\end{proof}


\section{Multiple backtracking in path representatives of minimal type}
\label{sec:multitracking}
As before, we keep working under Convention~\ref{conv:main}.
In this section we deal with multiple backtracking in path representatives of elements from $\langle Q',R'\rangle$. Proposition~\ref{prop:multitracking_path} below uses condition \descref{C3} to show that any instance of multiple backtracking essentially takes place inside a parabolic subgroup. In order to achieve this we first prove two auxiliary statements.

\begin{notation}
    Throughout this section \(C_1 \geq 0\) will be the constant given by Lemma~\ref{lem:shortspikes} and $\mathcal{P}_1$ will denote the finite collection of parabolic subgroups of $G$ defined by
    \[
        \mathcal{P}_1=\lbrace t H_\nu t^{-1} \mid \nu \in \Nu, \abs{t}_X \leq C_1 \rbrace.
    \]
Consider the subset $O=\{o \in PS \mid P \in \mathcal{P}_1,~ \abs{o}_X \le 2C_1\}$ of $G$. Since $|O|<\infty$, we can choose and fix a finite subset $\Omega \subseteq S$ such that every element $o \in O$ can be written as $o=fh$, where $f \in P$, for some $P \in \mathcal{P}_1$, and $h \in \Omega$. We define a constant $E$ by  
\begin{equation}\label{eq:def_of_E}
E=\max\{\abs{h}_X \mid h \in \Omega\} \ge 0.    
\end{equation}
    
\end{notation}

\begin{lemma}
\label{lem:end_sides_constr} 
    There exists a constant $ D\ge 0$ such that the following holds.

    Let $\nu \in \Nu$ and $b \in G$ be an element with $|b|_X \le C_1$, so that $P=bH_\nu b^{-1} \in\mathcal{P}_1$,  and let $p$ be a geodesic path in $\ga$ with $\elem{p} \in Q \cup R$.

    Suppose that there is a vertex $v$ of $p$ and an element $u \in P$ such that $v \in Pb=bH_\nu$ and $u^{-1}p_- \in S=Q \cap R$. Then there exists a geodesic path $p'$ in $\ga$ such that
    \begin{itemize}
        \item $p'_-=u$ and $d_X(p'_+,v) \le D$;
        \item if $\elem{p} \in Q$ then $\elem{p'} \in Q \cap P$, otherwise $\elem{p'} \in R \cap P$.
    \end{itemize}
\end{lemma}

\begin{proof}
    Let $K=\max\{C_1, \varepsilon\} \ge 0$, where $\varepsilon$ is the quasiconvexity constant of $Q$ and $R$, and let
    \begin{equation}
    \label{eq:def_of_D}
        D= \max\{K'(Q,P,K),K'(R,P,K) \mid P \in \mathcal{P}_1\},
    \end{equation}
    where $K'(Q,P,K)$ and $K'(R,P,K)$ are  obtained from Lemma~\ref{lem:nbhdintersection}.

    Denote $x =p_- \in G$ and assume, without loss of generality, that $\elem{p} \in Q$ (the case $\elem{p} \in R$ can be treated similarly). By the quasiconvexity of \(Q\), we have that \( d_X(v,x Q) \leq \varepsilon\). Moreover, $xQ=uQ$ as $u^{-1}x \in S \subseteq Q$.

    By the assumptions, $vb^{-1} \in P$, hence $d_X(v,P) \le |b|_X \le C_1$. Since $uP=P$ we see that
    \begin{equation*}
        v \in N_X(uQ,\varepsilon) \cap N_{X}( uP,C_1).
    \end{equation*}
    Applying Lemma \ref{lem:nbhdintersection}, we find \(w \in u(Q \cap P)\) such that \(d_X( v,w) \leq D\) (see Figure~\ref{fig:multitracking1}).

    \begin{figure}[ht]
        \centering
        \includegraphics{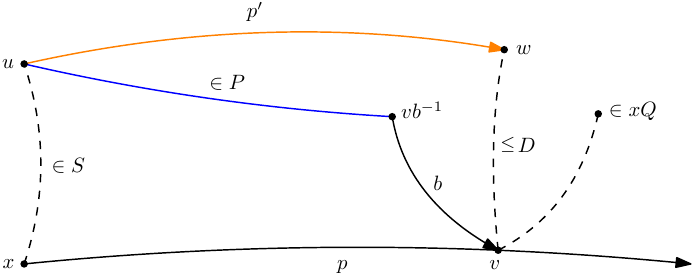}
        \caption{Illustration of Lemma \ref{lem:end_sides_constr}.}
        \label{fig:multitracking1}
    \end{figure}

    Let \(p'\) be any geodesic in \(\Gamma(G,X\cup\mathcal{H})\) starting at $u$ and ending at $w$. It is easy to see that $p'$ satisfies all of the required properties, so the lemma is proved.
\end{proof}

The next lemma describes how condition \descref{C3} is used in this paper.

\begin{lemma}
\label{lem:(c3)->vertex_constr} 
    Assume that subgroups $Q' \leqslant Q$ and $R' \leqslant R$ satisfy conditions \descref{C1} and \descref{C3} with constant \(C\) and family \(\mathcal{P}  \) such that $C \ge 2C_1 + 1$ and $\mathcal{P}_1 \subseteq \mathcal{P}$.
    Let $P=bH_\nu b^{-1} \in\mathcal{P}_1$, for some $\nu \in \Nu$ and $b \in G$, with $|b|_X \le C_1$, and let $p$ be a path in $\ga$ with $\elem{p} \in Q' \cup R'$.

    Suppose that there is a vertex $v$ of $p$ and an element $u \in P$ satisfying $u^{-1}p_- \in S$, $v \in Pb$, and $d_X(v,p_+) \le C_1$. Then there exists a geodesic path $p'$ such that $(p')_-=u$, $\elem{p'} \in P$, $(p')_+^{-1}p_+ \in S$ and $d_X((p')_+,p_+) \le E$, where $E$ is the constant from \eqref{eq:def_of_E}. In particular, if $\elem{p} \in Q'$ (respectively, $\elem{p} \in R'$) then $\elem{p'} \in Q' \cap P$ (respectively, $\elem{p'} \in R' \cap P$).
\end{lemma}

\begin{proof} 
    Denote $x=p_-$,  $y=p_+$ and $z=vb^{-1} \in P$ (see Figure~\ref{fig:multitracking2}). Then $u^{-1}z \in P$ and $x^{-1}y=\elem{p} \in Q' \cup R'$.

    \begin{figure}[ht]
        \centering
        \includegraphics{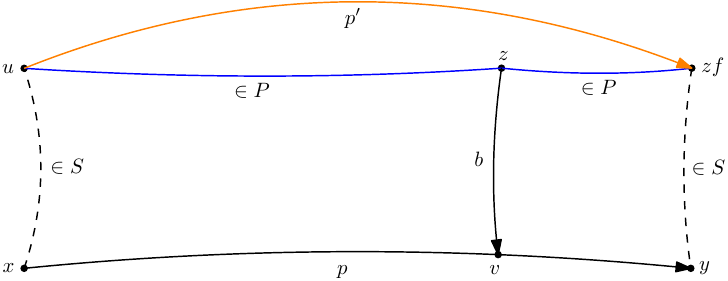}
        \caption{Illustration of Lemma \ref{lem:(c3)->vertex_constr}.}
        \label{fig:multitracking2}
    \end{figure}

    Since $u^{-1} x \in S = Q' \cap R'$, we obtain 
    \[
        u^{-1}y=(u^{-1}x) (x^{-1}y) \in Q' \cup R',
    \]
    whence $z^{-1}y=(z^{-1}u )(u^{-1}y) \in P (Q' \cup R')$. 
    Now, observe that
    \[
        |z^{-1} y|_X=d_X(z,y) \le d_X(z,v)+d_X(v,y) \le |b|_X+C_1 \le 2 C_1 < C.
    \] 
    Condition \descref{C3} now implies that $z^{-1}y \in PS $. 
    That is, $z^{-1}y =fh$ for some $f \in P$ and $h \in \Omega$, where $\Omega$ is the finite subset of $S$ defined above the statement of the lemma.
    Let $p'$ be a geodesic path starting at $u$ and ending at $zf \in P$. Then $\elem{p'}=u^{-1}zf \in P$,
    \[
        (p')_+^{-1}p_+=f^{-1}z^{-1}y=h \in S ~\text{ and }~ d_X((p')_+,p_+)=\abs{h}_X \le E.
    \]

    The last statement of the lemma follows from \descref{C1} and the observation that 
    \[
        \elem{p'}=u^{-1}(p')_+=u^{-1} p_- \, \elem{p} \, (p_+)^{-1}(p')_+ \in S\,\elem{p}\, S. \qedhere
    \]
\end{proof}

\begin{proposition}
\label{prop:multitracking_path}
    Let $D \ge 0$ is the constant provided by Lemma~\ref{lem:end_sides_constr}, and let $E$ be given by \eqref{eq:def_of_E}. Suppose that $Q' \leqslant Q$ and $R' \leqslant R$ are subgroups satisfying \descref{C1} and \descref{C3}, with constant \(C \ge 2C_1 + 1\) and family \(\mathcal{P} \supseteq \mathcal{P}_1\).

    Let \(p = p_1 \dots p_n\) be a path representative for an element \(g \in \langle Q', R' \rangle\) with minimal type.
    If \(p\) has consecutive backtracking along \(\mathcal{H}\)-components \(h_i, \dots, h_j\) of the subpaths \(p_i, \dots, p_j\) respectively, then there is a subgroup \(P \in \mathcal{P}_1\) and a path \(p' = p'_i \dots p'_j\) satisfying the following properties:
    \begin{itemize}
        \item[(i)] \(p'_k\) is geodesic with \(\elem{p'_k} \in P\) for all \(k = i, \dots, j\);
        \item[(ii)] \((p'_i)_+ = (p_i)_+\), \((p'_k)_+^{-1} (p_k)_+ \in S\) and  $d_X((p'_k)_+,(p_k)_+) \le E$, for all $k=i+1,\dots,j-1$;
        \item[(iii)] \(d_X(p'_-,(h_i)_-) \leq D\) and \(d_X(p'_+,(h_j)_+) \leq D\);  
        \item[(iv)] \(\elem{p'_i} \in Q\cap P\) if \(\elem{p_i} \in Q'\), and  \(\elem{p'_i} \in R \cap P\) if \(\elem{p_i} \in R'\); similarly, \(\elem{p'_j} \in Q \cap P\) if \(\elem{p_j} \in Q'\),  and  \(\elem{p'_j} \in R \cap P\) if \(\elem{p_j} \in R'\);
        \item[(v)] for each \(k \in \{i+1, \dots, j-1\}\), \(\Lab(p'_k)\) either represents an element of \(Q' \cap P\) or an element of \(R' \cap P\).
    \end{itemize}
\end{proposition}

\begin{proof} 
    Figure \ref{fig:multitracking_parabolic_path} below is a sketch of the path \(p'\) above the subpath \(p_i p_{i+1} \dots p_{j-1} p_j\) of \(p\). 

    \begin{figure}[ht]
        \centering
        \includegraphics{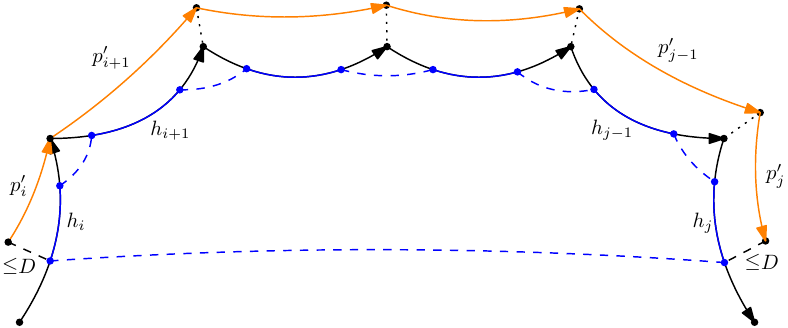}
        \caption{The new path \(p'\) constructed in Proposition~\ref{prop:multitracking_path}. The dotted lines between \(p\) and \(p'\) are paths whose labels represent elements of \(S\).}
        \label{fig:multitracking_parabolic_path}
    \end{figure}

    Note that claim (v) follows from claim (ii) and condition \descref{C1}, so we only need to establish claims (i)--(iv).

    By the assumptions, there is $\nu \in \mathcal{N}$ such that for each \(k \in \{i, \dots, j\}\), the path $p_k$ is a concatenation \(p_k = a_k h_k b_k\), where \(h_k\) is an \(H_\nu\)-component of \(p_k\) and \(a_k,b_k\) are subpaths of \(p_k\).

    According to Lemma~\ref{lem:shortspikes}, we have
    \begin{equation}\label{eq:lengths_of_a_k_and_b_l}
    \abs{b_k}_X \le C_1, \text{ for } k=i,\dots,j-1.
    \end{equation}

    After translating everything by $(p_i)_+^{-1}$ we can assume that $(p_i)_+=1$.
    From here on, we let $b=\elem{b_i}^{-1} \in G$ and \(P= b H_\nu b^{-1}\).
    As noted in \eqref{eq:lengths_of_a_k_and_b_l}, \(\abs{b}_X =\abs{{b_i}}_X \leq C_1\), so \(P \in \mathcal{P}_1\).

    Since the components \(h_i\) and \(h_k\) are connected, for every $k=i+1,\dots,j$, the elements $(h_i)_+=(b_i)_-=b$ and $(h_k)_+$ all belong to the same left coset $bH_\nu=Pb$, thus
    \begin{equation}\label{eq:in_Pb}
        (h_k)_+ \in Pb, \text{ for all } k=i+1,\dots,j.
    \end{equation}

    The rest of the argument will be divided into three steps.

    \medskip
    \noindent \underline{\emph{Step 1:}} construction of the path $p_i'$.

        Set $u_i=(p_i)_+=1$ and $v_i=(h_i)_-$. Then $v_i =\elem{b_i}^{-1} \elem{h_i}^{-1} \in bH_\nu=Pb$, so the path $p_i^{-1}$, its vertex $v_i$ and the element $u_i=1 \in P$ satisfy the assumptions of Lemma~\ref{lem:end_sides_constr}. Therefore there exists a path $q$ with $q_-=u_i$, $d_X(q_+,v) \le D$ and such that $\elem{q}  \in Q \cap P$ if $\elem{p_i} \in Q$ and $\elem{q}  \in R \cap P$ if $\elem{p_i} \in R$.

        It is easy to check that the path $p_i'=q^{-1}$ enjoys the required properties.

    \medskip
    \noindent \underline{\emph{Step 2:}} construction of the paths $p_k'$, for $k=i+1,\dots,j-1$.

        We will define the paths $p_k'$ by induction on $k$. For $k=i+1$ we consider the path $p_{i+1}$, its vertex $v_{i+1}=(h_{i+1})_+$ and the element $u_i=1=(p_{i+1})_-$.
        Since $v_{i+1} \in Pb$ by \eqref{eq:in_Pb} and $d_X(v_{i+1}, (p_{i+1})_+)=|b_{i+1}|_X \le C_1$ by \eqref{eq:lengths_of_a_k_and_b_l},  we can apply Lemma~\ref{lem:(c3)->vertex_constr} to find a geodesic path $p_{i+1}'$ starting at $u_i$ and satisfying the required conditions.

        Now suppose that the required paths $p_{i+1}',\dots,p_m'$ have already been constructed for some $m \in \{i+1,\dots,j-2\}$. To construct the path $p_{m+1}'$, let $v_{m+1}$ be the vertex $(h_{m+1})_+$ of $p_{m+1}$ and set $u_m=(p_m')_+$. Then $u_m \in P$ and $u_m^{-1} (p_{m+1})_-=(p_m')_+^{-1} (p_{m})_+ \in S$ by the induction hypothesis. In view of \eqref{eq:in_Pb} and \eqref{eq:lengths_of_a_k_and_b_l}, $v_{m+1} \in Pb$ and $d_X(v_{m+1}, (p_{m+1})_+) \le C_1$, therefore we can find a geodesic path $p_{m+1}'$ with the desired properties by using Lemma~\ref{lem:(c3)->vertex_constr}.

        Thus we have described an inductive procedure for constructing the paths $p_k'$, for $k=i+1,\dots,j-1$.

    \medskip
    \noindent \underline{\emph{Step 3:}} construction of the path $p_j'$.

        This step is similar to Step 1: the path $p_j'$ will start at $u_{j-1}=(p_{j-1}')_+ \in P$ and can be constructed by applying Lemma~\ref{lem:end_sides_constr} to the path $p_j$ and the elements $v_j=(h_j)_+ \in P b $, $u_{j-1} \in P$.

    \medskip

    We have thus constructed a sequence of geodesic paths \(p'_i, \dots, p'_j\) whose concatenation \(p'\) satisfies all the properties from the proposition.
\end{proof}

We will now prove the main result of this section, which states that the initial and terminal vertices of an instance of multiple backtracking in a minimal type path representative must lie far apart in the proper metric $d_X$, provided \(Q' \leqslant Q\) and \(R' \leqslant R\) satisfy \descref{C1}--\descref{C5} with sufficiently large constants.

\begin{proposition}[Multiple backtracking is long]
\label{prop:long_multitracking}
    For any \(\zeta \geq 0\) there is a constant \(C_2 = C_2(\zeta) \geq 0\) such that if \(Q' \leqslant Q\) and \(R' \leqslant R\) are subgroups satisfying conditions \descref{C1}-\descref{C5} with constants \(B \ge C_2\) and \(C \ge C_2\) and a family \(\mathcal{P} \supseteq \mathcal{P}_1\), then the following is true.

    Let \(p = p_1 \dots p_n\) be a minimal type path representative for an element \(g \in \langle Q', R' \rangle\).
    If \(p\) has multiple backtracking along \(\mathcal{H}\)-components \(h_i, \dots, h_j\) of \(p_i, \dots, p_j\)
    then \(d_X((h_i)_-,(h_j)_+) \geq \zeta\).
\end{proposition}

\begin{proof}
    Let \(\zeta \geq 0\) and define \(C_2(\zeta) = \max{\{2C_1, \zeta + 2D\}} + 1\), where \(D \geq 0\) is the constant obtained from Lemma~\ref{lem:end_sides_constr}.

    In view of the assumptions we can apply Proposition~\ref{prop:multitracking_path} to find a path \(p' = p'_i \dots p'_j\) and \(P \in \mathcal{P}_1\) satisfying properties (i)--(v) from its statement.
    Let \(\alpha\) be a geodesic with \(\alpha_- = (p_j')_-\) and \(\alpha_+ = (p_j)_-\), and let $\beta=p'_{i+1} \dots p'_{j-1}$.
    We will denote \(x_k = \elem{p_k}\) and \(x'_k = \elem{p'_k}\), for each \(k \in \{i, \dots, j\}\), and \(z = \elem{\alpha}\).
    Condition \descref{C1}, together with claim (ii) of Proposition~\ref{prop:multitracking_path}, tell us that \(z \in S = Q' \cap R'\), and claim (v) yields that
    \begin{equation}
    \label{eq:elem(beta)_in_Q'R'}
        \elem{\beta} = x_{i+1}' \dots x_{j-1}' \in \langle Q'_P,R'_P \rangle
    \end{equation}
    (as before, for a subgroup $H \leqslant G$ we denote by $H_P \leqslant G$ the intersection $H \cap P$).

    Now suppose, for a contradiction, that \(d_X((h_i)_-,(h_j)_+) < \zeta\).
    Then
    \begin{equation}
    \label{eq:d_X(p'_-, p'_+)}
    |p'|_X=d_X(p'_-, p'_+) < \zeta + 2D < C_2 \le \min\{B,C\},
    \end{equation} by claim (iii) of Proposition~\ref{prop:multitracking_path}.
    There are four cases to consider depending on whether \(\elem{p_i}\) and \(\elem{p_j}\) are elements of \(Q'\) or \(R'\).

    \medskip
    \underline{\emph{Case 1:}} $x_i=\elem{p_i} \in Q'$ and $x_j=\elem{p_j} \in Q'$.
        Then, by claim (iv) of Proposition~\ref{prop:multitracking_path}, both \(x'_i\) and \(x'_j\) are elements of \(Q_P\) .
        It follows that \( \elem{p'} \in Q_P \langle Q'_P, R'_P \rangle Q_P \subseteq Q \langle Q', R' \rangle Q\).
        By \eqref{eq:d_X(p'_-, p'_+)} and \descref{C2}, there is \(q \in Q\) such that \(\elem{p'} = q\). Therefore
        \begin{equation}
        \label{eq:elem(beta)_in_Q}
            \elem{\beta}={x_i'}^{-1}\, \elem{p'} \,  {x_j'}^{-1}= {x_i'}^{-1} \, q \,  {x_j'}^{-1} \in Q.
        \end{equation}

        Combining \eqref{eq:elem(beta)_in_Q} with \eqref{eq:elem(beta)_in_Q'R'}  and using condition \descref{C4}, we get
        \[
            \elem{\beta} \in Q \cap \langle Q'_P,R'_P \rangle=Q_P \cap \langle Q'_P,R'_P \rangle=Q'_P.
        \]

        Let $\gamma$ be any geodesic path in $\ga$ starting at $(p_i)_-$ and ending at $(p_j)_+$. 
        Then $\gamma$ shares the same endpoints with the path $p_i \beta  \alpha  p_j$, therefore their labels represent the same element of $G$:
        \[
            \elem{\gamma}=x_i \, \elem{\beta} \, z \, x_j \in Q'\, Q_P'\,  S\, Q'=Q'. 
        \]
        Thus we can use $\gamma$ to obtain another path representative for \(g\) through $p_1 \dots p_{i-1} \gamma p_{j+1} \dots p_n$,
        which consists of strictly fewer geodesic subpaths than \(p=p_1 \dots p_n\).
        This contradicts the minimality of the type of \(p\), so Case 1 has been considered.

    \medskip
    \underline{\emph{Case 2:}} both $\elem{p_i}$ and $\elem{p_j}$ are elements of $R'$.
        This case can be dealt with identically to Case 1.

    \medskip
    \underline{\emph{Case 3:}} $x_i=\elem{p_i} \in Q'$ and $x_j=\elem{p_j} \in R'$.
        Then \(x'_i \in Q_P\) and \(x'_j \in R_P\) by claim (iv) of Proposition~\ref{prop:multitracking_path}.
        Hence \(\Lab(p')\) represents an element of \(x'_i \langle Q'_P, R'_P \rangle R_P\) with \(x'_i \in Q_P\).
        In view of \eqref{eq:d_X(p'_-, p'_+)}, we can use condition \descref{C5} to deduce that \(\elem{p'} \in x'_i Q'_P R_P\).
        It follows that
        \[
            \elem{\beta}=(x_i')^{-1} \, \elem{p'} \, (x_j')^{-1} \in Q_P'\, R_P,
        \]
        so there exist $q \in Q'_P$ and $r \in R_P$ such that $\elem{\beta}=qr$.
        Combining this with \eqref{eq:elem(beta)_in_Q'R'} we conclude that $r=q^{-1}\elem{\beta} \in R_P \cap \langle Q_P',R_P' \rangle$, so $r \in R_P'$ by condition \descref{C4}, whence
        \begin{equation}
        \label{eq:elem(beta)_in_dc}
            \elem{\beta}=qr \in Q'_P\, R_P'.
        \end{equation}

        Observe that the paths $\gamma=p_i \dots p_j$ and $p_i \beta \alpha p_j$ have the same endpoints, hence their labels represent the same element of $G$:
        \[
            \elem{\gamma}=x_i \elem{\beta} z x_j  \in Q' \, Q'_P\, R_P' \, S \, R' \subseteq Q'\, R'.
        \]
        Therefore there are elements $q_1 \in Q'$ and $r_1 \in R'$ such that $\elem{\gamma}=q_1 r_1$.

        Let $\gamma_1$ be a geodesic path in $\ga$ starting at $\gamma_-=(p_i)_-$ and ending at $\gamma_- q_1$ and let $\gamma_2$ be a geodesic path starting at $(\gamma_1)_+$ and ending at $(\gamma_1)_+ r_1=\gamma_+=(p_j)_+$. 
        Since $\elem{\gamma_1} =q_1\in Q'$ and $\elem{\gamma_2}=r_1 \in R'$ the path \(p_1 \dots p_{i-1} \gamma_1 \gamma_2 p_{j+1} \dots p_n\) is a path representative of \(g\). 
        Moreover, it consists of fewer than \(n\) geodesic segments because $j>i+1$ (by the definition of multiple backtracking), contradicting the minimality of the type of \(p\). 
        This contradiction shows that Case~3 is impossible.

    \medskip
    \underline{\emph{Case 4:}} $x_i=\elem{p_i} \in R'$ and $x_j=\elem{p_j} \in Q'$. 
        Then \(x'_i \in R_P\) while \(x'_j \in Q_P\), which implies that $\elem{p'} \in R_P \langle Q'_P,R'_P \rangle x_j'$, hence $\elem{p'}^{-1} \in (x_j')^{-1}\langle Q'_P, R'_P \rangle R_P$.

        By \eqref{eq:d_X(p'_-, p'_+)}, we can use \descref{C5} to conclude that \(\elem{p'}^{-1} \in (x'_j)^{-1} Q'_P R_P\), thus $\elem{p'} \in R_P Q'_P x_j'$.
        The rest of the argument proceeds similarly to the previous case, leading to a contradiction with the minimality of the type of $p$. Hence Case~4 is also impossible.

    \medskip
    We have arrived at a contradiction in each of the four cases, so  \(d_X((h_i)_-,(h_j)_+) \ge \zeta\), as required.
\end{proof}


\section{Constructing quasigeodesics from broken lines}
\label{sec:quasigeods}
In this section we detail a procedure that takes as input a broken line and a natural number, and outputs another broken line together with some additional vertex data.
We show that if a broken line satisfies certain metric conditions, then the new path constructed through this procedure is uniformly quasigeodesic.

We assume that $G$ is a group generated by a finite set $X$ and hyperbolic relative to a finite family of subgroups $\{H_\nu \mid \nu \in \Nu\}$. As usual we set $\mathcal{H}=\bigsqcup_{\nu \in \Nu} (H_\nu\setminus\{1\})$, and by Lemma~\ref{lem:Cayley_graph-hyperbolic} we know that the Cayley graph $\ga$ is $\delta$-hyperbolic, for some $\delta \ge 0$.

The outline of the construction is as follows: we begin with a broken line \(p = p_1 \dots p_n\) in \(\ga\).
Starting from the initial vertex \(p_-\), we note in sequence (along the vertices of \(p\)) the vertices marking the start and end of maximal instances of consecutive backtracking in \(p\) involving sufficiently long \(\mathcal{H}\)-components.
Once we have done this, we construct the new path by connecting (in the same sequence) the marked vertices with geodesics.

\begin{procedure}[$\Theta$-shortcutting]
\label{proc:shortcutting}
    Fix a natural number \(\Theta \in \NN\) and let \(p = p_1 \dots p_n\) be a broken line in \(\Gamma(G,X\cup\mathcal{H})\). Let \(v_0, \dots, v_d\) be the enumeration of all vertices of \(p\) in the order they occur along the path (possibly with repetition), so that \(v_0 = p_-\), \(v_d = p_+\) and $d=\ell(p)$.

    We construct a broken line \(\Sigma(p,\Theta)\), called the \(\Theta\)-\emph{shortcutting} of \(p\), which comes with a finite set  \(V(p,\Theta) \subset \{0,\dots,d\} \times \{0,\dots,d\}\) corresponding to indices of vertices of \(p\) that we shortcut along.

    In the algorithm below we will refer to numbers \(s,t,N \in \{0,\dots,d\}\) and a subset \(V \subseteq \{0,\dots,d\} \times \{0,\dots,d\}\). To avoid excessive indexing these will change value throughout the procedure.
    The parameters $s$ and $t$ will indicate the starting and terminal vertices of subpaths of $p$ in which all $\mathcal H$-components have lengths less than $\Theta$. The parameter $N$ will keep track of how far along the path $p$ we have proceeded. The set $V$ will collect all pairs of indices $(s,t)$ obtained during the procedure.
    We initially take \(s = 0\), $N=0$ and \(V = \emptyset\).
    
    \begin{steps}
        \item
            If there are no edges of \(p\) between \(v_N\) and \(v_d\) that are labelled by elements of \(\mathcal{H}\), then add the pair $(s,d)$ to the set $V$ and skip ahead to Step 4.
            Otherwise, continue to Step 2.
        \item
            Let \(t \in \{0,\dots,d\}\) be the least natural number with \(t \geq N\) for which the edge of \(p\) with endpoints \(v_t\) and \(v_{t+1}\) is an \(\mathcal{H}\)-component $h_i$ of a geodesic segment $p_i$ of \(p\), for some \(i \in \{1, \dots, n\}\).

            If $i=n$ or if $h_i$ is not connected to a component of $p_{i+1}$ then set $j=i$. Otherwise, let \(j \in \{i+1,\dots,n\}\) be the maximal integer such that \(p\) has consecutive backtracking along \(\mathcal{H}\)-components \(h_i, \dots, h_j\) of segments \(p_i, \dots, p_j\).             
            Proceed to Step 3.

        \item 
            If \[\max\Big\{\abs{h_k}_X \, \Big| \, k = i, \dots, j\Big\} \geq \Theta,\] then add the pair \((s,t)\) to the set \(V\) and redefine $s=N$   in \(\{1,\dots,d\}\) to be the index of the vertex $(h_j)_+$ in the above enumeration $v_0,\dots,v_d$ of the vertices of $p$.
            Otherwise let $N$ be the index of $(h_i)_+$, and leave $s$ and \(V\) unchanged.
            
            Return to Step~1 with the new values of \(s\), $N$ and \(V\).
        \item
            Set \(V(p,\Theta) = V\). The above constructions gives a natural ordering of $V(p,\Theta)$: \[V(p,\Theta) = \{(s_0, t_0), \dots, (s_m,t_m)\},\] where \(s_k  \le t_k <  s_{k+1}\), for all \(k = 0, \dots, m-1\). Note that $s_0=0$ and $t_m=d$. Proceed to Step 5.
        \item
            For each $k=0,\dots,m$, let $f_k$ be a geodesic segment (possibly trivial) connecting $v_{s_k}$ with $v_{t_k}$. Note that when $k <m$,
            \(v_{t_k}\) and \(v_{s_{k+1}}\) are in the same left coset of $H_\nu$, for some $\nu \in \Nu$. 
            If \(v_{t_k} = v_{s_{k+1}}\) then let \(e_{k+1}\) be the trivial path at \(v_{t_k}\), otherwise let \(e_{k+1}\) be an edge of $\ga$ starting at \(v_{t_k}\), ending at \(v_{s_{k+1}}\) and  labelled by an element of \(H_\nu \setminus\{1\}\).

            We define the broken line \(\Sigma(p,\Theta)\) to be the concatenation \(f_0 e_1 f_1 e_2 \dots f_{m-1} e_m f_m\).
    \end{steps}
\end{procedure}

\begin{remark} 
\label{rem:shortcutting}
    Let us collect some observations about Procedure~\ref{proc:shortcutting}.
    \begin{itemize}
        \item[(a)] Since \(p\) has only finitely many vertices and $N$ increases at each iteration of Step 3 above, the procedure will always terminate after finitely many steps.
        \item[(b)] The newly constructed broken line \(\Sigma(p,\Theta)\) has the same endpoints as $p$, and each node of \(\Sigma(p,\Theta)\) is a vertex of $p$.
        \item[(c)] By construction, for any $k \in \{0,\dots,m\}$ the subpath of $p$ between $v_{s_k}$ and $v_{t_k}$ contains no edge labelled by an element $h \in \mathcal{H}$ satisfying $\abs{h}_X \ge \Theta$.
    \end{itemize}
\end{remark}

Figure~\ref{fig:shortcutting_example} below sketches an example of the output of Procedure~\ref{proc:shortcutting}.
\begin{figure}[hb]
    \centering
    \includegraphics[]{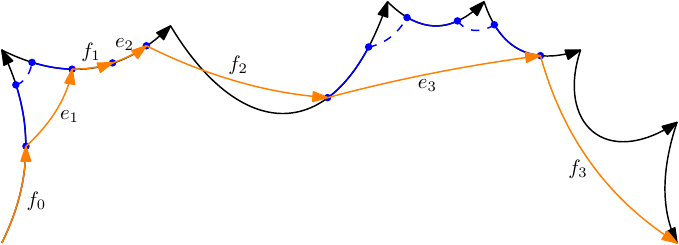}
    \caption{An example of a shortcutting of a path \(p\) in $\ga$. The path \(p\) contains long \(\mathcal{H}\)-components, some of which are involved in instances of consecutive backtracking, as indicated by the dashed lines. The path \(\Sigma(p,\Theta)=f_0e_1f_1e_2f_2e_3f_3\) is drawn on top of \(p\).}
    \label{fig:shortcutting_example}
\end{figure}

In the next definition we describe paths that will serve as input for the above procedure.

\begin{definition}[Tamable broken line]
\label{def:tamable}
    Let \(p = p_1 \dots p_n\) be a broken line in $\ga$, and let \(B, C, \zeta \geq 0, \Theta \in \NN\).
    We say that \(p\) is \emph{\((B,C,\zeta,\Theta)\)-tamable} if all of the following conditions hold:
    \begin{enumerate}[label=(\itshape{\roman*})]
        \item \label{cond:tam_1} \(\abs{p_i}_X \geq B\), for \(i = 2, \dots, n-1\);
        \item \label{cond:tam_2} \(\langle (p_i)_-, (p_{i+1})_+ \rangle_{(p_i)_+}^{rel} \leq C\), for each \(i = 1, \dots, n-1\);
        \item \label{cond:tam_3} whenever \(p\) has consecutive backtracking along \(\mathcal{H}\)-components \(h_i, \dots, h_j\), of segments \(p_i, \dots, p_j\), such that
            \[
                \max\Big\{ \abs{h_k}_X \, \Big| \, k = i, \dots, j \Big\} \geq \Theta,
            \]
        it must be that \(d_X\Bigl( (h_i)_-,(h_j)_+ \Bigr) \geq \zeta\).
    \end{enumerate}
\end{definition}

The remainder of this section is devoted to showing the following result about quasigeodesicity of shortcuttings for tamable paths with appropriate constants.

\begin{proposition}
\label{prop:shortcutting_quasigeodesic}
    Given arbitrary \(c_0 \geq 14\delta\) and \(\eta \geq 0\) there are constants \( \lambda = \lambda(c_0) \geq 1\), \(c = c(c_0) \geq 0\) and \(\zeta = \zeta(\eta,c_0) \geq 1\) such that for any natural number \(\Theta \geq \zeta\) there is \(B_0 = B_0(\Theta,c_0) \geq 0\) satisfying the following.

    Let \(p = p_1 \dots p_n\) be a \((B_0,c_0,\zeta,\Theta)\)-tamable broken line in $\ga$ and let \(\Sigma(p,\Theta)\) be the \(\Theta\)-shortcutting, obtained by applying Procedure~\ref{proc:shortcutting} to \(p\),  \(\Sigma(p,\Theta) = f_0 e_1 f_1 \dots f_{m-1} e_m f_m\). 
    Then $e_k$ is non-trivial, for each $k=1,\dots,m$, and $\Sigma(p,\Theta)$  is \((\lambda,c)\)-quasigeodesic without backtracking.  

    Moreover, for any \(k \in \{ 1, \dots, m\}\), if we denote by \(e'_k\) the \(\mathcal{H}\)-component of \(\Sigma(p,\Theta)\) containing \(e_k\), then  \(\abs{e'_k}_X \geq \eta\).
\end{proposition}

The idea of the proof will be to show that under the above assumptions the broken line \(\Sigma(p,\Theta)\) satisfies the hypotheses of Proposition~\ref{prop:mpquasigeodesic}.

\begin{notation}
    For the remainder of this section we fix arbitrary constants \(c_0 \geq 14\delta\) and \(\eta \geq 0\). We let \(\rho = \kappa(4,c_3,0)\), where \(c_3 = c_3(c_0) \geq 0\) is the constant from Lemma~\ref{lem:concat} and \(\kappa(4,c_3,0)\) is the constant obtained by applying Proposition~\ref{prop:osinbcp} to \((4,c_3)\)-quasigeodesics.
    Let $\zeta_1>0$, $\lambda \ge 1$ and \(c \ge 0\) be the constants given by Proposition~\ref{prop:mpquasigeodesic}, applied with constant \(\rho\). 
    Note that the constants $\lambda$ and $c$ only depend on $c_0$ and do not depend on $\eta$.

    We now define the constant \(\zeta\) by
    \begin{equation}\label{eq:choice_of_eta}
        \zeta = \max\Big\{ \zeta_1, \eta \Big\} + 2\rho + 1.
    \end{equation}
    Finally we take any natural number $\Theta \ge \zeta$ and 
    \begin{equation}\label{eq:choice_of_B_0}
        B_0 = \max\Big\{ (12 c_0 + 12\delta + 1)\Theta, (4+c_3)\Theta+1 \Big\}.
    \end{equation}
\end{notation}

The proof of Proposition~\ref{prop:shortcutting_quasigeodesic} will consist of the following four lemmas. 
Throughout these lemmas we use the constants defined above and assume that \(p = p_1 \dots p_n\) is a \((B_0,c_0,\zeta,\Theta)\)-tamable broken line in $\ga$.
As before, we write \(v_0, \dots, v_d\) for the set of vertices of \(p\) in the order of their appearance. 
We let  \(\Sigma(p,\Theta) = f_0 e_1 f_1 \dots f_{m-1} e_m f_m\) be the \(\Theta\)-shortcutting and \(V(p,\Theta) = \{(s_0,t_0), \dots, (s_m,t_m)\}\) be the set obtained by applying Procedure~\ref{proc:shortcutting} to $p$.

\begin{lemma}
\label{lem:e_j-long}
    For each $k=1,\dots,m$, we have $|e_k|_X \ge \zeta >0$.
\end{lemma}

\begin{proof}
    By the construction in Procedure~\ref{proc:shortcutting}, there are pairwise connected \(\mathcal{H}\)-components \(h_1, \dots\), \( h_j\) of consecutive segments of \(p\), such that $j \ge 1$, \((h_1)_- = (e_k)_-\), \((h_s)_+ = (e_k)_+\) and $\max\{\abs{h_l}_X \mid l=1,\dots,j\} \ge \Theta$.

    If $j=1$ we see that $|e_k|_X=|h_1|_X \ge \Theta \ge \zeta$, and if $j>1$ then we know that $|e_k|_X \ge  \zeta$ by property \ref{cond:tam_3} from Definition~\ref{def:tamable}.
\end{proof}

\begin{lemma}
\label{lem:beta_quasigeodesic}
    The subpaths of \(p\) between \(v_{s_k}\) and \(v_{t_k}\), for \(k = 0, \dots, m\), are \((4,c_3)\)-quasigeodesic.
\end{lemma}

\begin{proof}
    We write \(c_1=c_1(c_0)=12 c_0 + 12\delta + 1\), as in Lemma~\ref{lem:concat}.
    
    Choose any \(k \in \{0,\dots,m\}\) and denote by \(p'\) be the subpath of \(p\) starting at \(v_{s_k}\) and terminating at \(v_{t_k}\).
    If \(v_{s_k}\) and \(v_{t_{k}}\) are both vertices of $p_i$, for some $i  \in \{1,\dots,n\}$, then \(p'\) is geodesic and we are done.
    Otherwise \(p' = p'_i p_{i+1} \dots p_{j-1} p'_j\), for some \(i, j \in \{1, \dots, n\}\), with \(i < j\), where \(p'_i\) is a terminal segment of \(p_i\) and \(p'_j\) is an initial segment of \(p_j\).

    By Remark~\ref{rem:shortcutting}(c), the paths \(p_{i+1}, \dots, p_{j-1}\) contain no \(\mathcal{H}\)-components \(h\) with \(|h|_X \geq \Theta\).
    Since \(p\) is \((B_0,c_0,\zeta,\Theta)\)-tamable, \(\abs{p_l}_X \geq B_0\) for each \(l = i+1, \dots, j-1\) by condition~\ref{cond:tam_1}.
    Thus we can combine Lemma~\ref{lem:rel_geods_with_short_comps} with \eqref{eq:choice_of_B_0} to obtain
    \[
        \dxh\Bigl((p_l)_-,(p_l)_+\Bigr)= \ell(p_l) \geq \frac{1}{\Theta}|p_l|_X \geq \frac{B_0}{\Theta} \geq c_1, \text{ for each } l \in \{i+1, \dots, j-1\}.
    \]

    Again, from the assumption that \(p\) is \((B_0,c_0,\zeta,\Theta)\)-tamable, we have that 
    \[
        \langle (p_l)_-,(p_{l+1})_+ \rangle_{(p_l)_+}^{rel} \leq c_0, \text{ for all } l=i,\dots,j-1,
    \]
    using condition~\ref{cond:tam_2}. 
    In view of Remark~\ref{rem:Gr_prod_ineq},
    \[
        \langle (p'_i)_-,(p_{i+1})_+ \rangle_{(p'_i)_+}^{rel} \leq c_0~\text{ and }~\langle (p_{j-1})_-,(p'_{j})_+ \rangle_{(p_{j-1})_+}^{rel} \leq c_0.
    \]
    Therefore we can use Lemma~\ref{lem:concat} to conclude that \(p'\) is \((4,c_3)\)-quasigeodesic, as required.
\end{proof}

\begin{lemma}
\label{lem:short_ending_components}
    If \(k \in \{0, \dots, m-1\}\) and \(h\) is an \(\mathcal{H}\)-component of \(f_{k}\) or \(f_{k+1}\) that is connected to \(e_{k+1}\), then \(\abs{h}_X \leq \rho\).
\end{lemma}

\begin{proof} 
    Arguing by contradiction, suppose that \(h\) is an \(\mathcal{H}\)-component of \(f_{k}\)  connected to \(e_{k+1}\) and satisfying \(\abs{h}_X > \rho\) (the other case when \(h\) is an \(\mathcal{H}\)-component of \(f_{k+1}\) is similar).  
    Remark~\ref{rem:comp_of_geod_is_an_edge} tells us that \(h\) is a single edge of \(f_k\).
    Moreover, since \(h\) and \(e_{k+1}\) are connected and \((f_k)_+ = (e_{k+1})_-\), we have \(d_{X\cup\mathcal{H}}(h_-,(f_k)_+) \leq 1\).
    The geodesicity of \(f_{k}\) in \(\ga\) now implies that \(h\) must in fact be the last edge of \(f_{k}\), 
    so that \(h_+ = (f_k)_+ = v_{t_k}\).

    Let \(p' = p'_i p_{i+1} \dots p_{j-1} p'_j\) be the subpath of \(p\) with \(p'_- = v_{s_k}\) and \(p'_+ = v_{t_k}\), where \(p'_i\) and \(p'_j\) are non-trivial subpaths of \(p_i\) and \(p_j\) respectively.
    By Lemma~\ref{lem:beta_quasigeodesic}, \(p'\) is \((4,c_3)\)-quasigeodesic.

    Since \(\abs{h}_X > \rho = \kappa(4,c_3,0)\) we may apply Proposition \ref{prop:osinbcp} to find that \(h\) is connected to an \(\mathcal{H}\)-component of \(p'\) (which may consist of multiple edges, each of which is an \(\mathcal{H}\)-component of a segment of \(p\)).
    We write \(h'\) for the final edge of this \(\mathcal{H}\)-component and denote by \(u\) the edge of \(p\) with endpoints \(v_{t_k}\) and \(v_{t_k + 1}\) (see Figure~\ref{fig:short_ending_comps}).
    Procedure~\ref{proc:shortcutting} and the assumption that $h$ is connected to $e_{k+1}$ imply that \(u\) is an \(\mathcal{H}\)-component of a segment of \(p\) and \(h'\) and \(u\) are connected as \(\mathcal{H}\)-subpaths of \(p\).
    
    \begin{figure}[hb]
        \centering
        \includegraphics[]{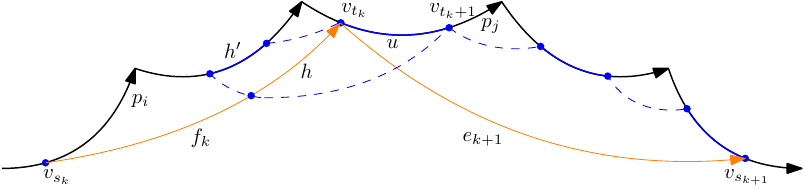}
        \caption{Illustration of Lemma~\ref{lem:short_ending_components}.}
        \label{fig:short_ending_comps}
    \end{figure}

    Suppose, first, that \(p'_j\) is a proper subpath of \(p_j\), so that \(u\) belongs to the segment \(p_j\), as shown on Figure~\ref{fig:short_ending_comps}.
    Then there are the following possibilities.

    \medskip

    \emph{\underline{Case 1}: \(h'\) is an edge of \(p_j\).}

        In this case \(h'\) and \(u\) are connected distinct \(\mathcal{H}\)-subpaths of \(p_j\), which is a geodesic.
        This contradicts the observation of Remark~\ref{rem:comp_of_geod_is_an_edge}, that geodesics are without backtracking and \(\mathcal{H}\)-components of geodesics are single edges.

    \medskip

    \emph{\underline{Case 2}: \(h'\) is an \(\mathcal{H}\)-component of \(p_{j-1}\).}

        Let \(t \in \{0,\dots,d\}\) be such that \(v_t = h'_-\), and note that
        \begin{equation}
        \label{eq:tk-1_s_sk}
            s_k \leq t < t_k.
        \end{equation}
        By the construction from Procedure~\ref{proc:shortcutting}, there are pairwise connected \(\mathcal{H}\)-components \(h_j, \dots, h_{j+l}\), of segments \(p_j, \dots, p_{j+l}\), with \((e_{k+1})_- = (h_j)_-=v_{t_k}\) and \((e_{k+1})_+ = (h_{j+l})_+=v_{s_{k+1}}\), such that \[\max\{\abs{h_j}_X, \dots,\abs{h_{j+l}}_X\} \geq \Theta\] and \(l \in \{0,\dots,n-j\}\)  is chosen to be maximal with this property.
        Then the components \(h', h_j, \dots, h_{j+l}\) constitute a larger instance of consecutive backtracking, starting at $h'_-=v_t$, with \[\max  \Big\{\abs{h'}_X,\abs{h_j}_X, \dots,\abs{h_{j+l}}_X\} \geq \Theta.\] In view of \eqref{eq:tk-1_s_sk}, this contradicts the choice of $t_k$ and the inclusion of $(s_k,t_k)$ in the set $V(p,\Theta)$ at Steps 2 and 3 of Procedure~\ref{proc:shortcutting}.

    \medskip

    \emph{\underline{Case 3}: \(h'\) is an \(\mathcal{H}\)-component of one of the paths \(p'_i, p_{i+1}, \dots, p_{j-2}\).}

        Then the subpath $q$ of $p'$ from $h'_+$ to $p'_+=v_{t_k}$ contains all of $p_{j-1}$.
        By Remark~\ref{rem:shortcutting}(c), \(p_{j-1}\) contains no \(\mathcal{H}\)-components \(q\) satisfying \(\abs{q}_X \geq \Theta\).
        Therefore, in view of Lemma~\ref{lem:rel_geods_with_short_comps} and the assumption that \(p\) is \((B_0,c_0,\zeta,\Theta)\)-tamable, we can deduce that
        \(
            \Theta \ell(p_{j-1}) \geq \abs{p_{j-1}}_X \ge B_0.
        \)
        Combining this with the $(4,c_3)$-quasigeodesicity of \(p'\), we obtain
        \begin{equation*}\label{eq:dist_of_h'}
            \dxh(h'_+,p'_+) \geq \frac{1}{4} \Big( \ell(q) - c_3 \Big) \geq \frac{1}{4} \Big( \ell(p_{j-1}) - c_3 \Big) \geq \frac{B_0}{4\Theta} - \frac{c_3}{4} > 1,
        \end{equation*}
        where the last inequality follows from \eqref{eq:choice_of_B_0}.
        On the other hand, the fact that \(h'\) and \(h\) are connected gives \(d_{X\cup\mathcal{H}}(h'_+,p'_+) = \dxh(h'_+,h_+) \leq 1\), contradicting the above.

    \medskip

    In each case we arrive at a contradiction, so it is impossible that \(\abs{h}_X > \rho\) if \(p'_j\) is a proper subpath of \(p_j\).
    If \(p'_j\) is instead the whole subpath \(p_j\), we may carry out a similar analysis.
    In this situation it must be that \(u\) is an \(\mathcal{H}\)-component of the segment \(p_{j+1}\).
    We now have only two relevant cases to consider:  \(h'\) is an \(\mathcal{H}\)-component of \(p_j\) or \(h'\) is an \(\mathcal{H}\)-component of one of the paths \(p'_i, p_{i+1}, \dots, p_{j-1}\). Both of them will lead to contradictions similarly to Cases 2 and 3 above.

    Therefore it must be that \(\abs{h}_X \leq \rho\), as required.
\end{proof}

\begin{lemma}
\label{lem:ek_ek+1_not_connected}
    For each \(k \in \{ 1, \dots, m-1\}\), the  \(\mathcal{H}\)-subpaths \(e_k\) and \(e_{k+1}\) of $\Sigma(p,\Theta)$ are not connected.
\end{lemma}

\begin{proof}
    Suppose that $e_k$ is connected to $e_{k+1}$ for some $k\in \{1,\dots,m-1\}$.
    As before, according to Procedure~\ref{proc:shortcutting}, there exist two sets of pairwise connected \(\mathcal{H}\)-components of consecutive segments of $p$, \(h_1, \dots, h_i\) and \(q_1, \dots, q_j\), such that \((h_1)_- = (e_k)_-\), \((h_i)_+ = (e_k)_+\), \((q_1)_- = (e_{k+1})_-\), \((q_j)_+ = (e_{k+1})_+\) and
    \[
        \max\Big\{\abs{h_1}_X, \dots,\abs{h_i}_X\Big\} \geq \Theta,~~  \max\Big\{\abs{q_1}_X, \dots,\abs{q_j}_X \Big\} \geq \Theta .
    \]

    Since \(e_k\) and \(e_{k+1}\) are connected, \(h_i\) and \(q_1\) will be connected $\mathcal{H}$-subpaths of $p$, in particular they cannot be contained in the same segment of the broken line $p$ by Remark~\ref{rem:comp_of_geod_is_an_edge}.
    If \(h_i\) and \(q_1\) are \(\mathcal{H}\)-components of adjacent segments of \(p\), then the components \(h_1, \dots, h_i, q_1, \dots, q_j\) constitute a longer instance of consecutive backtracking in $p$, which contradicts the construction of $e_k$ in Procedure~\ref{proc:shortcutting}.

    Therefore it must be the case that the subpath \(p'\) of \(p\) between \((e_k)_+=(h_i)_+=v_{s_{k}}\) and \((e_{k+1})_-=(q_1)_-=v_{t_{k}}\) contains at least one full segment \(p_l\) (with \(1 < l < n\)).
    By Remark~\ref{rem:shortcutting}(c) the path \(p'\) has no \(\mathcal{H}\)-components \(h\) satisfying \(\abs{h}_X \geq \Theta\). Therefore we can combine Lemma~\ref{lem:rel_geods_with_short_comps} with the fact that $p$ is \((B_0,c_0,\zeta,\Theta)\)-tamable to deduce that 
    \begin{equation}\label{eq:len_of_p'_adj}
        \ell(p') \geq \ell(p_l) \geq \frac{\abs{p_l}_X}{\Theta} \geq \frac{B_0}{\Theta}.
    \end{equation}
    Moreover, by Lemma~\ref{lem:beta_quasigeodesic} the path \(p'\) is \((4,c_3)\)-quasigeodesic so that
    \[
        \ell(p') \leq 4d_{X\cup\mathcal{H}}((e_k)_+,(e_{k+1})_-) + c_3 \leq 4 + c_3,
    \]
    where the last inequality is true because \(e_k\) and \(e_{k+1}\) are connected.
    Combined with (\ref{eq:len_of_p'_adj}), the above inequality gives
    $ B_0 \leq (4 + c_3)\Theta$,
    which contradicts the choice of \(B_0\) in \eqref{eq:choice_of_B_0}.

    Therefore $e_k$ and $e_{k+1}$ cannot be connected, for any $k \in \{1,\dots,m-1\}$.
\end{proof}

\begin{proof}[Proof of Proposition~\ref{prop:shortcutting_quasigeodesic}]
    The construction, together with Lemmas~\ref{lem:e_j-long}, \ref{lem:short_ending_components} and \ref{lem:ek_ek+1_not_connected}, show that the \(\Theta\)-shortcutting \(\Sigma(p,\Theta) = f_0 e_1 f_1 \dots f_{m-1} e_m f_m\) satisfies the hypotheses of Proposition~\ref{prop:mpquasigeodesic} and $e_k$ is non-trivial, for each $k=1,\dots,m$.
    Therefore  \(\Sigma(p,\Theta)\) is \((\lambda,c)\)-quasigeodesic without backtracking.

    For the final claim of the proposition, consider any \(k \in \{ 1, \dots, m\}\) and denote by \(e'_k\) the \(H_\nu\)-component of \(\Sigma(p,\Theta)\) containing \(e_k\), for some $\nu \in \Nu$.     Lemma~\ref{lem:ek_ek+1_not_connected} implies that $e'_k$ is the concatenation $h_1 e_k h_2$, where $h_1$ is either trivial or it is an $H_\nu$-component of $f_{k-1}$, and $h_2$     is either trivial or it is an $H_\nu$-component of $f_{k}$.
    Combining the triangle inequality with Lemmas~\ref{lem:e_j-long}, \ref{lem:short_ending_components} and equation \eqref{eq:choice_of_eta}, we obtain
    \[
        \abs{e'_k}_X \geq \abs{e_k}_X - \abs{h_1}_X-\abs{h_2}_X  \geq \zeta-2\rho \ge \eta,
    \]
    as required.
\end{proof}


\section{Metric quasiconvexity theorem}
\label{sec:metric_qc}

This section comprises a proof of Theorem \ref{thm:metric_qc}, and, as usual, we work under Convention~\ref{conv:main}.
First we will show that if some subgroups \(Q'\leqslant Q\) and \(R'\leqslant R\) satisfy conditions \descref{C1}-\descref{C5} with appropriately large constants, then minimal type path representatives of \(\langle Q', R' \rangle\) meet the conditions of Proposition~\ref{prop:shortcutting_quasigeodesic}.
We will then use the quasigeodesicity of shortcuttings of these path representatives to obtain properties \descref{P1}-\descref{P3}.

\begin{lemma}
\label{lem:C2_implies_old_C2}
    Suppose that \(Q' \leqslant Q\) and \(R' \leqslant R\) satisfy \descref{C2} with constant \(B \geq 0\). Then
    \[
        \minx \Bigl( (Q' \cup R') \setminus S \Bigr) \geq B.
    \]
\end{lemma}

\begin{proof}
    Let \(g \in (Q' \cup R') \setminus S\).
    If \(g \in Q'\) then \(g \notin R\) as \(g \notin S\).
    Therefore \(g \in Q' \setminus R \subseteq R \langle Q', R' \rangle R \setminus R\), whence \(\abs{g}_X \geq B\) by \descref{C2}.
    Similarly, if \(g \in R'\) then \(g \in Q \langle Q', R' \rangle Q \setminus Q\), and  \descref{C2} again implies that \(\abs{g}_X \geq B\).
\end{proof}

\begin{notation}\label{not:constants_and_P_1_in_part_1}
\label{not:main_in_part_2}
    For the remainder of this section we fix the following notation:
    \begin{itemize}
        \item \(C_0\) is the constant provided by Lemma~\ref{lem:bddinnprod};
        \item $c_0=\max\{C_0,14\delta\}$ and $c_3=c_3(c_0)$ is the constant obtained by applying Lemma~\ref{lem:concat};
        \item $\lambda=\lambda(c_0)$ and $c=c(c_0)$ are the first two constants from Proposition~\ref{prop:shortcutting_quasigeodesic};
        \item \(C_1 \ge 0\) is the constant from Lemma \ref{lem:shortspikes};
        \item $\mathcal{P}_1$ is the finite family of parabolic subgroups of $G$ defined by
        \begin{equation*}\label{eq:family_P_1}
            \mathcal{P}_1 = \lbrace t H_\nu t^{-1} \, | \, \nu \in \Nu, \abs{t}_X \leq C_1 \rbrace.
        \end{equation*}
    \end{itemize}
\end{notation}

\begin{lemma}
\label{lem:pathreps_have_qgd_shortcutting}
    For each \(\eta \geq 0\) there are constants \(C_3 = C_3(\eta) \geq 0\), $\zeta=\zeta(\eta) \ge 1$, \( \Theta_1 = \Theta_1(\eta) \in \NN\)  and  \(B_1 = B_1(\eta) \geq 0\)  such that the following is true.

    Suppose that \(Q'\leqslant Q\) and \(R'\leqslant R\) are subgroups satisfying conditions \descref{C1}-\descref{C5} with constants \(B \geq B_1\) and \(C \geq C_3\) and family \(\mathcal{P} \supseteq \mathcal{P}_1\). If \(p = p_1 \dots p_n\) is a minimal type path representative for an element \(g \in \langle Q', R' \rangle\) then \(p\) is \((B,c_0,\zeta,\Theta_1)\)-tamable.

    Moreover, let \(\Sigma(p,\Theta_1) = f_0 e_1 f_1 \dots f_{m-1} e_m f_m\) be the \(\Theta_1\)-shortcutting of \(p\) obtained from Procedure~\ref{proc:shortcutting}, and let $e_k'$ be the \(\mathcal{H}\)-component of \(\Sigma(p,\Theta_1)\) containing \(e_k\), $k=1,\dots,m$. Then \(\Sigma(p,\Theta_1)\) is a \((\lambda,c)\)-quasigeodesic without backtracking and  \(\abs{e'_k}_X \geq \eta\), for each \(k = 1, \dots, m\).
\end{lemma}

\begin{proof}
    We define the following constants:
    \begin{itemize}
        \item \(\zeta = \zeta(\eta,c_0) \geq 0\), the constant provided by Proposition~\ref{prop:shortcutting_quasigeodesic};
        \item \(C_3 = C_2(\zeta) \geq 0\), where \(C_2(\zeta)\) is given by Proposition~\ref{prop:long_multitracking};
        \item \(\Theta_1 = \max\{\Theta_0(\zeta), \zeta\}\), where \(\Theta_0\) is the constant of Lemma~\ref{lem:longadjbacktracking};
        \item \(B_1 = \max\{ B_0(\Theta_1,c_0), C_2(\zeta) \} \geq 0\), where \(B_0\) is the remaining constant of Proposition~\ref{prop:shortcutting_quasigeodesic}.
    \end{itemize}

    Let \(B \geq B_1\) and \(C \geq C_3\).
    Suppose that $Q'$, $R'$, $g$ and $p$ are as in the statement of the lemma.
    In view of Remark~\ref{rem:alt}, \(\elem{p_i} \in (Q' \cup R') \setminus S\), for every \(i = 2, \dots, n-1\).
    Therefore, by Lemma~\ref{lem:C2_implies_old_C2}, we have
    \begin{equation}
    \label{eq:pi_long}
        \abs{p_i}_X \geq B,  \text{ for each } i = 2, \dots, n-1.
    \end{equation}

    On the other hand, Lemma~\ref{lem:bddinnprod} tells us that
    \begin{equation}
    \label{eq:pathreps_bdd_in_prod}
        \langle (p_i)_-, (p_{i+1})_+ \rangle_{(p_i)_+}^{rel} \leq C_0 \le c_0, \text{ for each }i = 1, \dots, n-1.
    \end{equation}

    Now suppose that \(p\) has consecutive backtracking along \(\mathcal{H}\)-components \(h_i, \dots, h_j\) of segments \(p_i, \dots, p_j\) satisfying
    \[
        \max\Big\{ \abs{h_i}_X,\dots, \abs{h_j}_X \Big\} \geq \Theta_1.
    \]
    If \(j = i+1\) then Lemma~\ref{lem:longadjbacktracking} and the choice of \(\Theta_1\) give that \(d_X((h_i)_-,(h_j)_+) \geq \zeta\).
    Otherwise Proposition~\ref{prop:long_multitracking} gives the same inequality.
    The above together with (\ref{eq:pi_long}) and (\ref{eq:pathreps_bdd_in_prod}) show that \(p\) is \((B,c_0,\zeta,\Theta_1)\)-tamable.
    
    The remaining claims of the lemma follow from Proposition~\ref{prop:shortcutting_quasigeodesic}.
\end{proof}

We can now deduce the relative quasiconvexity of \(\langle Q', R' \rangle\) by applying Lemma~\ref{lem:pathreps_have_qgd_shortcutting} with $\eta=0$.

\begin{proposition}
\label{prop:metric_P1} 
    Let $\beta_1=B_1(0)$ and $\gamma_1=C_3(0)$ be the constants provided by Lemma~\ref{lem:pathreps_have_qgd_shortcutting} applied to the case when $\eta=0$.

    Suppose that \(Q'\leqslant Q\) and \(R'\leqslant R\) are relatively quasiconvex subgroups of $G$ satisfying conditions \descref{C1}-\descref{C5} with family \(\mathcal{P} \supseteq \mathcal{P}_1\) and constants \(B \geq \beta_1\), \( C \geq \gamma_1\). 
    Then the subgroup \(\langle Q', R' \rangle\) is relatively quasiconvex in $G$.
\end{proposition}

\begin{proof}
    By assumption the subgroups \(Q'\) and \(R'\) are relatively quasiconvex, with some quasiconvexity constant $\varepsilon' \ge 0$.
    For any element \(g \in \langle Q', R' \rangle\) consider a geodesic \(\tau\) in \(\Gamma(G,X\cup\mathcal{H})\) with \(\tau_- = 1\) and \(\tau_+ = g\). Let \(u\) be any vertex of \(\tau\).

    Since \(g \in \langle Q', R' \rangle\), it has a path representative \(p = p_1 \dots p_n\) of minimal type, with $p_-=1$.
    Let \(\Sigma(p,\Theta) = f_0 e_1 f_1 \dots f_{m-1} e_m f_m\) be the \(\Theta\)-shortcutting of \(p\) obtained from Procedure~\ref{proc:shortcutting}, where $\Theta=\Theta_1(0)$ is provided by Lemma~\ref{lem:pathreps_have_qgd_shortcutting}.
    This lemma implies that $p$ is \((B,c_0,\zeta,\Theta)\)-tamable and \(\Sigma(p,\Theta)\) is a \((\lambda,c)\)-quasigeodesic without backtracking, where $\lambda \ge 1$ and $c \ge 0$ are the constants fixed in Notation~\ref{not:main_in_part_2}.    Therefore, by Proposition~\ref{prop:osinbcp}, there is a phase vertex \(v\) of \(\Sigma(p,\Theta)\) with \(d_X(u,v) \leq \kappa(\lambda,c,0)\).

    Since each $e_i$ is a single edge, the vertex \(v\) lies on the geodesic subpath \(f_i\) of \(\Sigma(p,\Theta)\), for some \(i \in \{0, \dots, m\}\).
    The subpath of \(p\) sharing endpoints with \(f_i\) is \((4,c_3)\)-quasigeodesic by Lemma~\ref{lem:beta_quasigeodesic}.
    Hence there is a vertex \(w\) of \(p\) such that \(d_X(v,w) \leq \kappa(4,c_3,0)\), by Proposition~\ref{prop:osinbcp}.

    Now \(w\) is a vertex of a subpath \(p_j\) of \(p\), for some \(j \in \{ 1, \dots, n \}\).
    Let \(x = (p_j)_-\), and note that \(x \in \langle Q', R' \rangle\).
    Without loss of generality, suppose that \(\elem{p_j} \in Q'\) (the case when  \(\elem{p_j} \in R'\) can be treated similarly).
    Then by the relative quasiconvexity of \(Q'\), \(d_X(w,xQ') \leq \varepsilon'\), whence \(d_X(w,\langle Q', R' \rangle) \leq \varepsilon'\).
    Therefore
    \begin{align*}
        d_X(u,\langle Q', R' \rangle) &\leq d_X(u,v) + d_X(v,w) + d_X(w,\langle Q', R' \rangle) \\
            &\leq \kappa(\lambda,c,0) + \kappa(4,c_3,0) + \varepsilon',
    \end{align*}
    so that \(\langle Q', R' \rangle\) is a relatively quasiconvex subgroup of $G$, with the quasiconvexity constant $\kappa(\lambda,c,0) + \kappa(4,c_3,0) + \varepsilon'$.
\end{proof}

We will next show that properties \descref{P2} and \descref{P3} will be satisfied if one chooses the constants \(B\) and \(C\) of \descref{C1}-\descref{C5} to be sufficiently large with respect to \(A\).

\begin{lemma}
\label{lem:metric_P2}
    For any \(A \geq 0\) there exist constants \(\beta_2 = \beta_2(A) \geq 0\) and $\gamma_2=\gamma_2(A) \ge 0$ such that if \(Q'\leqslant Q\) and \(R'\leqslant R\) satisfy conditions \descref{C1}-\descref{C5} with constants \(B \geq \beta_2\) and \(C \geq \gamma_2\) and family \(\mathcal{P} \supseteq \mathcal{P}_1\), then \[\minx \Big(\langle Q', R' \rangle \setminus S\Big) \geq A.\]
\end{lemma}

\begin{proof} 
    Given any \(A \geq 0\) let \(\eta = \eta(\lambda,c,A)\) be the constant provided by Lemma~\ref{lem:qgds_with_long_comps}.
    Using Lemma~\ref{lem:pathreps_have_qgd_shortcutting}, set 
    \[ \Theta = \Theta_1(\eta),
    ~ \gamma_2 = C_3(\eta)~ \text{ and }\beta_2 = \max\{B_1(\eta),(4A + c_3)\Theta\}.
    \]

    Suppose that \(Q'\) and \(R'\) satisfy conditions \descref{C1}-\descref{C5} with constants \(B \geq \beta_2\) and \(C \geq \gamma_2\), and let \(g \in \langle Q', R' \rangle\) be any element with \(\abs{g}_X < A\). Let \(p = p_1 \dots p_n\) be a path representative of \(g\) with minimal type.
    By Lemma~\ref{lem:pathreps_have_qgd_shortcutting}, $p$ is \((B,c_0,\zeta,\Theta_1)\)-tamable, the \(\Theta\)-shortcutting \(\Sigma(p,\Theta) = f_0 e_1 f_1 \dots f_{m-1} e_m f_m\) is \((\lambda,c)\)-quasigeodesic without backtracking, and, for each \(k = 1, \dots, m\), \(e'_k\), the \(\mathcal{H}\)-component of \(\Sigma(p,\Theta)\) containing \(e_k\), is isolated and satisfies \(\abs{e'_k}_X \geq \eta\).

    If \(m \geq 1\), then, according to Lemma~\ref{lem:qgds_with_long_comps}, \(\abs{g}_X = \abs{\Sigma(p,\Theta)}_X \geq A\), contradicting our assumption.
    Therefore it must be the case that \(m = 0\) and \(\Sigma(p,\Theta) = f_0\). 
    Since \(p_- = (f_0)_-\) and \(p_+ = (f_0)_+\), Lemma~\ref{lem:beta_quasigeodesic} tells us that \(p\) is \((4,c_3)\)-quasigeodesic.
    Moreover, following Remark~\ref{rem:shortcutting}(c), we see that \(p_i\) has no $\mathcal{H}$-component $h$ with $\abs{h}_X \ge \Theta$, for each $i=1,\dots,n$.

    Now, arguing by contradiction, suppose that $g \notin S$. Then $\elem{p_1} \in (Q' \cup R') \setminus S$ (by Remark~\ref{rem:alt}), so
    \(\abs{p_1}_X \geq B \ge \beta_2\), by Lemma~\ref{lem:C2_implies_old_C2}. Lemma~\ref{lem:rel_geods_with_short_comps} now implies that
    \[
        \ell(p_1)\ge \beta_2/\Theta \ge 4A+c_3.
    \]

    Since $\ell(p) \ge \ell(p_1)$,  the $(4,c_3)$-quasigeodesicity of \(p\) yields
    \[
        A > \abs{g}_{X} \ge |g|_{X\cup \mathcal{H}}  = |p|_{X\cup \mathcal{H}} \geq \frac{1}{4} \left(\ell(p) - c_3\right) \geq A,
    \]
    which is a contradiction. Therefore $g \in S$ and the lemma is proved.
\end{proof}

In order to prove that property \descref{P3} holds for the subgroups \(Q'\) and \(R'\), we need to consider path representatives of elements $g \in Q \langle Q', R' \rangle R$. These path representatives will necessarily have to be slightly different from those in Definition~\ref{def:path_reps}.

\begin{definition}[Path representative, II] 
\label{def:double_coset_path_reps}
    Let \(g\) be an element of \(Q \langle Q', R' \rangle R\).
    Suppose that \(p = q p_1 \dots p_n r\) is  a broken line in \(\Gamma(G,X\cup\mathcal{H})\), satisfying the following conditions:
    \begin{itemize}
        \item \(\elem{p} = g\);
        \item \(\elem{q} \in Q\)  and \(\elem{r} \in R\);
        \item \(\elem{p_i} \in Q' \cup R'\), for each \(i \in \{1, \dots n\}\).
    \end{itemize}
    Then we say that \(p\) is a \emph{path representative} of \(g\) in the product \(Q \langle Q', R' \rangle R\).
\end{definition}

Similarly to Definition~\ref{def:type_of_path_rep}, we can define types for such path representatives.

\begin{definition}[Type of a path representative, II] \label{def:type_of_double_coset_path_rep}
    Suppose that \(p = q p_1 \dots p_n r\) is a path representative of some \(g \in Q \langle Q', R' \rangle R\), as described in Definition~\ref{def:double_coset_path_reps}.
    Let \(Y\) denote the set of all \(\mathcal{H}\)-components of the segments of \(p\).
    We define the \emph{type} of the path representative \(p\) to be the triple
    \[\tau(p) = \Big(n, \ell(p),\sum_{y \in Y} |y|_X \Big) \in {\NN_0}^3.\]
\end{definition}

\begin{remark}
\label{rem:props_of_double_coset_path_reps}
    Note that, by Definition~\ref{def:double_coset_path_reps}, a path representative $p=q p_1 \dots p_n r$, of an element $g \in Q \langle Q', R' \rangle R\setminus QR$, must necessarily satisfy $n>0$. 
    Moreover, if $p$ has minimal type (so $n$ is the smallest possible) then $\elem{p_1} \in R' \setminus S$, $\elem{p_n} \in Q' \setminus S$ and the labels of $p_1,\dots,p_{n}$ will alternate between representing elements of $R' \setminus S$ and $Q' \setminus S$. 
    It follows that the integer $n$ must be even, so $n \ge 2$.
\end{remark}

For example, if $g \in R'Q'\setminus QR$ then a minimal type path representative of $g$ will have the form $qp_1p_2 r$, where $q$ and $r$ are trivial paths, $\elem{p_1} \in R'$ and  $\elem{p_2} \in Q'$.

It is not difficult to check that the results of Sections~\ref{sec:path_reps}, \ref{sec:adj_backtracking}, and \ref{sec:multitracking} hold equally well for minimal type path representatives of the above form for elements $g \in Q \langle Q', R' \rangle R \setminus QR$, with only superficial adjustments to the proofs in those sections. It follows that Lemma~\ref{lem:pathreps_have_qgd_shortcutting} also remains valid in these settings.

\begin{lemma}
\label{lem:metric_P3}
    In the statement of Lemma~\ref{lem:metric_P2} we can add that
    \[
        \minx \Big( Q \langle Q', R' \rangle R \setminus QR\Big) \geq A.
    \]
\end{lemma}

\begin{proof}
    For any \(A \geq 0\) we define the constants $\eta$, $\Theta$, $\gamma_2$ and $\beta_2$ exactly as in Lemma~\ref{lem:metric_P2}.

    Suppose that for some element $g \in Q \langle Q',R'\rangle R \setminus QR$ we have $\abs{g}_X<A$. 
    Let  \(p = q p_1 \dots p_n r\) be a minimal type path representative of $g$, of the form described in Definition~\ref{def:double_coset_path_reps}.

    Arguing in the same way as in Lemma~\ref{lem:metric_P2}, we can deduce that $p$ is $(4,c_3)$-quasigeodesic and for each $i=1,\dots,n$, $p_i$ has no $\mathcal{H}$-component $h$ with $\abs{h}_X \ge\Theta$.

    According to Remark~\ref{rem:props_of_double_coset_path_reps}, $n \ge 2$ and  \(\elem{p_1} \in R' \setminus S\). So, by Lemma~\ref{lem:C2_implies_old_C2}, \(\abs{p_1}_X \geq B \ge \beta_2\). 
    The same argument as in Lemma~\ref{lem:metric_P2} now yields that $\abs{g}_X \ge A$, leading to a contradiction.
    Therefore it must be that $\abs{g}_X \ge A$ for any $g \in Q \langle Q',R'\rangle R \setminus QR$, and the proof is complete.
\end{proof}

We are finally able to prove Theorem \ref{thm:metric_qc}.

\begin{proof}[Proof of Theorem \ref{thm:metric_qc}]
    Choose $\mathcal{P}$ to be the finite family $\mathcal{P}_1$, defined in Notation~\ref{not:main_in_part_2}. Given any \(A \geq 0\), we apply Proposition~\ref{prop:metric_P1} and Lemma~\ref{lem:metric_P2} to define the constants \[B = \max\{\beta_1,\beta_2(A)\}~\text{ and }~ C = \max\{\gamma_1,\gamma_2(A)\}.\]
 
    Suppose that \(Q' \leqslant Q\) and \(R' \leqslant R\) are subgroups satisfying conditions \descref{C1}-\descref{C5} with constants \(B\) and \(C\) and the finite family of parabolic subgroups \(\mathcal{P}\).
    Then property \descref{P1} holds by Proposition~\ref{prop:metric_P1}, while properties \descref{P2} and \descref{P3} are satisfied by Lemmas~\ref{lem:metric_P2} and~\ref{lem:metric_P3}  respectively.
\end{proof}


\section{Using separability to establish the conditions of the quasiconvexity theorem}
\label{sec:sep->metric}

In this section we will show how one can prove the existence of finite index subgroups $Q' \leqslant_f Q$ and $R' \leqslant_f R$, satisfying the conditions \descref{C1}--\descref{C5} from Subsection~\ref{subsec:3.1}, using certain separability assumptions. We start with finding such assumptions for establishing
 \descref{C2} and \descref{C3}.

\begin{proposition}
\label{prop:sep->C2-C3} 
    Let $G$ be a group generated by a finite subset $X$, let $Q,R \leqslant G$ and $S=Q \cap R$, and let $\mathcal P$ be a finite collection of subgroups of $G$. Suppose that $Q$ and $R$ are separable in $G$ and $PS$ is separable in $G$, for each $P \in \mathcal{P}$.

    Then for any constants $B,C \ge 0$ there exists a finite index subgroup $L \leqslant_f G$, with $S \subseteq L$, such that conditions \descref{C2} and \descref{C3} are satisfied by arbitrary subgroups $Q' \leqslant Q \cap L$ and $R' \leqslant R \cap L$.
\end{proposition}

\begin{proof} 
    Combining the separability of $Q$ and $R$ in $G$ with Lemma~\ref{lem:sep->large_minx}, we can find $E_1,E_2 \lhd_f G$ such that $\minx(Q E_1 \setminus Q) \ge B$ and $\minx(R E_2 \setminus R) \ge B$. 
    Set $N_0=E_1 \cap E_2 \lhd_f G$ and observe that \[QS N_0 Q=QN_0Q=QQ N_0=Q N_0 \subseteq Q E_1,\] as $Q$ is a subgroup containing $S$ and normalising $N_0$ in $G$. 
    Similarly, $RSN_0 R=R N_0 \subseteq RE_2$, therefore
    \begin{equation}
    \label{eq:1N_0}
        \minx(QS N_0 Q \setminus Q) \ge B \text{ and } \minx(RS N_0 R \setminus R) \ge B.
    \end{equation}

    Let $\mathcal{P}=\{P_1,\dots,P_k\}$. 
    The assumptions imply that for every $i \in \{1,\dots,k\}$ the double coset $P_i S$ is separable in $G$, hence we can apply Lemma~\ref{lem:sep->large_minx} again to find finite index normal subgroups $N_i \lhd_f G$ satisfying
    \begin{equation}
    \label{eq:1N_i}
        \minx(P_i S N_i \setminus P_i S) \ge C,~\text{ for each } i=1,\dots,k.
    \end{equation}

    Now set $L= \bigcap_{i=0}^k SN_i \leqslant_f G$, and choose arbitrary subgroups $Q' \leqslant Q \cap L$ and $R' \leqslant R \cap L$. 
    Then $S \subseteq L$ and $\langle Q',R' \rangle \subseteq L \subseteq SN_i$, for all $i=0,\dots,k$, by construction, hence \descref{C2} holds by \eqref{eq:1N_0} and \descref{C3} holds by \eqref{eq:1N_i}, as desired.
\end{proof}

To establish condition \descref{C5} we need to be able to lift certain finite index subgroups of a maximal parabolic subgroup $P \leqslant G$ to finite index subgroups of $G$ in a controlled way. The next statement shows how a double coset separability assumption can help with this task.

\begin{lemma}
\label{lem:Lemma_1}
    Let \(G\) be a group, \(P, Q \leqslant G\) be subgroups of \(G\) and let \(K \leqslant_f P\) be a finite index subgroup of \(P\), with \(Q \cap P \subseteq K\). 
    If \(KQ\) is separable in \(G\), then there is a finite index subgroup \(M \leqslant_f G\) such that \(Q \subseteq M\) and \(M \cap P \subseteq K\).
\end{lemma}

\begin{proof}
    Let \(P = K \cup Kh_1 \cup \dots \cup Kh_m\), where \(h_1, \dots, h_m \in P \setminus K\).
    Note that \(KQ \cap P = K(Q \cap P) = K\), so \(h_1, \dots, h_m \notin KQ\).
    The double coset \(KQ\) is profinitely closed, so, by Lemma~\ref{lem:sep->large_minx}(a), there exists $N\lhd_f G$ such that
    \[\{h_1, \dots, h_m\} \cap KQN = \emptyset.\]
    Let \(M = QN \leqslant_f G\), so that the above implies $Kh_i \cap M=\emptyset$, for each $i=1,\dots,m$.
    We then have $Q \subseteq M$ and $M \cap P \subseteq K$, as required.
\end{proof}

We are now in  position to prove the main result of this section.

\begin{theorem}
\label{thm:sep->qc_comb} 
    Assume that $G$ is a group generated by a finite set $X$,  $Q,R \leqslant G$ are subgroups of \(G\), and denote $S=Q \cap R$. 
    Let $\mathcal{P}$ be a finite collection of subgroups of $G$ such that for every $P \in \mathcal{P}$ all of the following hold:

    \begin{enumerate}[label={\normalfont (S\arabic*)}]
        \item \label{cond:1} $Q$ and $R$ are separable in $G$;
        \item \label{cond:2} the double coset $PS$ is separable in $G$;
        \item \label{cond:3} for all $K \leqslant_f P$ and $T \leqslant_f Q$, satisfying $S \subseteq T$ and $T \cap P \subseteq K$, the double coset $K T$ is separable in $G$;
        \item \label{cond:4} for all $U \leqslant_f Q \cap P$, with $S \cap P \subseteq U$, the double coset $U(R \cap P)$ is separable in $P$.
    \end{enumerate}

    Then, given arbitrary constants $B,C \ge 0$, there exist finite index subgroups $Q' \leqslant_f Q$ and $R' \leqslant_f R$ such that conditions \descref{C1}--\descref{C5} are all satisfied.

    More precisely, there exists $L \leqslant_f G$, with $S \subseteq L$, such that for any $L' \leqslant_f L$, satisfying $S \subseteq L'$, we can choose $Q'=Q\cap L' \leqslant _f Q$ and there exists $M \leqslant_f L'$, with $Q' \subseteq M$, such that for any $M' \leqslant_f M$, satisfying $Q' \subseteq M'$, we can choose $R'=R \cap M' \leqslant_f R$.
\end{theorem}

\begin{proof}
    The idea is that assumption \ref{cond:1} will take care of condition \descref{C2}, \ref{cond:2} will take care of \descref{C3}, and \ref{cond:3}, \ref{cond:4} will take care of \descref{C5}. 
    The subgroups $Q'$ and $R'$ will satisfy $Q'=Q \cap M'$ and $R=R \cap M'$, for some $M' \leqslant_f G$, with $S \subseteq M'$, which will immediately imply \descref{C1} and \descref{C4}.

    Let $\mathcal{P}=\{P_1,\dots,P_k\}$.
    Arguing just like in the proof of Proposition~\ref{prop:sep->C2-C3} (using the assumptions \ref{cond:1} and \ref{cond:2}), we can find finite index normal subgroups $N_i \lhd_f G$, $i=0,\dots,k$, such that
    \[
        \minx(QS N_0 Q \setminus Q) \ge B,~ \minx(RS N_0 R \setminus R) \ge B~\text{ and }
    \]
    \[
        \minx(P_i S N_i \setminus P_i S) \ge C,~\text{ for each } i=1,\dots,k.
    \]

    We can now define a finite index subgroup $L \leqslant_f G$ by $L=\bigcap_{i=0}^k SN_i$.
    Note that $S \subseteq L$ by construction, and for each $i \in \{1,\dots,k\}$ we have
    \begin{equation}
    \label{eq:L}
        \minx(QLQ \setminus Q) \ge B,~\minx(RLR \setminus R) \ge B~\text{ and }~\minx(P_i L \setminus P_i S) \ge C.
    \end{equation}

    Choose an arbitrary finite index subgroup $L' \leqslant_f L$, with $ S \subseteq L'$, and define $Q'=Q \cap L'$, so that $S \leqslant Q' \leqslant_f Q$.

    To construct $R' \leqslant_f R$, consider any $i \in \{1,\dots,k\}$ and denote $Q_i=Q \cap P_i$, $R_i=R \cap P_i$ and $Q_i'=Q' \cap P_i \leqslant_f Q_i$.
    Choose some elements $a_{i1}, \dots,a_{i n_i} \in Q_i$ such that $Q_i=\bigsqcup_{j=1}^{n_i} a_{ij} Q_i'$. Assumption \ref{cond:4} implies that the subset $Q_i'R_i$ is separable in $P_i$, hence, by claim (c) of Lemma~\ref{lem:sep->large_minx}, there exists $F_i \lhd_f P_i$ such that
    \begin{equation}
    \label{eq:F_i}
        \minx \Bigl(a_{ij}Q'_i R_i F_i \setminus a_{ij}Q'_i R_i\Bigr) \ge C, ~\text{ for } j=1,\dots,n_i.
    \end{equation}
    Define
    $K_i=Q_i'F_i \leqslant_f P_i$. Then $Q' \cap P_i=Q_i' \subseteq K_i$ and $a_{ij} K_i R_i=a_{ij}Q'_i R_i F_i $, for each $j=1,\dots,n_i$.
    Therefore, from \eqref{eq:F_i} we can deduce that
    \begin{equation}
    \label{eq:K_i}
        \minx\Bigl(a_{ij}K_i R_i \setminus a_{ij}Q'_i R_i\Bigr) \ge C, ~\text{ for all } j=1,\dots,n_i.
    \end{equation}

    By assumption \ref{cond:3}, the double coset $K_iQ'$ is separable in $G$, so we can apply Lemma~\ref{lem:Lemma_1} to find $M_i \leqslant_f G$ such that $Q' \subseteq M_i$ and $M_i \cap P_i \subseteq K_i$.

    We now let $\displaystyle M=\bigcap_{i=1}^k M_i \cap L'$ and observe that $Q' \leqslant M \leqslant_f L'$ and $M \cap P_i \subseteq K_i$ for each $i \in \{1,\dots,k\}$. Inequality \eqref{eq:K_i} yields
    \begin{equation}
    \label{eq:M}
        \minx\Bigl(a_{ij}(M \cap P_i) R_i \setminus a_{ij}Q'_i R_i\Bigr) \ge C, ~\text{ for all } i=1,\dots,k \text{ and } j=1,\dots, n_i.
    \end{equation}

    We can now choose an arbitrary finite index subgroup $M' \leqslant_f M$, with $Q' \subseteq M'$, and define $R'=R \cap M'$. Observe that $M' \leqslant_f G$, by construction, hence $R' \leqslant_f R$.

    Let us check that the subgroups $Q'$ and $R'$ obtained above satisfy conditions \descref{C1}--\descref{C5}. Indeed, by construction, $S=Q \cap R \subseteq Q'$, so $S \subseteq R \cap M'=R'$, hence
    \[
        S \subseteq Q' \cap R' \subseteq Q \cap R=S,
    \]
    thus \descref{C1} holds. We also have $Q'=Q \cap L'=Q \cap M'$, as $Q' \subseteq M' \subseteq L'$, hence
    \[
        Q' \subseteq Q \cap \langle Q',R'\rangle \subseteq Q \cap M' =Q',\]
    thus  $Q \cap \langle Q',R'\rangle = Q'$. After intersecting both sides of the latter equation with an arbitrary  $P \in \mathcal{P}$, we get $Q_P \cap \langle Q',R'\rangle=Q'_P$, hence
    \[
        Q_P' \subseteq Q_P \cap \langle Q_P',R_P'\rangle \subseteq Q_P \cap \langle Q',R'\rangle = Q_P',
    \]
    thus $Q_P \cap \langle Q_P',R_P'\rangle=Q_P'$. Similarly, $R_P \cap \langle Q_P',R_P'\rangle=R_P'$, so condition \descref{C4} is satisfied.

    Conditions \descref{C2} and \descref{C3} hold by \eqref{eq:L}, because $Q', R' \subseteq L$ by construction.

    To prove \descref{C5}, take $P_i \in \mathcal{P}$ for any $i \in \{1,\dots,k\}$, and denote $Q_i=Q \cap P_i$, $Q_i'=Q' \cap P_i$, $R_i=R \cap P_i$ and $R_i'=R' \cap P_i$, as before.
    For any $q \in Q_i$ there exists $j \in \{1,\dots,n_i\}$ such that $q \in a_{ij}Q_i'$. It follows that
   \begin{equation}\label{eq:q->a_ij}
      q \langle Q_i',R_i' \rangle R_i=a_{ij} \langle Q_i',R_i' \rangle R_i~\text{ and }~q Q_i' R_i=a_{ij}Q_i'R_i.  
   \end{equation}

    Since $\langle Q_i',R_i' \rangle \leqslant M \cap P_i$, we can combine \eqref{eq:q->a_ij} with \eqref{eq:M} to deduce that
    \[
        \minx\Bigl(q\langle Q_i',R_i' \rangle R_i \setminus qQ'_i R_i\Bigr) \ge C,
    \]
    which establishes condition \descref{C5}.
    Thus the proof is complete.
\end{proof}


\section{Double coset separability in amalgamated free products}
\label{sec:dcs_in_amalgams}
In this section we  develop a method for establishing the separability assumptions \ref{cond:2} and \ref{cond:3} of Theorem~\ref{thm:sep->qc_comb} using amalgamated products. The idea is that when $G$ is a relatively hyperbolic group, $P$ is a maximal parabolic subgroup and $Q$ is a relatively quasiconvex subgroup of $G$, we can apply the combination theorem of Mart\'{i}nez-Pedroza (Theorem~\ref{thm:M-P_comb}) to find a finite index subgroup $H \leqslant_f P$ such that $A=\langle H,Q \rangle \cong H*_{H \cap Q} Q$, so proving the separability of $PQ$ in $G$ can be reduced to proving the separability of $HQ$ in the amalgamated free product $A$.

The next proposition gives a new criterion for showing separability of double cosets in amalgamated free products. This criterion may be of independent interest.  

\begin{proposition}
\label{prop:dc_in_am}
    Let $A=B*_D C$ be an amalgamated free product, where we consider $B$, $C$ and $D$ as subgroups of $A$ with $B \cap C=D$. 
    Suppose that $D$ is separable in $A$, and $U \subseteq B$, $V \subseteq C$ are arbitrary subsets.

    If the product $UD$ (respectively, $DV$) is separable in $A$ then the product $UC$ (respectively, $BV$) is separable in $A$.
\end{proposition}

\begin{proof} 
    We will prove the statement in the case of $UC$, as the other case is similar.

    If $U= \emptyset$ then $UC=\emptyset$, so we can suppose that $U$ is non-empty. 
    Take any $u \in U$.
    According to Remark~\ref{rem:sep_props}, without loss of generality we can replace $U$ with $u^{-1}U$ to assume that $1 \in U$.

    Consider any element $g \in A \setminus UC$; since $1 \in U$, we deduce that $g \notin C$. We will construct a homomorphism from $A$ to a finite group $L$ which separates the image of $g$ from the image of $UC$.

    Since $g \notin D$, it has a reduced form $g=x_1x_2 \dots x_k$, where $x_i$ belongs to one of the factors $B$, $C$, for each $i$, consecutive elements $x_i$, $x_{i+1}$ belong to different factors, and $x_i \notin D$ for all $i=1,\dots,k$ (see \cite[p. 187]{LS}).

    Since $D$ is separable in $A$, by Lemma~\ref{lem:sep->large_minx}(a) there is a finite group $M$ and a homomorphism $\varphi: A \to M$ such that
    \begin{equation}
    \label{eq:notinD}
        \varphi(x_i) \notin \varphi(D) ~\text{ in } M, \text{ for every } i=1,\dots,k.
    \end{equation}

    Denote by $\overbar{B}$, $\overbar{C}$ and $\overbar{D}$ the $\varphi$-images if $B$, $C$ and $D$ in $M$ respectively. 
    We can then consider the amalgamated free product $\overbar{A}=\overbar{B}*_{\overbar{D}} \overbar{C}$, together with the natural homomorphism $\psi:A \to \overbar{A}$, which is compatible with $\varphi$ on $B$ and $C$ (in other words, $\psi|_B=\varphi|_B$ and $\psi|_C=\varphi|_C$). 
    It follows that $\varphi$ factors through $\psi$.
    That is, $\varphi=\overbar{\varphi} \circ \psi$, where $\overbar\varphi:\overbar{A} \to M$ is the natural homomorphism extending the embeddings of $\overbar{B}$ and $\overbar{C}$ in $M$. 
    Equation \eqref{eq:notinD} now implies that
    \begin{equation}
    \label{eq:notinDbar}
        \psi(x_i) \notin \overbar{D} ~\text{ in } \overbar{A}, \text{ for every } i=1,\dots,k.
    \end{equation}

    Denote $\overline{x}_i=\psi(x_i) \in \overbar{A}$, $i=1,\dots,k$. In view of \eqref{eq:notinDbar}, $\psi(g)=\overline{x}_1\dots \overline{x}_k$ is a reduced form in the amalgamated free product $\overbar{A}$.
    We will now consider several cases.

    \medskip
    \underline{\emph{Case 1:}} assume that $k \ge 3$.
    Then the above reduced form for $\psi(g)$ has length $k \ge 3$, so by the normal form theorem for amalgamated free products \cite[Theorem IV.2.6]{LS}, it cannot be equal to an element from $\psi(UC)=\psi(U)\overbar{C}\subseteq \overbar{B}\overbar{C}$, which would necessarily have a reduced form of length at most $2$ in $\overbar{A}$. 
        Therefore $\psi(g) \notin \psi(UC)$ in $\overbar A$.

        Since $\overbar{B}$ and $\overbar C$ are finite groups, their amalgamated free product $\overbar A$ is residually finite (in fact, $\overbar A$ is a virtually free group -- see \cite[Proposition 2.6.11]{Serre}), so the finite subset $\psi(UC)$ is closed in the profinite topology on $\overbar A$.
        Hence there is a finite group $L$ and a homomorphism $\eta:\overbar A \to L$ such that $\eta(\psi(g)) \notin \eta(\psi(UC))$ in $L$. 
        The composition $\eta\circ \psi:A \to L$ is the required homomorphism separating the image of $g$ from the image of $UC$, and the consideration of Case~1 is complete.

    \medskip
    \underline{\emph{Case 2:}} suppose that $k=2$, $x_1\in C\setminus D$ and $x_2 \in B\setminus D$.
        Then $\overline{x}_1 \in \overbar{C}\setminus\overbar{D}$ and $\overline{x}_2 \in \overbar{B}\setminus\overbar{D}$ by \eqref{eq:notinDbar}, so that $\psi(g)=\overline{x}_1 \overline{x}_2$ is a reduced form of length $2$ in $\overbar A$. 
        Again, the normal form theorem for amalgamated free products implies that $\psi(g) \notin \overbar{B}\overbar{C}$ in $\overbar{A}$, hence $\psi(g) \notin \psi(U C)$ and we can find the required finite quotient $L$ of $A$ as in Case 1.

    \medskip
    \underline{\emph{Case 3:}} $g=bc$, where $b \in B\setminus UD$ and $c \in C$ (here we allow $c \in D$, so this case also covers the situation when $k=1$).

        This is the only case where we need to use the assumption that $UD$ is separable in $A$. This assumption implies that we can find a finite group $M$ and a homomorphism  $\varphi:A \to M$ satisfying
        \begin{equation*}
            \varphi(b) \notin \varphi(UD)   ~ \text{ in } M.
        \end{equation*}

        As above, we can construct the amalgamated free product $\overbar{A}=\overbar{B}*_{\overbar{D}} \overbar{C}$, together with the natural homomorphism $\psi:A \to \overbar{A}$, such that $\varphi$ factors through $\psi$.
        It follows that
        \begin{equation}
        \label{eq:notinUDbar}
            \psi(b) \notin \psi(UD)= \psi(U) \overbar{D} ~\text{ in } \overbar A.
        \end{equation}

        Observe that $\psi(g) \notin \psi(UC)=\psi(U) \overbar{C}$ in $\overbar{A}$. 
        Indeed, otherwise we would have 
        \[
            \psi(b)=\psi(g)\psi(c^{-1}) \in \psi(U)\overbar{C} \cap \overbar{B}=\psi(U)(\overbar{C} \cap \overbar{B})=\psi(U)\overbar{D},
        \] 
        which would contradict \eqref{eq:notinUDbar} (in the first equality we used the fact that $\overbar B$ is a subgroup of $\overbar{A}$ containing the subset $\psi(U)$).
        We can now argue as in Case 1 above to find a homomorphism from $A$ to a finite group $L$ separating the image of $g$ from the image of $UC$.

    It is not hard to see that since $g \notin UC$ in $A$, the above three cases cover all possibilities, hence the proof is complete.
\end{proof}

In the next two corollaries we assume that $A=B*_D C$  is the amalgamated free product of its subgroups $B,C$, with $B \cap C=D$.

\begin{corollary}
\label{cor:sep_double_coset_amalg} 
    Suppose that $D$ is a separable subgroup in $A$. Then $B$, $C$ and $BC$ are all separable in $A$.
\end{corollary}

\begin{proof}
    The separability of $C$ and $B$ in $A$ follows from Proposition~\ref{prop:dc_in_am}, after choosing $U=\{1\}$ and $V=\{1\}$.

    The separability of $BC$ is also a consequence of Proposition~\ref{prop:dc_in_am}, where we take $U=B$ (so that $UD=BD=B$).
\end{proof}

We will not need the next corollary in this paper, but it may be of independent interest and can be used to strengthen some of the statements proved in Section~\ref{sec:dc_when_one_is_parab}.

\begin{corollary}
\label{cor:sep_triple_coset_amalg} 
    Suppose that $U \subseteq B$, $V \subseteq C$ are subsets such that  $UD$ and $DV$ are separable in $A$. Then the triple product $UDV$ is separable in $A$.
\end{corollary}

\begin{proof} 
    If either $U$ or $V$ are empty then $UDV$ is empty, and, hence, separable in $A$. 
    Thus we can suppose that there exist some elements $u \in U$ and $v \in V$. By Remark~\ref{rem:sep_props}. the subsets $u^{-1}UD \subseteq B$ and $DVv^{-1} \subseteq C$ are separable in $A$. 
    Since both of them contain $D$, we see that $D=u^{-1}UD \cap DVv^{-1}$,  thus $D$ is separable in $A$.

    By Proposition~\ref{prop:dc_in_am}, the products $UC$ and $BV$ are separable in $A$, so the statement follows from the observation that
    \[
        UC \cap BV=UDV~\text{ in } A. \qedhere
    \]
\end{proof}

In the case when $U$ and $V$ are subgroups, the above corollary shows that we can use separability of double cosets $UD$ and $DV$ to deduce separability of the triple coset $UDV$. Moreover, if both $U$ and $V$ are subgroups containing $D$, Corollary~\ref{cor:sep_triple_coset_amalg} implies that the double coset $UV=UDV$ is separable in $A$, as long as $U$ and $V$ are separable in $A$.


\section{Separability of double cosets when one factor is parabolic} 
\label{sec:dc_when_one_is_parab}
Throughout this section we will assume that $G$ is group generated  by a finite subset $X$ and hyperbolic relative to a collection of peripheral subgroups \(\{H_\nu \mid \nu \in \Nu\}\), with \(|\Nu|<\infty\).

Our goal in this section will be to establish separability of double cosets required by conditions \ref{cond:2} and \ref{cond:3} of Theorem~\ref{thm:sep->qc_comb}.
All statements in this section will assume that finitely generated relatively quasiconvex subgroups of $G$ are separable -- that is, $G$ is QCERF (see Definition~\ref{def:QCERF}).

\begin{lemma}
\label{lem:SA1->prof_top_on_qc_sbgps_is_induced}
    Suppose that $G$ is QCERF. If $A$ is a finitely generated relatively quasiconvex subgroup of $G$ then every subset of $A$ which is closed in $\pt(A)$ is also closed in $\pt(G)$.
\end{lemma}

\begin{proof}  
    By Lemma~\ref{lem:props_of_qc_sbgps} every subgroup of finite index in $A$ is finitely generated and relatively quasiconvex, hence it is separable in $G$ as $G$ is QCERF. 
    The claim of the lemma now follows from Lemma~\ref{lem:induced_top}(b).
\end{proof}

The next statement is essentially a corollary of the combination theorem of Mart\'{i}nez-Pedroza (Theorem~\ref{thm:M-P_comb}).

\begin{proposition}
\label{prop:M-P_comb_in_prof_terms} 
    Suppose that $G$ is QCERF. Let $P$ be a maximal parabolic subgroup of $G$, let $Q \leqslant G$ be a finitely generated relatively quasiconvex subgroup and let $D=P \cap Q$. 
    Then there exists a finite index subgroup $H \leqslant_f P$ such that all of the following properties hold:
    \begin{itemize}
        \item $ H \cap Q=D$;
        \item the subgroup $A=\langle H,Q \rangle$ is relatively quasiconvex in $G$;
        \item $A$ is naturally isomorphic to $H*_{D} Q$;
        \item $D$ is separable in $A$;
        \item every subset of $A$ which is closed in $\pt(A)$ is also closed in $\pt(G)$.
    \end{itemize}
\end{proposition}

\begin{proof}
    Let $C \ge 0$ be the constant provided by Theorem~\ref{thm:M-P_comb}, applied to the maximal parabolic subgroup $P$ and the relatively quasiconvex subgroup $Q$. 
    By QCERF-ness, $Q$ is separable in $G$, so by Lemma~\ref{lem:sep->large_minx} there exists $N \lhd_f G$ such that $\minx(QN \setminus Q) \ge C$. 
    Therefore, after setting $H=P \cap QN \leqslant_f P$, we get $\minx(H \setminus D)=\minx(H \setminus Q) \ge C$.

    Note that since $D=P \cap Q \subseteq H \subseteq P$, we have $H \cap Q=D$. 
    Hence we can apply Theorem~\ref{thm:M-P_comb} to conclude that $A=\langle H,Q \rangle$ is relatively quasiconvex in $G$ and is naturally isomorphic to the amalgamated free product $H*_{D} Q$.

    Recall, from Lemma~\ref{lem:fg_qc_int_parab_is_fg} and Corollary~\ref{cor:parab->qc}, that $P$ is finitely generated and relatively quasiconvex in $G$, hence it is separable in $G$ by QCERF-ness. 
    It follows that $D=P \cap Q$ is separable in $G$, which implies that it is separable in $A$ by Lemma~\ref{lem:induced_top}.

    Observe that $H$ and $Q$ are both finitely generated, hence $A$ is finitely generated and relatively quasiconvex in $G$. 
    Therefore Lemma~\ref{lem:SA1->prof_top_on_qc_sbgps_is_induced} yields the last assertion of the proposition, that every subset of $A$ which is closed in $\pt(A)$ is also closed in $\pt(G)$.
\end{proof}

By combining Proposition~\ref{prop:M-P_comb_in_prof_terms} with Proposition~\ref{prop:dc_in_am} we obtain the first double coset separability result when one of the factors is parabolic and the other one is finitely generated and relatively quasiconvex.

\begin{proposition} 
\label{prop:parab_times_fg_qc-sep}
    Assume that $G$ is QCERF. Let $P$ be a maximal parabolic subgroup of $G$, let $R \leqslant G$ be a finitely generated relatively quasiconvex subgroup of $G$. 
    Suppose that $D \leqslant P$ is a subgroup satisfying the following condition:
    \begin{equation}
    \label{eq:fin_ind_if_D}
        \text{ for each } U \leqslant_f D \text{ the double coset } U(P \cap R) \text{ is separable in } P.
    \end{equation}
    Then the double coset $DR$ is separable in $G$.
\end{proposition}

\begin{proof} 
    According to Proposition~\ref{prop:M-P_comb_in_prof_terms}, there exists $H \leqslant_f P$ such that the subgroup $A=\langle H,R \rangle$ is naturally isomorphic to the amalgamated free product $H*_{E} R$, where $E=P \cap R=H \cap R$ is separable in $A$, and every closed subset from $\pt(A)$ is separable in $G$.

    Denote $U=D \cap H \leqslant_f D$. 
    By assumption \eqref{eq:fin_ind_if_D}, $UE$ is separable in $P$. 
    Since $P$ is finitely generated and relatively quasiconvex in $G$, we can conclude that $UE$ is separable in $G$ by Lemma~\ref{lem:SA1->prof_top_on_qc_sbgps_is_induced}. 
    As $UE \subseteq A \leqslant G$, $UE$ will also be closed in $\pt(A)$, so we can apply Proposition~\ref{prop:dc_in_am} to deduce that the double coset $UR$ is closed in $\pt(A)$. 
    It follows that this double coset is separable in $G$ and, since $U \leqslant_f D$, Lemma~\ref{lem:fi_dc} implies that $DR$ is separable in $G$, as desired.
\end{proof}

We can now prove that assumption \ref{cond:3} of Theorem~\ref{thm:sep->qc_comb} holds as long as the relatively hyperbolic group $G$ is QCERF.

\begin{corollary}
\label{cor:fi_in_parab_times_fgqc_is_sep} 
    Suppose that $G$ is QCERF, $P$ is a maximal parabolic subgroup of $G$ and $Q \leqslant G$ is a finitely generated relatively quasiconvex subgroup. 
    Then for all finite index subgroups $K \leqslant_f P$ and $T \leqslant_f Q$ the double coset $KT$ is separable in $G$.
\end{corollary}

\begin{proof} 
    Note that $T$ is finitely generated and relatively quasiconvex in $G$ by Lemma~\ref{lem:props_of_qc_sbgps}. 
    Hence, to apply Proposition~\ref{prop:parab_times_fg_qc-sep} we simply need to check that for any $U \leqslant_f K$ the double coset $U(P \cap T)$ is separable in $P$. 
    The latter is true because $U(P \cap T)$ is a basic closed set in $\pt(P)$, being a finite union of right cosets to $U \leqslant_f P$. 
    Therefore $KT$ is separable in $G$ by Proposition~\ref{prop:parab_times_fg_qc-sep}.
\end{proof}

The proof of assumption \ref{cond:2} of Theorem~\ref{thm:sep->qc_comb} is slightly more involved because the intersection of two finitely generated relatively quasiconvex subgroups need not be finitely generated.

\begin{proposition}
\label{prop:parab_times_S_is_sep} 
    Let $P$ be a maximal parabolic subgroup of $G$, let $Q,R \leqslant G$ be finitely generated relatively quasiconvex subgroups, let $S=Q \cap R$ and $D=P \cap Q$. 
    Suppose that $G$ is QCERF and condition \eqref{eq:fin_ind_if_D} is satisfied. 
    Then the double coset $PS$ is separable in $G$.
\end{proposition}

\begin{proof} 
    Proposition~\ref{prop:parab_times_fg_qc-sep} tells us that the double coset $DR$ is separable in $G$, and \(G\) is QCERF so \(Q\) is separable in \(G\). 
    Now, observe that $DR \cap Q=D(R \cap Q)=DS$, because $D \leqslant Q$. 
    It follows that the double coset \(DS\) is separable in \(G\).

    According to Proposition~\ref{prop:M-P_comb_in_prof_terms}, there exists a finite index subgroup $H \leqslant_f P$ such that $H \cap Q=D$, $A=\langle H, Q\rangle \cong H*_D Q$, $D$ is separable in $A$ and every closed subset in $\pt(A)$ is closed in $\pt(G)$.
    The double coset $DS$ is separable in $A$ by Lemma~\ref{lem:induced_top}, so $HS$ is closed in $\pt(A)$ by Proposition~\ref{prop:dc_in_am}. 
    It follows that $HS$ is closed in $\pt(G)$, which implies that the double coset $PS$ is separable in $G$ by Lemma~\ref{lem:fi_dc}. 
    Thus the proof is complete.
\end{proof}


\section{Quasiconvexity of a virtual join from separability properties}
\label{sec:sep->qc}

In this section we will prove Theorems~\ref{thm:sep->qc_intro} and \ref{thm:sep->qc_for_ab_parab} from the Introduction. The latter follows from the following result and the observation that a finite index subgroup of a relatively quasiconvex subgroup is itself relatively quasiconvex (see  Lemma~\ref{lem:props_of_qc_sbgps}).

\begin{theorem}
\label{thm:sep->qc_for_ab_parab-detailed}
     Let \(G\) be a group generated by a finite set \(X\) and hyperbolic relative to a finite collection of abelian subgroups. Assume that $G$ is QCERF. 
     If \(Q, R \leqslant G\) are relatively quasiconvex subgroups and \(S = Q \cap R\) then for every \(A \geq 0\) there exists a finite index subgroup $L \leqslant_f G$, with $S \subseteq L$, such that properties \descref{P1}--\descref{P3}  from Subsection~\ref{subsec:3.1} hold for arbitrary subgroups $Q' \leqslant Q \cap L$ and $R' \leqslant R \cap L$ satisfying $Q' \cap R'=S$.
 \end{theorem}

\begin{proof}
    By combining the assumptions with Lemma~\ref{lem:fg_qc_int_parab_is_fg}, we know that maximal parabolic subgroups of $G$ are finitely generated abelian groups. 
    Since such groups are slender, all relatively quasiconvex subgroups of $G$ are finitely generated (see \cite[Corollary~9.2]{HruskaRHCG}). 
    Moreover, finitely generated abelian groups are LERF, and hence, they are double coset separable (because the product of two subgroups is again a subgroup). 
    Therefore the double coset $PS$ is separable in $G$ for any maximal parabolic subgroup $P \leqslant G$ by Proposition~\ref{prop:parab_times_S_is_sep}.

    In view of Proposition~\ref{prop:sep->C2-C3}, for any finite collection $\mathcal P$, of maximal parabolic subgroups of $G$, and any $B, C \ge 0$ there exists $L \leqslant_f G$, with $S \subseteq L$, such that any subgroups $Q' \leqslant Q \cap L$ and $R' \leqslant R \cap L$ satisfy conditions \descref{C1}--\descref{C3}, as long as $Q'\cap R'=S$. 
    Remark~\ref{rem:ab_periph->C4_and_C5} tells us that these subgroups automatically satisfy conditions \descref{C4} and \descref{C5}. 
    Thus we can obtain  the desired statement by applying Theorem~\ref{thm:metric_qc}.
\end{proof}

\begin{corollary}
\label{cor:virt_ab_periph}
    Suppose that \(G\) is a QCERF group generated by a finite subset \(X\) and hyperbolic relative to a finite family $\{H_\nu\mid \nu \in \Nu\}$ of virtually abelian subgroups. 
    Let \(Q, R \leqslant G\) be relatively quasiconvex subgroups and let \(S = Q \cap R\).
    Then there exists $L \leqslant_f G$ such that if $Q' \leqslant Q \cap L$ and $R' \leqslant R \cap L$ are relatively quasiconvex subgroups of $G$ satisfying $Q' \cap R'=S \cap L$ then the subgroup $\langle Q',R'\rangle$ is also relatively quasiconvex in $G$.
\end{corollary}

\begin{proof}
    By the assumptions for each $\nu \in \Nu$ there exists a finite index abelian subgroup $K_\nu\leqslant_f H_\nu$. 
    Since $G$ is QCERF, each $K_\nu$ is separable in $G$ (it is finitely generated by Lemma~\ref{lem:fg_qc_int_parab_is_fg} and it is relatively quasiconvex by Corollary~\ref{cor:parab->qc}). 
    Thus, in view of Lemma~\ref{lem:Lemma_0}, for every $\nu \in \Nu$ there exists $L_\nu \leqslant_f G$ such that $L_\nu \cap H_\nu=K_\nu$.

    Since $|\Nu|<\infty$, the intersection $\bigcap_{\nu \in \Nu} L_\nu $ has finite index in $G$, hence it contains a finite index normal subgroup $G_1 \lhd_f G$. 
    Note that for any $g \in G$ and any $\nu \in \Nu$ we have
    \begin{equation}
    \label{eq:G_1}
        G_1 \cap gH_\nu g^{-1}= g(G_1 \cap H_\nu)g^{-1} \subseteq g(L_\nu \cap H_\nu)g^{-1} = g K_\nu g^{-1},
    \end{equation}
    where the first equality follows from the normality of \(G_1\), the middle inclusion follows from the fact that \(G_1 \subseteq L_\nu\), and the last equality is due to the fact that \(L_\nu \cap H_\nu = K_\nu\).
    By Lemma~\ref{lem:props_of_qc_sbgps}, $G_1$ is finitely generated and relatively quasiconvex in $G$, hence, by \cite[Theorem 9.1]{HruskaRHCG} it is hyperbolic relative to representatives of $G_1$-conjugacy classes of the intersections $G_1 \cap g H_\nu g^{-1}$, $g \in G$. 
    Thus, in view of \eqref{eq:G_1}, all peripheral subgroups in $G_1$ are abelian.

    By \cite[Corollary 9.3]{HruskaRHCG}, a subgroup of $G_1$ is relatively quasiconvex in $G_1$ (with respect to the above family of peripheral subgroups) if and only if it is relatively quasiconvex in $G$. 
    Therefore $G_1$ is QCERF and $Q_1=Q \cap G_1 \leqslant_f Q$, $R_1=R \cap G_1 \leqslant_f R$ are finitely generated relatively quasiconvex subgroups of $G_1$ by Lemma~\ref{lem:props_of_qc_sbgps}. 
    After denoting $S_1=S \cap G_1=Q_1 \cap R_1$, we can apply Theorem~\ref{thm:sep->qc_for_ab_parab} to find a finite index subgroup $L \leqslant_f G_1$ such that $S_1 \subseteq L$ (thus, $S_1=S \cap L$) and the subgroup $\langle Q',R'\rangle$ is relatively quasiconvex in $G_1$, for arbitrary $Q' \leqslant Q_1 \cap L=Q \cap L$ and $R' \leqslant R_1 \cap L=R \cap L$ satisfying $Q' \cap R'=Q_1 \cap R_1=S_1$.
    We can use \cite[Corollary 9.3]{HruskaRHCG} again to deduce that  $\langle Q',R'\rangle$ is relatively quasiconvex in $G$.
\end{proof}

The following collects the results of the previous sections, allowing us to find subgroups \(Q'\) and \(R'\) to which Theorem~\ref{thm:metric_qc} can be applied.

\begin{proposition}
\label{prop:sep->C1-C5}
    Let \(G\) be a finitely generated QCERF relatively hyperbolic group with double coset separable peripheral subgroups, and let \(Q\) and \(R\) be finitely generated relatively quasiconvex subgroups.
    Then for any \(B \geq 0, C \geq 0,\) and finite family \(\mathcal{P}\) of maximal parabolic subgroups of \(G\), there are finite index subgroups \(Q' \leqslant_f Q\) and \(R' \leqslant_f R\) satisfying \descref{C1}-\descref{C5} with constants \(B\) and \(C\) and family \(\mathcal{P}\).
    
    More precisely, writing \(S = Q \cap R\), there exists $L \leqslant_f G$ with $S \subseteq L$ such that for any $L' \leqslant_f L$ satisfying $S \subseteq L'$, we can choose $Q'=Q\cap L' \leqslant _f Q$ and there exists $M \leqslant_f L'$ with $Q' \subseteq M$ such that for any $M' \leqslant_f M$ satisfying $Q' \subseteq M'$, we can choose $R'=R \cap M' \leqslant_f R$.
\end{proposition}

\begin{proof} 
    We check that all the assumptions of Theorem~\ref{thm:sep->qc_comb} are satisfied for every $P \in \mathcal{P}$. 
    Indeed, assumption~\ref{cond:1} holds because $G$ is QCERF and assumption \ref{cond:3} is true by Corollary~\ref{cor:fi_in_parab_times_fgqc_is_sep}.

    Note that the subgroups $D=Q \cap P$ and $R \cap P$ are finitely generated by Lemma~\ref{lem:fg_qc_int_parab_is_fg}, hence condition \eqref{eq:fin_ind_if_D} follows from the double coset separability of $P$, thus \ref{cond:4} is satisfied. 
    Finally, assumption \ref{cond:2} holds by Proposition~\ref{prop:parab_times_S_is_sep}.

    The statement now follows by applying Theorem~\ref{thm:sep->qc_comb}.
\end{proof}

\begin{theorem}
\label{thm:sep->qc_intro-detailed}
     Let \(G\) be a group generated by a finite set \(X\) and hyperbolic relative to a finite collection of subgroups $\{H_\nu \mid \nu \in \Nu\}$. 
     Suppose that $G$ is QCERF and $H_\nu$ is double coset separable, for each $\nu \in \Nu$. 
     If \(Q, R \leqslant G\) are finitely generated relatively quasiconvex subgroups and  \(S = Q \cap R\) then for every \(A \geq 0\) there exist finite index subgroups $Q'\leqslant_f Q$ and $R' \leqslant_f R$ which satisfy properties \descref{P1}--\descref{P3}.

    More precisely, there exists $L \leqslant_f G$ with $S \subseteq L$ such that for any $L' \leqslant_f L$ satisfying $S \subseteq L'$, we can choose $Q'=Q\cap L'  \leqslant _f Q$ and there exists $M \leqslant_f L'$ with $Q' \subseteq M$ such that for any $M' \leqslant_f M$ satisfying $Q' \subseteq M'$, we can choose $R'=R \cap M' \leqslant_f R$.
 \end{theorem}

\begin{proof} 
    Let $\mathcal{P}$ be the finite collection of maximal parabolic subgroups of $G$ provided by Theorem~\ref{thm:metric_qc}. 
    The statement follows immediately from a combination of Theorem~\ref{thm:metric_qc} with Proposition~\ref{prop:sep->C1-C5}.
\end{proof}

Recall that \(Q\) and \(R\) are said to have almost compatible parabolics if for every maximal parabolic subgroup \(P \leqslant G\), either \(Q \cap P \preccurlyeq R \cap P\) or \(R \cap P \preccurlyeq Q \cap P\).
We find that in the case when $Q$ and $R$ have almost compatible parabolics, it is actually not necessary to assume that the peripheral subgroups are double coset separable:

\begin{theorem}
\label{thm:almost_compat->qc_comb} 
    Suppose that $G$ is a finitely generated QCERF relatively hyperbolic group, $Q,R \leqslant G$ are finitely generated relatively quasiconvex subgroups with almost compatible parabolics and  \(S = Q \cap R\). 
    Then for every \(A \geq 0\) there exist finite index subgroups $Q'\leqslant_f Q$ and $R' \leqslant_f R$ which satisfy properties \descref{P1}--\descref{P3}.

    More precisely, there exists $L \leqslant_f G$, with $S \subseteq L$, such that for any $L' \leqslant_f L$, satisfying $S \subseteq L'$, we can choose $Q'=Q\cap L' \leqslant _f Q$ and there exists $M \leqslant_f L'$, with $Q' \subseteq M$, such that for any $M' \leqslant_f M$, satisfying $Q' \subseteq M'$, we can choose $R'=R \cap M' \leqslant_f R$.
\end{theorem}

\begin{proof} 
    As before, we will be verifying the assumptions of Theorem~\ref{thm:sep->qc_comb}.
    Let $P$ be an arbitrary maximal parabolic subgroup of $G$.
    Assumption~\ref{cond:1} follows from the QCERF-ness of $G$ and assumption \ref{cond:3} follows from Corollary~\ref{cor:fi_in_parab_times_fgqc_is_sep}.

    Let $D=Q \cap P$ and $U \leqslant_f D$. Since $Q$ and $R$ have almost compatible parabolics and $Q \cap P \preccurlyeq U$, we know that either $U \preccurlyeq R \cap P$ or $R \cap P \preccurlyeq U$. 
    Note that both $U$ and $R \cap P$ are finitely generated by Lemma~\ref{lem:fg_qc_int_parab_is_fg} and relatively quasiconvex by Corollary~\ref{cor:parab->qc}, so they are separable because $G$ is QCERF.  
    Lemma~\ref{lem:V_sep_and_U_leq_V->UV-sep} now implies that the double coset $U(R \cap P)$ is separable in $G$, thus condition \eqref{eq:fin_ind_if_D} is satisfied by Lemma~\ref{lem:induced_top}. 
    This shows that assumption \ref{cond:4} of Theorem~\ref{thm:sep->qc_comb} is satisfied; furthermore, assumption \ref{cond:2} holds by Proposition~\ref{prop:parab_times_S_is_sep}.

    We can now deduce the theorem by combining Theorem~\ref{thm:metric_qc} with Theorem~\ref{thm:sep->qc_comb}.
\end{proof}


\section{Separability of double cosets in QCERF relatively hyperbolic groups}
\label{sec:double_coset_sep}
In this section we prove Corollary~\ref{cor:double_cosets_sep} from the Introduction.

\begin{proof}[Proof of Corollary~\ref{cor:double_cosets_sep}] 
    Let $X$ be a finite generating set of $G$. Consider any $g \in G \setminus QR$, and set $A=|g|_X+1$.
    By Theorem \ref{thm:sep->qc_intro-detailed} there are subgroups \(Q' \leqslant_f Q\), \(R' \leqslant_f R\) satisfying properties \descref{P1} and \descref{P3}. 
    The latter property, combined with the definition of $A$, implies that  \(g \notin Q \langle Q', R' \rangle R\).

    On the other hand, property \descref{P1} tells us that $H=\langle Q',R' \rangle$ is relatively quasiconvex in $G$. Clearly it is also finitely generated, hence it must be separable in $G$ by QCERF-ness. Observe that since \(Q'\) and \(R'\) are finite index subgroups in \(Q\) and \(R\) respectively,
    \[
        Q H R = \bigcup_{i=1}^n \bigcup_{j=1}^m a_i H b_j,
    \]
    where \(a_1, \dots, a_n\) are left coset representatives of \(Q'\) in \(Q\), and \(b_1, \dots, b_m\) are right coset representatives of \(R'\) in \(R\). 
    Recalling Remark~\ref{rem:sep_props}, we see that the subset $QHR$ is separable in $G$, thus it is a closed set containing \(QR\) but not containing \(g\). 
    Since we found such a set for an arbitrary $g \in G \setminus QR$, we can conclude that $QR$ is closed in $\pt(G)$, as required.
\end{proof}

Corollary~\ref{cor:almost_comp->sep_dc} from the Introduction can be proved in the same way as Corollary~\ref{cor:double_cosets_sep}, except that one needs to use Theorem~\ref{thm:almost_compat->qc_comb} instead of Theorem~\ref{thm:sep->qc_intro-detailed}.


\part{Separability of products of subgroups}
\label{part:multicosets}
This part of the paper is dedicated to proving Theorem~\ref{thm:RZs} from the Introduction.  
In order to do this we must generalise the discussion of path representatives in Sections~\ref{sec:path_reps}-\ref{sec:multitracking}, adapting the arguments there to deal with additional technicalities.
Let us give a summary of the argument.

Let \(G\) be a QCERF finitely generated relatively hyperbolic group with a finite collection of peripheral subgroups \(\{H_\nu \mid \nu \in \Nu\}\).
Suppose that, for each \(\nu \in \Nu\), the subgroup \(H_\nu\) has property \(RZ_s\).
Let \(F_1, \dots, F_s \leqslant G\) be finitely generated relatively quasiconvex subgroups. 
In order to show that the product \(F_1 \dots F_s\) is separable, we proceed by induction on \(s\).
The case that \(s = 1\) is the QCERF condition and \(s = 2\) is Corollary~\ref{cor:double_cosets_sep}, so we may assume \(s > 2\).
For ease of reading we now relabel the subgroups \(F_1 = Q, F_2 = R, F_3 = T_1, \dots, F_s = T_m\), where \(m = s-2 > 0\).

We approximate the product \(Q R T_1 \dots T_m\) with sets of the form \(Q \langle Q', R' \rangle R T_1 \dots T_m\), where \(Q' \leqslant_f Q\) and \(R' \leqslant_f R\) are finite index subgroups of \(Q\) and \(R\) respectively.
Observe that we can write these sets as finite unions
\begin{equation}\label{eq:approx_subsets}
 Q \langle Q', R' \rangle R T_1 \dots T_m = \bigcup_{i,j} a_i \langle Q', R' \rangle b_j T_1 \dots T_m,    
\end{equation}
where the elements \(a_i\) and \(b_j\) are coset representatives of \(Q'\) and \(R'\) in $Q$ and $R$ respectively.
Note that the products on the right-hand side of \eqref{eq:approx_subsets} now involve only \(s-1\) subgroups. By Theorem~\ref{thm:sep->qc_intro}, the subgroups \(Q'\) and \(R'\) can be chosen so that \(\langle Q', R' \rangle\) is relatively quasiconvex, hence we can apply the induction hypothesis to show that such products are separable in $G$.

It then remains to prove that the product \(Q R T_1 \dots T_m\) is, in fact, an intersection of subsets of the form \(Q \langle Q', R' \rangle R T_1 \dots T_m\) as above.
To this end, we study path representatives \(q p_1 \dots p_n r t_1 \dots t_m\) of elements of \(Q \langle Q', R' \rangle R T_1 \dots T_m\) in a similar manner to Part~\ref{part:metric_qc_double_cosets}.
The main additional difficulty comes from controlling instances of multiple backtracking that involve segments in the \(t_1 \dots t_m\) part of the path.
We introduce new metric conditions \descref{C2-m} and \descref{C5-m} to deal with these technicalities.

\section{Auxiliary definitions}
\label{sec:multicoset_defs}

\begin{convention} 
\label{conv:main_multicoset}
    We write \(G\) for a group generated by a finite set \(X\) and hyperbolic relative to a family of subgroups \(\{H_\nu \mid \nu \in \Nu\}\), $|\Nu|<\infty$.
    Let $\mathcal{H}=\bigsqcup_{\nu \in \Nu} (H_\nu\setminus\{1\})$ and choose $\delta \in \NN$ so that the Cayley graph $\ga$ is $\delta$-hyperbolic (see Lemma~\ref{lem:Cayley_graph-hyperbolic}).

    We will assume that  \(Q, R, T_1, \dots , T_m \leqslant G\) are fixed relatively quasiconvex subgroups of \(G\), with quasiconvexity constant \(\varepsilon \ge 0\), where $m \in \NN_0$. Denote 
    \(S=Q \cap R\).
\end{convention}

Throughout this section we use \(Q'\) and \(R'\) to denote subgroups of $Q$ and $R$ respectively. We will also assume that \(Q' \cap R'= Q \cap R = S\) (that is, \(Q'\) and \(R'\) satisfy \descref{C1}).

\subsection{New metric conditions}
Suppose $B, C \ge 0$ are some constants, $\mathcal{P}$ is a finite collection of maximal parabolic subgroups of $G$, and \(\mathcal{U}\) is a finite family of finitely generated relatively quasiconvex subgroups of \(G\).
We will be interested in the following generalisations of conditions \descref{C2} and \descref{C5} to the multiple coset setting:

\begin{itemize}
    \descitem{C2-m}
        \(\minx \Bigl(R \langle Q', R' \rangle R T_1 \dots T_j \setminus R T_1 \dots T_j \Bigr) \geq B\), for each \(j=0,\dots, m\);
    \descitem{C5-m}
        \(\minx \Bigl(q \langle Q'_P, R'_P \rangle R_P (U_1)_P \dots (U_j)_P \setminus qQ'_P R_P (U_1)_P \dots (U_j)_P \Bigr) \geq C \), for each \(P \in \mathcal{P}\), all \(q \in Q_P\),  any $j \in \{0, \dots, m\}$  and arbitrary \(U_1, \dots, U_j \in \mathcal{U}\), where $(U_i)_P=U_i \cap P \leqslant P$.
\end{itemize}

\begin{remark} Let us make the following observations.
\label{rem:C5m->C5}
    \begin{itemize}
        \item  When $j=0$, the inequality from condition \descref{C2-m} reduces to $\minx(R \langle Q',R' \rangle R \setminus R) \ge B$, which is a part of \descref{C2}; on the other hand, the inequality from condition \descref{C5-m} simply becomes \descref{C5}. 
        In particular, for each $m \ge 0$, \descref{C5-m} implies  \descref{C5}.
        \item In our usage of \descref{C5-m}, the set $\mathcal U$  will consists of finitely many conjugates of $T_1,\dots,T_m$; in fact, $U_i=T_i^{a_i}$, for some $a_i \in G$, $i=1,\dots,m$.
    \end{itemize}
\end{remark}

\begin{remark}
    Similarly to conditions \descref{C1}-\descref{C5}, the above conditions are best understood with a view towards the profinite topology.
    \begin{itemize}
        \item
            To prove separability of products of relatively quasiconvex subgroups we argue by induction on the number of factors.
            That is, we assume that the product of \(m+1\) relatively quasiconvex subgroups is separable and then deduce the separability of the product of \(m+2\) relatively quasiconvex subgroups.
            The existence of finite index subgroups \(Q' \leqslant_f Q\) and \(R' \leqslant_f R\) realising condition \descref{C2-m} will be deduced from this inductive assumption.
        \item
            The existence of finite index subgroups \(Q' \leqslant_f Q\) and \(R' \leqslant_f R\) realising condition \descref{C5-m}, given a finite family \(\mathcal{U}\), will be deduced from the assumption that the peripheral subgroups \(\{H_\nu \, | \, \nu \in \Nu\}\) of \(G\) each satisfy the property \(\mathrm{RZ}_{m+2}\).
    \end{itemize}
\end{remark}

\subsection{Path representatives for products of subgroups}
In this subsection we define path representatives for elements of $Q \langle Q', R' \rangle RT_1\dots T_m$ similarly to the path representatives for elements of $Q \langle Q', R' \rangle R$ from Definition~\ref{def:double_coset_path_reps} and discuss their properties.

\begin{definition}[Path representative, III] 
\label{def:multicoset_path_reps}
    Let \(g\) be an element of \(Q \langle Q', R' \rangle R T_1 \dots T_m\).
    Suppose that \(p = q p_1 \dots p_n r t_1 \dots t_m\) is a broken line in \(\Gamma(G,X\cup\mathcal{H})\) satisfying the following properties:
    \begin{itemize}
        \item \(\elem{p} = g\);
        \item \(\elem{q} \in Q\)  and \(\elem{r} \in R\);
        \item \(\elem{p_i} \in Q' \cup R'\) for each \(i \in \{1, \dots n\}\);
        \item \(\elem{t_i} \in T_i\) for each \(i \in \{1, \dots m\}\).
    \end{itemize}
    Then we say that \(p\) is a \emph{path representative} of \(g\) in the product \(Q \langle Q', R' \rangle R T_1 \dots T_m\).
\end{definition}

The type of a path representative is defined as before (cf Definitions~\ref{def:type_of_path_rep} and \ref{def:type_of_double_coset_path_rep}).

\begin{definition}[Type and width of a path representative, III]  
\label{def:type_of_multicoset_path_rep}
    Let  \(g \in Q \langle Q', R' \rangle R T_1 \dots T_m\) and let \(p = q p_1 \dots p_n r t_1 \dots t_m\) be a path representative of $g$ in the sense of Definition~\ref{def:multicoset_path_reps}.
    Denote by \(Y\)  the set of all \(\mathcal{H}\)-components of the segments of \(p\).
    We define the \emph{width} of $p$ as the integer $n$ and the
    \emph{type} of \(p\) as the triple
    \[\tau(p) = \Big(n, \ell(p),\sum_{y \in Y} |y|_X \Big) \in {\NN_0}^3.\]
\end{definition}

The following observation will be useful.
\begin{remark}
\label{rem:path_rep_from_product} 
    Suppose $g \in Q \langle Q', R' \rangle R T_1 \dots T_m$ can be written as a product 
    \[
        g=x y_1 \dots y_n z u_1 \dots u_m,
    \] 
    where $x \in Q$, $y_1,\dots y_n \in Q' \cup R'$, $z \in R$ and $u_i \in T_i$, for each $i=1,\dots,m$. Then $g$ has a  path representative of width $n$.
\end{remark}

Similarly to path representatives of elements of \(\langle Q', R' \rangle\) (in the sense defined in Section~\ref{sec:path_reps}), we will be interested in path representatives whose type is minimal (as an element of \({\NN_0}^3\) under the lexicographic ordering).
Given an element \(g \in Q \langle Q', R' \rangle R T_1 \dots T_m\), such a path representative is always guaranteed to exist. Let us make the following observation (cf Remark~\ref{rem:props_of_double_coset_path_reps}).

\begin{remark}
\label{rem:multicoset_path_reps} 
    Suppose that \(p = q p_1 \dots p_n r t_1 \dots t_m\) is a minimal type path representative of an element \(g \in Q \langle Q', R' \rangle R T_1 \dots T_m\) such that $g \notin QR T_1 \dots T_m$. 
    Then $n>0$, $\elem{p_1} \in R' \setminus S$, $\elem{p_n} \in Q' \setminus S$ and the labels of $p_1,\dots,p_{n}$ alternate between representing elements of $R' \setminus S$ and $Q' \setminus S$. 
    In particular, the integer $n$ must be even.
\end{remark}

Note that in Definition~\ref{def:multicoset_path_reps} the geodesic paths $q$, $r$ and $t_1,\dots,t_m$ are always counted as segments of the path $p$, even if they end up being trivial paths. 
For example a minimal type path representative of an element $g \in R'Q'T_1\dots T_m\setminus QRT_1 \dots T_m$ will be a broken line $p=qp_1p_2rt_1 \dots t_m$ with $m+4$ segments, where $q$ and $r$ are trivial paths.

The proofs of the main results from Sections \ref{sec:path_reps} and \ref{sec:adj_backtracking} can be easily adapted to apply to minimal type path representatives of elements $g \in Q \langle Q', R' \rangle R T_1 \dots T_m\setminus QRT_1 \dots T_m$ (in the sense of Definitions~\ref{def:multicoset_path_reps} and \ref{def:type_of_multicoset_path_rep}), with only superficial differences, so the proofs of the following generalisations of Lemmas~\ref{lem:bddinnprod}, \ref{lem:shortspikes} and \ref{lem:longadjbacktracking}, respectively, will be omitted.

\begin{lemma}
\label{lem:multicoset_bdd_inn_prod}
    There is a constant \(C_0 \geq 0\) such that the following holds.

    Assume that $Q' \leqslant Q$ and $R' \leqslant R$  are subgroups satisfying condition \descref{C1}.
    Consider any element \(g \in Q \langle Q', R' \rangle R T_1 \dots T_m\) with \(g \notin QRT_1\dots T_m\). Let \(p = q p_1 \dots p_n r t_1 \dots t_m\) be a path representative of $g$ of minimal type,
    with nodes $f_0,\dots,f_{n+m+2}$ (that is, \(f_0 = q_-\), \(f_{i} = (p_i)_-\),  for each \(i \in \{1, \dots, n\}\), $f_{n+1}=r_-$,  \(f_{n+1+j} = (t_j)_-\), for each \(j \in \{1, \dots, m\}\), and $f_{n+m+2}=(t_m)_+$). 
    Then \(\langle f_{i-1}, f_{i+1} \rangle_{f_i}^{rel} \leq C_0\), for all \(i \in \{1, \dots, n+m+1\}\).
\end{lemma}

\begin{lemma}
\label{lem:multicoset_bdd_cusps}
    There is a constant \(C_1 \geq 0\) such that the following is true.

    Let $Q' \leqslant Q$ and $R' \leqslant R$  be subgroups satisfying condition \descref{C1}.
    Consider a minimal type path representative
    \(p = q p_1 \dots p_n r t_1 \dots t_m\) for an element \(g \in Q \langle Q', R' \rangle R T_1 \dots T_m \setminus QRT_1\dots T_m\).
    If $a$ and $b$ are adjacent segments of $p$, with $a_+=b_-$, and
    \(h\) and \(k\) are connected \(\mathcal{H}\)-components of $a$ and $b$ respectively,  then \(d_X(h_+,a_+) \leq C_1\) and \(d_X(a_+, k_-) \leq C_1\).
\end{lemma}

\begin{lemma}
\label{lem:multicoset_adj_backtracking}
    For any \(\zeta \geq 0\) there is \(\Theta_0 = \Theta_0(\zeta) \in \NN\) such that the following is true.

    Let $Q' \leqslant Q$ and $R' \leqslant R$  be subgroups satisfying condition \descref{C1}.
    Consider a minimal type path representative
    \(p = q p_1 \dots p_n r t_1 \dots t_m\) for an element \(g \in Q \langle Q', R' \rangle R T_1 \dots T_m \setminus QRT_1\dots T_m\).
    Suppose that $a$ and $b$ are adjacent segments of $p$, with $a_+=b_-$, and \(h\) and \(k\) are connected \(\mathcal{H}\)-components of $a$ and $b$ respectively, such that
    \[
        \max\{ \abs{h}_X, \abs{k}_X \} \geq \Theta_0.
    \]
    Then \(d_X(h_-, k_+) \geq \zeta\).
\end{lemma}


\section{Multiple backtracking in product path representatives: two special cases}
\label{sec:mcs_multitracking1}
Just like in Theorem~\ref{thm:metric_qc}, the main difficulty in proving Theorem~\ref{thm:RZs} consists in dealing with multiple backtracking in path representatives. In this section we will consider two of the possible cases.
We will be working under Convention~\ref{conv:main_multicoset}.

Throughout the rest of the paper we fix the following notation.

\begin{notation}
\label{not:C_1-P_1}
    let $C_1$ be the larger of the two constants provided by Lemmas~\ref{lem:shortspikes} and \ref{lem:multicoset_bdd_cusps}, and denote by $\mathcal{P}_1$ the finite collection of maximal parabolic subgroups of $G$ given by
    \[
        \mathcal{P}_1 = \{{H_\nu}^ b \, | \, \nu \in \Nu, \abs{b}_X \leq C_1\}.
    \]
\end{notation}

The following lemma is roughly analogous to Lemma~\ref{lem:end_sides_constr}.

\begin{lemma}
\label{lem:final_path_constr}
    For any \(L \geq 0\) and any relatively quasiconvex subgroup \(T \leqslant G\) there is a constant \(L' = L'(L,T) \geq 0\) such that the following is true.

    Let \(P={H_\nu}^ b \in \mathcal{P}_1\), for some \(\nu \in \Nu\) and \(b \in G\), with \(\abs{b}_X \leq C_1\), and let \(t\) be a geodesic path in \(\Gamma(G,X\cup\mathcal{H})\), with \(\elem{t} \in T\).
    Suppose that \(v \in Pb = bH_\nu\) is a vertex of \(t\) and \(u \in P\) is an element satisfying \(d_X(u,t_-) \leq L\). Denote \(a = u^{-1} t_- \in G\).
    Then there is a geodesic path \(t'\)  in $\ga$ such that
    \begin{itemize}
        \item \(t'_- = u\) and \(d_X(t'_+,v) \leq L'\);
        \item \(\elem{t'} \in T^{a} \cap P\);
        \item \({(t'_+)}^{-1}t_+ \in a T\).
    \end{itemize}
\end{lemma}

\begin{proof} 
    Let $K=\max\{C_1,\sigma+L\}$, where $\sigma \ge 0$ is a quasiconvexity constant for $T$. Denote
    \begin{equation}
    \label{eq:def_of_new_D}
        L'= \max\{K'(P,T^a,K) \mid P \in \mathcal{P}_1,~a \in G,~\abs{a}_X \le L\},
    \end{equation}
    where $K'(P,T^a,K)$ is obtained from Lemma~\ref{lem:nbhdintersection}.

    The hypotheses that \(v \in Pb\) and \(\abs{b}_X \leq C_1\) imply that \(d_X(v,P) \leq \abs{b}_X \leq C_1\).
    As \(u \in P\), we have \(P = uP\) and so
    \begin{equation}
    \label{eq:v_dist_from_uP}
        d_X(v,uP) \leq C_1.
    \end{equation}

    Set $x= t_-=ua$. Since $\elem{t} \in T$, we have $d_X(v, xT) \le \sigma$, as $T$ is $\sigma$-quasiconvex. Hence
    \[
        d_X(v,u T^a)=d_X(v, xT a^{-1}) \le d_X(v,xT)+\abs{a}_X \le \sigma+L.
    \]

    Combining the latter inequality with \eqref{eq:v_dist_from_uP} allows us to apply Lemma~\ref{lem:nbhdintersection} to find an element $z \in u(T^a \cap P)$ such that $d_X(v,z) \le L'$, where $L' \ge 0$ is the constant from \eqref{eq:def_of_new_D}.
    Now take \(t'\) to be any geodesic in $\ga$ with \(t'_- = u\) and \(t'_+ = z\).
    It is straightforward to verify that \(t'\) satisfies the first two of the required properties. For the last property, observe that
    \[
        {(t'_+)}^{-1}t_+ =\left({(t'_+)}^{-1} u \right) \left(u^{-1}t_- \right) \left(t_-^{-1} t_+\right) =\elem{t'}^{-1} a \elem{t}\in T^a a T=aT. \qedhere
    \]
\end{proof}

The following notation will be fixed for  the remainder of the paper.

\begin{notation}
\label{not:Li_def}
    Let \(D\) be the constant from Lemma~\ref{lem:end_sides_constr}, corresponding to $C_1$ and $\mathcal{P}_1$ (from Notation~\ref{not:C_1-P_1}) and subgroups $Q,R$.
    We define constants $L_1, \dots. L_{m+1}$ as follows: 
    \[
        L_1 = D + C_1 \quad \textrm{and} \quad L_{i+1} = L'(L_{i},T_i)+C_1, \text{ for each }i = 1, \dots, m,
    \]
    where \(L'\) is  obtained from Lemma~\ref{lem:final_path_constr}.

    We also define the family of subgroups
    \[
        \mathcal{U}_1 = \bigcup_{i=1}^m \Big\{ T_i^g  \, \Big| \, i \in \{1, \dots, m\}, g \in G, \abs{g}_X \leq L_i \Big\},
    \]
    consisting of finitely many conjugates of the subgroups \(T_1, \dots, T_m\).
    Note that, by Lemma~\ref{lem:props_of_qc_sbgps}, each \(U \in \mathcal{U}_1\) is a relatively quasiconvex subgroup of \(G\).
\end{notation}

The next proposition describes how we deal with consecutive backtracking that involves the $t_1 \dots t_m$-part of a path representative of an element $g \in Q \langle Q', R' \rangle R T_1 \dots T_m \setminus Q R T_1 \dots T_m$; it complements Proposition~\ref{prop:multitracking_path} which takes care of backtracking within the $qp_1\dots p_n r$-part.

\begin{proposition} 
\label{prop:backtracking_in_t-tail} 
    Suppose that \(p = q p_1 \dots p_n r t_1 \dots t_m\) is a path representative of minimal type for an element \(g \in Q \langle Q', R' \rangle R T_1 \dots T_m\setminus  Q R T_1 \dots T_m\), where $Q' \leqslant Q$ and $R' \leqslant R$ are some subgroups satisfying \descref{C1}. 
    Let \(P={H_\nu}^ b \in \mathcal{P}_1\), for some \(\nu \in \Nu\) and \(b \in G\), with \(\abs{b}_X \leq C_1\).

    Suppose that $h_1,\dots,h_j$ are connected $H_\nu$-components of the segments $t_1,\dots,t_j$, respectively, with $j \in \{1,\dots,m\}$, such that $(h_1)_- \in Pb=bH_\nu$. If $u_1 \in P$ is an element satisfying $d_X(u_1,(t_1)_-) \le L_1$ then there exist elements $a_1,\dots,a_j \in G$ and a broken line $t_1' \dots t_j'$ in $\ga$ such that the following conditions hold:
    \begin{enumerate}
        \item[(i)] $(t_1')_-=u_1$ and $d_X((t_j')_+,(h_j)_+) \le L_{j+1}$;
        \item[(ii)] $a_{i+1} \in a_i T_i$, for $i=1,\dots,j-1$;
        \item[(iii)] $a_i=(t_i')_-^{-1} (t_i)_-$ and $\abs{a_i}_X \le L_i$, for each $i=1,\dots,j$;
        \item[(iv)] $\elem{t'_i} \in T_i^{a_i} \cap P$, for all $i=1,\dots,j$.
    \end{enumerate}
\end{proposition}

\begin{proof} 
    We start by setting $a_1=u_1^{-1} (t_1)_-$, so that $\abs{a_1}_X = d_X(u_1,(t_1)_-) \le L_1$. Note that $(h_1)_+=(h_1)_- \elem{h_1} \in bH_\nu=Pb $. 
    Therefore we can apply Lemma~\ref{lem:final_path_constr} to find a geodesic path $t_1'$ in $\ga$ such that $(t_1')_-=u_1$, $d_X((t_1')_+,(h_1)_+) \le L'(L_1,T_1)$, $\elem{t_1'} \in T_1^{a_1} \cap P$ and
    \begin{equation}
    \label{eq:t_1'-t_1}
        (t_1')_+^{-1} (t_1)_+ \in a_1T_1.
    \end{equation}
    It follows that properties (ii)--(iv) are satisfied for $i=1$, while property (i) holds because $L_2 \ge L'(L_1,T_1)$ by definition.
    If $j=1$ then property (ii) is vacuously true.

    We can now suppose that $j >1$. 
    Then $h_1$ is connected to the component $h_2$ of $t_2$, so, according to Lemma~\ref{lem:multicoset_bdd_cusps},  $d_X((h_1)_+,(t_1)_+) \le C_1$. Set $u_2=(t_1')_+$ and $a_2=u_2^{-1} (t_1)_+$. 
    Note that $a_2 \in a_1T_1$ by \eqref{eq:t_1'-t_1} and
    \[
        \abs{a_2}_X=d_X((t_1)'_+,(t_1)_+) \le d_X((t_1')_+,(h_1)_+)+d_X((h_1)_+,(t_1)_+) \le L'(L_1,T_1)+C_1=L_2.
    \] 
    Since $(t_2)_-=(t_1)_+$, we see that $a_2=u_2^{-1} (t_2)_-$ and $d_X(u_2,(t_2)_-)=\abs{a_2}_X \le L_2$.

    Now, observe that $u_2 =u_1 \elem{t_1'} \in P $ and $(h_2)_+ \in bH_\nu=Pb$, as $h_2$ is connected to $h_1$. 
    This allows us to use Lemma~\ref{lem:final_path_constr} to find a geodesic path $t_2'$ in $\ga$ such that $(t_2')_-=u_2=(t_1')_+$, $d_X((t_2')_+,(h_2)_+) \le L'(L_2,T_2)$, $\elem{t_2'} \in T_2^{a_2} \cap P $ and $(t_2')_+^{-1}t_+ \in a_2T_2$ (see Figure~\ref{fig:new_tail}).

    \begin{figure}[ht]
        \centering
        \includegraphics{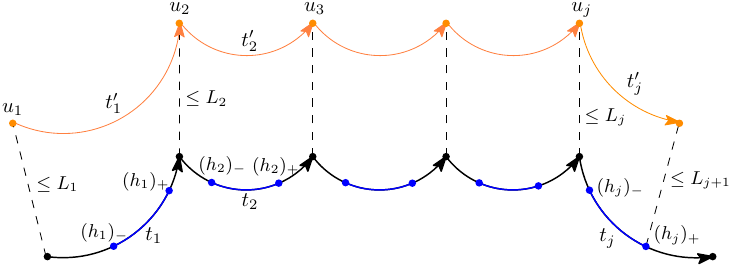}
        \caption{The new path \(t_1'\dots t_j'\) constructed in Proposition~\ref{prop:backtracking_in_t-tail}.}
        \label{fig:new_tail}
    \end{figure}

    If $j=2$ then we are done, otherwise  we construct the remaining elements $a_3, \dots,a_j$ and the paths $t_3',\dots,t_j'$ inductively, similarly to the construction of $a_2$ and $t_2'$ above.
\end{proof}

The next two propositions prove that, under certain conditions, instances of multiple backtracking are long. Essentially, they generalise Proposition~\ref{prop:long_multitracking}.
The first of these shows how we can use condition \descref{C5-m} to deal with particular instances of multiple backtracking.

\begin{proposition}
\label{prop:C5n_multiracking}
    For each \(\zeta \geq 0\) there is a constant \(C_2 = C_2(\zeta) \geq 0\) such that if $Q' \leqslant Q$ and $R' \leqslant R$ satisfy conditions \descref{C1}, \descref{C3} and \descref{C5-m} with constant \(C \geq C_2\) and finite  families $\mathcal{P}$ and $\mathcal{U}$, such that \(\mathcal{P}_1 \subseteq \mathcal{P}\) and \(\mathcal{U}_1 \subseteq \mathcal{U}\), then the following is true.

    Let \(p = q p_1 \dots p_n r t_1 \dots t_m\) be a  minimal type path representative for some \(g \in Q \langle Q', R' \rangle R T_1 \dots T_m\), with \(g \notin Q R T_1 \dots T_m\).
    Suppose that \(p\) has multiple backtracking along \(H_\nu\)-components \(h_1, \dots, h_k\) of its segments, for some $\nu \in \Nu$, such that
    \begin{itemize}
        \item \(h_1\) is an \(H_\nu\)-component of either \(q\) or \(p_i\), for some \(i \in \{ 1, \dots, n-1\}\), with \(\elem{p_i} \in Q'\);
        \item \(h_k\) is an \(H_\nu\)-component of a segment \(t_j\), for some $j \in \{1,\dots,m\}$.
    \end{itemize}
    Then \(d_X((h_1)_-,(h_k)_+) \geq \zeta\).
\end{proposition}

\begin{proof}
    Take \[C_2 = \max\{ 2C_1,  D + \zeta + L_{j} \mid j = 1, \dots, m+1 \} + 1,\] where \(D\) and \(L_{j}\) are defined in Notation~\ref{not:Li_def}, and suppose that \(C \geq C_2\).

    The proof employs the same strategy as Proposition~\ref{prop:long_multitracking}: we first construct a path whose endpoints are close to \((h_1)_-\) and \((h_k)_+\) and whose label represents an element of a parabolic subgroup.
    We will then obtain a contradiction with the minimality of the type of \(p\), using condition \descref{C5-m}.

    We will focus on the case when \(h_1\) is an \(H_\nu\)-component of \(p_i\), for some \(i \in \{1, \dots, n-1\}\) with \(\elem{p_i} \in Q'\), with the case when \(h_1\) is an \(H_\nu\)-component of \(q\) being similar. Note that since \(g \notin Q R T_1 \dots T_m\), it must be that \(n \geq 2\) by Remark~\ref{rem:multicoset_path_reps}.
    After translating by \((p_i)_+^{-1}\), we may assume that \((p_i)_+ = 1\).
    We write \(b = (h_1)_+\) and note that, according to Lemma~\ref{lem:multicoset_bdd_cusps},
    \begin{equation}
    \label{eq:mc_bdd_b}
        \abs{b}_X = d_X((h_1)_+,(p_i)_+) \leq C_1.
    \end{equation}

    Let \(P = bH_\nu b^{-1}\in \mathcal{P}_1 \subseteq \mathcal{P}\).
    Since \(h_1, \dots, h_k\) are pairwise connected, the vertices \((h_l)_+\) lie in the same left coset \(bH_\nu\), for all \(l = 1, \dots, k\), thus
    \begin{equation}
    \label{eq:hi+_in_Pb}
        (h_l)_+ \in Pb, \text{ for all } l = 1, \dots, k.
    \end{equation}

    We construct a new broken line \(p' = p_i' \dots p'_n r' t'_1 \dots t'_j\) in two steps. It will be used in conjunction with condition \descref{C5-m} to obtain a path representative of \(g\) with lesser type than \(p\).

    \medskip
    \noindent    
    \underline{\emph{Step 1:}}
        we start by constructing geodesic paths \(p'_i,p'_{i+1}, \dots, p'_n\) and $r'$ by using condition \descref{C3} and applying Lemmas~\ref{lem:end_sides_constr} and \ref{lem:(c3)->vertex_constr}, in exactly the same way as in the proof of Proposition~\ref{prop:multitracking_path}. The newly constructed paths will have the following properties:
        \begin{itemize}
            \item \(\elem{p'_i} \in Q_P\), \(\elem{p'_l} \in Q'_P \cup R'_P\), for each \(l = i+1, \dots, n\), and \(\elem{r'} \in R_P\);
            \item \(d_X((p'_i)_-,(h_1)_-) \leq D\) and \((p'_i)_+ = (p_i)_+=1\);
            \item \((p'_l)_+ = (p'_{l+1})_-\), for \(l = i, \dots, n-1\);
            \item \(r'_- = (p'_n)_+\) and \(d_X(r'_+, (h_{k-j})_+) \leq D\);
            \item \((p'_l)_+^{-1}(p_l)_+ \in S\), for \(l = i+1, \dots, n\).
        \end{itemize}
    \medskip
    \noindent    
    \underline{\emph{Step 2:}} 
        we now construct geodesic paths \(t'_1, \dots, t'_{j}\) as follows. Set $u_1=(r')_+$ and observe that since $(p_{i+1}')_-=(p_i')_+=1$, we have
        \[
            u_1=\elem{p_{i+1}'} \dots \elem{p_{n}'} \elem{r'} \in P.
        \]
        By Lemma~\ref{lem:multicoset_bdd_cusps}, we have \(d_X((h_{k-j})_+,(t_1)_-) =d_X((h_{k-j})_+,r_+) \leq C_1\). Moreover, by Step 1 above, \(d_X(u_1, (h_{k-j})_+) \leq D\). Therefore
        \begin{equation*} 
        \label{eq:dist_from_u_to_r+}
            d_X(u_1,(t_1)_-) \leq C_1 + D = L_1.
        \end{equation*}
        Together with \eqref{eq:hi+_in_Pb} this allows us to apply Proposition~\ref{prop:backtracking_in_t-tail} to find  elements $a_1,\dots,a_j \in G$ and a broken line $t_1' t_2' \dots t_j'$ in $\ga$ such that

        \begin{itemize}
            \item $(t_1')_-=u_1$ and $d_X((t_j')_+,(h_k)_+) \le L_{j+1}$;
            \item $a_{l+1} \in a_l T_l$, for $l=1,\dots,j-1$;
            \item $a_l=(t_l')_-^{-1} (t_l)_-$ and $\abs{a_l}_X \le L_l$, for each $l=1,\dots,j$;
            \item $\elem{t'_l} \in T_l^{a_l} \cap P$, for all $l=1,\dots,j$.
        \end{itemize}

        Observe that
        \begin{equation}
        \label{eq:a1_in_R}
            \begin{aligned}
                a_1 &= (t_1')_-^{-1} (t_1)_-=u_1^{-1}r_+= ({r'_+}^{-1} r'_-) ({r'_-}^{-1} r_-) (r_-^{-1} r_+) \\ 
                & = \elem{r'}^{-1} (p'_n)_+^{-1}(p_n)_+ \elem{r} \in R_P S R \subseteq R.   
            \end{aligned}
        \end{equation}

    We now define a new broken line $p'$ in $\ga$ by \[p'=p_i' \dots p_n' r' t_1' \dots t_j'.\] Note that \(d_X(p'_-,(h_1)_-) \leq D\), \(d_X(p'_+,(h_k)_+) \leq L_{j+1}\) and \(\elem{p'} \in \elem{p'_i} \langle Q'_P, R'_P \rangle R_P (T_1^{a_1})_P \dots (T_j^{a_j})_P\), where $\elem{p'_i} \in Q_P$. Moreover, $T_l^{a_l} \in \mathcal{U}_1 \subseteq \mathcal{U}$, for each $l=1,\dots,j$.

    Now, suppose, for a contradiction, that \(d_X((h_1)_-,(h_k)_+) < \zeta\).
    Then, by the triangle inequality,
    \[
        \abs{p'}_X \leq D + \zeta + L_{j+1} < C_2.
    \]
    Thus, as $C \ge C_2$, we can apply \descref{C5-m} to deduce that \(\elem{p'} \in \elem{p_i'}Q'_P R_P(T_1^{a_1})_P \dots (T_j^{a_j})_P\).
    Therefore, there exist elements $z \in \elem{p_i'}Q'_P$, \(x \in R\) and \(y_l \in T_l\), $l=1,\dots,j$, such that \(\elem{p'} = z x y_1^{a_1} \dots y_j^{a_j}\).
    By construction, for each \(l = 1, \dots, j-1\) there is \(b_l \in T_l\) such that \(a_{l+1} = a_l b_l\), and so
    $a_l^{-1} a_{l+1} = b_l \in T_l$.
    Recalling that \((p_i')_+ = (p_i)_+ = 1\), the above yields
    \begin{equation}
    \label{eq:end_of_p'}
        \elem{p'} = z x y_1^{a_1} \dots y_j^{a_j} = z x a_1 y_1 b_1 y_2 b_2 \dots b_{j-1} y_j a_j^{-1}.
    \end{equation}

    Let $\alpha$ and $\beta$ be geodesic segments in $\ga$ connecting $(p_i)_-$ with $(p_i')_-$ and $(t'_j)_+$ with $(t_j)_+$ respectively. Since $(p_i)_+=(p_i')_+$, we have
    \begin{equation} 
    \label{eq:elem_of_alpha}
        \elem{\alpha}=  (p_i)_-^{-1} (p_i')_-=(p_i)_-^{-1} (p_i)_+ (p_i')_+^{-1} (p_i')_-=
        \elem{p_i}\, \elem{p_i'}^{-1}.
    \end{equation}
    On the other hand, it follows from the construction that
    \begin{equation} 
    \label{eq:elem_of_beta}
        \elem{\beta}=(t_j')_+^{-1} (t_j)_+=\elem{t_{j}'}^{-1} (t_j')_-^{-1} (t_j)_- \elem{t_j}=\elem{t_{j}'}^{-1} a_j \elem{t_j} \in T_j^{a_j} a_j T_j=a_j T_j.
    \end{equation}
    The broken lines $p$ and $\gamma=qp_1 \dots p_{i-1} \alpha p' \beta t_{j+1} \dots t_m$ have the same endpoints in $\ga$.
    Hence, in view of \eqref{eq:elem_of_alpha} and \eqref{eq:end_of_p'}, we obtain
    \begin{equation}
    \label{eq:new_decomp_for_g}
        \begin{aligned}
            g &=\elem{p} =\elem{\gamma} = \elem{q}\, \elem{p_1} \dots \elem{p_{i-1}}\, \elem{\alpha} \,\elem{p'} \, \elem{\beta} \, \elem{t_{j+1}} \dots \elem{t_m} \\
            &= \elem{q}\, \elem{p_1} \dots \elem{p_{i-1}} (\elem{p_i} \, \elem{p'_i}^{-1}) (z x a_1 y_1 b_1 y_2 b_2 \dots b_{j-1} y_j a_j^{-1}) \elem{\beta} \, \elem{t_{j+1}} \dots \elem{t_m} \\
            &= \elem{q}\, \elem{p_1} \dots  \elem{p_{i-1}} (\elem{p_i} \, \elem{p'_i}^{-1} z) (x a_1) (y_1 b_1) \dots (y_{j-1} b_{j-1}) (y_j a_j^{-1} \elem{\beta}) \elem{t_{j+1}} \dots \elem{t_m}.
        \end{aligned} 
    \end{equation}

    Recall that $\elem{q} \in Q$, $\elem{p_1}, \dots, \elem{p_{i-1}} \in Q' \cup R'$ and $\elem{t_l} \in T_l$, for $l=j+1,\dots,m$, by definition. 
    On the other hand, $\elem{p_i} \, \elem{p'_i}^{-1} z \in Q'\, \elem{p'_i}^{-1}\, \elem{p'_i} \, Q'_P=Q'$,
    \(xa_1 \in R\) by \eqref{eq:a1_in_R} and \(y_l b_l \in T_l\), for each \(l = 1, \dots, j-1\), by construction. 
    Finally, \(y_j a_j^{-1} \, \elem{\beta} \in T_j a_j^{-1} a_jT_j= T_j\) by \eqref{eq:elem_of_beta}.
    Thus, following Remark~\ref{rem:path_rep_from_product}, the  product decomposition \eqref{eq:new_decomp_for_g} for \(g\) gives us a path representative of \(g\) with width \(i < n\).
    This contradicts the minimality of the type of \(p\), so the proposition is proved.
\end{proof}

Condition \descref{C2-m} can be used deal with another case of multiple backtracking.

\begin{proposition}
\label{prop:C2-m_multitracking}
    For every \(\zeta \geq 0\) there is a constant \(B_1 = B_1(\zeta) \geq 0\) such that if \(Q' \leqslant Q\) and \(R' \leqslant R\) satisfy condition \descref{C2-m} with constant \(B \geq B_1\) then the following is true.

    Let \(p = q p_1 \dots p_n r t_1 \dots t_m\) be a minimal type path representative for some \(g \in Q \langle Q', R' \rangle R T_1 \dots T_m\), with \(g \notin Q R T_1 \dots T_m\), and let  $\nu \in \Nu$.
    Suppose that \(p\) has multiple backtracking along \(H_\nu\)-components \(h_1, \dots, h_k\) of its segments  such that
    \begin{itemize}
        \item \(h_1\) is an \(H_\nu\)-component of \(p_i\), for some $i \in \{ 1, \dots, n-1\}$, with \(\elem{p_i} \in R'\);
        \item \(h_k\) is an \(H_\nu\)-component of \(t_j\) for some \(j \in \{ 1, \dots, m\}\).
    \end{itemize}
    Then \(d_X((h_1)_-,(h_k)_+) \geq \zeta\).
\end{proposition}

\begin{proof}
    Take \(B_1 = \zeta + 2 \varepsilon + 1\), where \(\varepsilon \geq 0\) is a quasiconvexity constant for the subgroups \(R\) and \(T_1, \dots, T_m\) (as in Convention~\ref{conv:main_multicoset}), and let \(B \geq B_1\).
    Suppose, for a contradiction, that \(d_X((h_1)_-,(h_k)_+) < \zeta\).

    Since $\elem{p_i} \in R'$, we have \(d_X((h_1)_-,(p_i)_+ \,R) \leq \varepsilon\), by the quasiconvexity of \(R\).
    Therefore there is a geodesic path \(p'_i\) in $\ga$, such that \(\elem{p'_i} \in R\), \(d_X((p'_i)_-,(h_1)_-) \leq \varepsilon\) and \((p'_i)_+ = (p_i)_+\).
    Similarly, using the quasiconvexity of \(T_j\), we can find a geodesic path \(t'_j\) in $\ga$, such that \(\elem{t'_j} \in T_j\), \((t'_j)_- = (t_j)_-\) and \(d_X((t'_j)_+,(h_k)_+) \leq \varepsilon\).
    Let \(p'\) be the broken line \(p'_i p_{i+1} \dots p_n r t_1 \dots t_{j-1} t'_j\).

    Observe that \(\elem{p'} \in R \langle Q', R' \rangle R T_1 \dots T_j\) and, by the triangle inequality, \(\abs{p'}_X \leq \zeta + 2 \varepsilon\).
    Therefore we can apply condition \descref{C2-m} to \(\elem{p'}\) to find that \(\elem{p'} = x y_1 \dots y_j\), where \(x \in R\) and \(y_l \in T_l\), for each \(l = 1, \dots, j\).

    The broken lines $p$ and $\gamma=qp_1 \dots p_i {p_i'}^{-1} p' {t'_j}^{-1} t_j \dots t_m$ have the same endpoints, hence
    \begin{equation}
    \label{eq:new_prod_for_g}
        \begin{aligned}
            g &=\elem{p}=\elem{\gamma} = \elem{q} \,\elem{p_1} \dots \elem{p_{i}}\, \elem{p'_i}^{-1} \, \elem{p'} \, \elem{t_j'}^{-1}\, \elem{t_j} \dots \elem{t_m}\\ 
            &= \elem{q} \,\elem{p_1} \dots \elem{p_{i-1}} (\elem{p_{i}}\, \elem{p'_i}^{-1} \, x) y_1 \dots y_{j-1} (y_j\,\elem{t_j'}^{-1}\, \elem{t_j}) \elem{t_{j+1}} \dots \elem{t_m}.
        \end{aligned}
    \end{equation}

    Note that $\elem{p_{i}} \, \elem{p'_i}^{-1} x \in R$ and $y_j\, \elem{t_j'}^{-1} \, \elem{t_j} \in T_j$. 
    In view of Remark~\ref{rem:path_rep_from_product}, the product decomposition of $g$ from \eqref{eq:new_prod_for_g} can be used to obtain a path representative $p''$ of $g$ with width $i-1<n$. Thus the type of \(p''\) is strictly less than the type of \(p\), which yields the desired contradiction.
\end{proof}

\section{Multiple backtracking in product path representatives: general case}
\label{sec:mcs_multitracking2}
Propositions~\ref{prop:long_multitracking}, \ref{prop:C5n_multiracking} and \ref{prop:C2-m_multitracking} above show that for $g  \in Q \langle Q', R' \rangle R T_1 \dots T_m \setminus Q R T_1 \dots T_m$, instances of multiple backtracking in a minimal type path representative $p=qp_1\dots p_n r t_1 \dots t_m$, that start at a component of $q$, $p_1$,$\dots$, or $p_{n-1}$, are long.
We cannot draw the same conclusion in all cases since we have no control over the elements $\elem{r}, \elem{t_1},\dots,\elem{t_m}$. 
Therefore in this section we use a different approach.
Proposition~\ref{prop:multicoset_multitracking} below shows that in the remaining cases we can find a path representative with one of the segments from the tail section $rt_1 \dots t_m$ being short with respect to the proper metric \(d_X\). 
Note that the main constant $\xi_0=\xi_0(Q',\zeta)$, produced in this proposition, will depend on $Q'$ (unlike the constants $C_1$, $D$, $C_2(\zeta)$, $B_1(\zeta)$, $\dots$, defined previously) but will be independent of $R'$.

As before, we work under Convention~\ref{conv:main_multicoset}. 
We will also keep using Notation~\ref{not:C_1-P_1} and \ref{not:Li_def}. Let us start with the following elementary observation.

\begin{lemma}
\label{lem:shorten_tail}
    For any \(\zeta \geq 0\) and any given subsets \(A_1, \dots, A_k \subseteq G\), $k \ge 1$, there is a constant \(\xi = \xi(\zeta, A_1, \dots, A_k) \geq 0\) such that if \(g \in A_1 \dots A_k\) and \(\abs{g}_X \leq \zeta\), then there exist \(a_1 \in A_1\),\(\dots\), \(a_k \in A_k\) such that \(g = a_1 \dots a_k\) and \(\abs{a_i}_X \leq \xi\), for all \(i \in \{1, \dots, k\}\).
\end{lemma}

\begin{proof}
    For each \(g \in A_1 \dots A_k\) fix some elements \(a_{1,g} \in A_1, \dots, a_{k,g} \in A_k\) such that \(g = a_{1,g} \dots a_{k,g}\).
    Now we can define
    \[
        \xi = \max \Big\{ \abs{a_{1,g}}_X, \dots, \abs{a_{k,g}}_X \, \Big| \, g \in A_1 \dots A_k, ~\abs{g}_X \leq \zeta \Big\} < \infty.
    \]
    Clearly $\xi$ has the required property.
\end{proof}

\begin{definition}[Tail height] 
\label{def:tail_height} 
    Suppose that $Q' \leqslant Q$, $R' \leqslant R$ and $p= q p_1 \dots p_n r t_1 \dots t_m$ is a path representative of an element $g \in Q \langle Q', R' \rangle R T_1 \dots T_m$. 
    The \emph{tail height} of $p$, $th_X(p)$, is defined as 
    \[
        th_X(p)=   \min\{ \abs{r}_X, \abs{t_1}_X, \dots, \abs{t_{m-1}}_X\} .
    \]
\end{definition}

\begin{proposition}
\label{prop:multicoset_multitracking}
    For each \(\zeta \geq 0\), let $C_2=C_2(\zeta)$ be the larger of the two constants provided by Propositions~\ref{prop:long_multitracking} and \ref{prop:C5n_multiracking}, and let $B_1=B_1(\zeta)$ be given by Proposition~\ref{prop:C2-m_multitracking}. Set $B_2 = B_2(\zeta) =\max\{C_2(\zeta),B_1(\zeta)\}$.

    Suppose that $Q' \leqslant Q$ is a relatively quasiconvex subgroup of $G$ containing $S=Q \cap R$. Then there exists a constant \(\xi_0 = \xi_0(Q',\zeta) \geq 0\) such that if \(R'\leqslant R\) and $Q'$, $R'$ satisfy conditions \descref{C1}-\descref{C4}, \descref{C2-m} and \descref{C5-m}, with constants \(B \geq B_2\) and \(C \geq C_2\) and collections of subgroups \(\mathcal{P} \supseteq \mathcal{P}_1\) and \(\mathcal{U} \supseteq \mathcal{U}_1\), then the following is true.

    Let \(p = q p_1 \dots p_n r t_1 \dots t_m\) be a minimal type path representative for some \(g \in Q \langle Q', R' \rangle R T_1 \dots T_m\), with \(g \notin Q R T_1 \dots T_m\).
    Suppose that \(p\) has multiple backtracking along \(\mathcal{H}\)-components \(h_1, \dots, h_k\) of its segments, with $k \ge 3$ and \(d_X((h_1)_-,(h_k)_+) \leq \zeta\).
    Then $m \ge 1$ and there is a path representative \(p'\) for \(g\) (not necessarily of minimal type) such that $th_X(p') \le \xi_0$.
\end{proposition}

\begin{proof}
    Let  \(\varepsilon' \geq 0\) be a quasiconvexity constant for \(Q'\).
    Take \(\xi_0=\xi_0(Q',\zeta) \geq 0\) to be the maximum, taken over all indices $i$ and $j$ satisfying \(1 \leq i \leq j \leq m\), of the constants
    \[ \xi(\zeta + \varepsilon + \varepsilon',Q', R, T_1, \dots, T_j), ~~ \xi(\zeta + 2\varepsilon,R, T_1, \dots, T_j) \text{ and } \xi(\zeta + 2\varepsilon,T_i, \dots, T_j),
    \]
    obtained from Lemma~\ref{lem:shorten_tail}.

    Suppose that \(h_1, \dots, h_k\) are as in the statement, with $ d_X((h_1)_-,(h_k)_+) \leq \zeta$.
    There are four possible cases to consider, depending on the segments of \(p\) to which the \(\mathcal{H}\)-components \(h_1\) and \(h_k\) belong to.
    If \(h_k\) is an \(\mathcal{H}\)-component of one of the segments \(p_2, \dots, p_n\) or \(r\), then one obtains a contradiction to the minimality of type of \(p\) by following the same argument as in Proposition~\ref{prop:long_multitracking} (recall that \descref{C5-m} implies \descref{C5} by Remark~\ref{rem:C5m->C5}).

    If \(h_1\) is an \(\mathcal{H}\)-component of one of the segments \(q, p_1, \dots, p_{n-1}\) and \(h_k\) is an \(\mathcal{H}\)-component of one of the segments \(t_1, \dots, t_m\), we obtain a contradiction by applying either Proposition~\ref{prop:C5n_multiracking} or \ref{prop:C2-m_multitracking} (depending on whether \(h_1\) is a component of a segment of $p$ representing an element of \(Q\) or \(R\), respectively).

    It remains to consider the possibility when \(h_1\) is an \(\mathcal{H}\)-component of one of the segments \(p_n, r, t_1, \dots, t_m\). 
    It follows that \(h_k\) is an \(\mathcal{H}\)-component of \(t_j\), for some \(j \in \{ 1, \dots, m\}\), in particular $m \ge 1$.
    For simplicity we treat only the case when \(h_1\) is an \(\mathcal{H}\)-component of \(p_n\); the remaining cases can be dealt with similarly.

    Note that \(\elem{p_n} \in Q'\) by Remark~\ref{rem:multicoset_path_reps}.
    By the relative quasiconvexity of \(Q'\) and \(T_j\) there are geodesic paths \(\alpha\) and \(\beta\) in $\ga$ satisfying
    \begin{gather*}
        d_X(\alpha_-,(h_1)_-) \leq \varepsilon', ~ \alpha_+ = (p_n)_+ \text{ and } \elem{\alpha} \in Q',\\
        \beta_- = (t_j)_-, ~d_X(\beta_+,(h_k)_+) \leq \varepsilon  ~ \text{ and } \elem{\beta} \in T_j.
    \end{gather*}
    Let \(\gamma = \alpha r t_1 \dots t_{j-1} \beta\). Observe that \(\elem{\gamma} \in Q' R T_1 \dots T_j\) and, by the triangle inequality,
    \[
        \abs{\gamma}_X=d_X(\alpha_-,\beta_+) \leq \varepsilon'+\zeta + \varepsilon.
    \]
    Thus, applying Lemma~\ref{lem:shorten_tail}, we can find elements \(x \in Q'\), \(y \in R\), \(z_1 \in T_1\), \(\dots\), \(z_j \in T_j\) such that \(\elem{\gamma} = x y z_1 \dots z_j\) and
    \begin{equation}
    \label{eq:new_path_rep_bds}
        \abs{y}_X \le \xi_0.
    \end{equation}

    Therefore
    \begin{equation}
    \label{eq:prod_decomp_for_p'}
        \begin{aligned}
            g &= \elem{p}=\elem{q}\, \elem{p_1} \dots \elem{p_n} (\elem{\alpha}^{-1}\,\elem{\alpha})\elem{r} \, \elem{t_1} \dots \elem{t_{j-1}} (\elem{\beta} \,\elem{\beta}^{-1}) \elem{t_j} \dots  \elem{t_m} \\
            &= \elem{q} \, \elem{p_1} \dots \elem{p_n} \,\elem{\alpha}^{-1} \, \elem{\gamma} \elem{\beta}^{-1} \, \elem{t_{j}} \dots \elem{t_m} \\
            &=\elem{q} \, \elem{p_1} \dots \elem{p_{n-1}} (\elem{p_n} \, \elem{\alpha}^{-1} \, x) y z_1 \dots z_{j-1} (z_j \, \elem{\beta}^{-1} \, \elem{t_j}) \elem{t_{j+1}} \dots \elem{t_m}.
        \end{aligned}
    \end{equation}

    Following Remark~\ref{rem:path_rep_from_product}, the product decomposition \eqref{eq:prod_decomp_for_p'} gives rise to a path representative $p'=q' p'_1 \dots p'_n r' t'_1 \dots t'_m$ for $g$, where $\elem{q'}=\elem{q} \in Q$, $\elem{p_i'}=\elem{p_i} \in Q' \cup R'$, for $i=1,\dots, n-1$, $\elem{p_n'}=\elem{p_n}\, \elem{\alpha}^{-1} \,x \in Q'$, $\elem{r'}=y \in R$, $\elem{t'_l}=z_l \in T_l$, for $l=1,\dots,j-1$, $\elem{t'_j}=z_j \, \elem{\beta}^{-1}\, \elem{t_j} \in T_j$ and $\elem{t'_s}=\elem{t_s} \in T_s$, for $s=j+1,\dots,m$.
    In view of (\ref{eq:new_path_rep_bds}), we see that \(th_X(p') \le\abs{y}_X \le \xi_0\), so the proof is complete.
\end{proof}

The following proposition is an analogue of Lemma~\ref{lem:pathreps_have_qgd_shortcutting}. It employs the constant \(c_0 = \max\{C_0,14\delta\}\), where \(C_0\) is provided by Lemma~\ref{lem:multicoset_bdd_inn_prod}, and the constants \(\lambda =\lambda(c_0) \geq 1\) and \(c=c(c_0) \geq 0\), given by Proposition~\ref{prop:shortcutting_quasigeodesic}.

\begin{proposition}
\label{prop:mcs_shortcutting_quasigeodesic}
    For any \(\eta \geq 0\) there are constants \(\zeta = \zeta(\eta) \geq 0\), \(C_3 = C_3(\eta) \geq 0\), \(\Theta_1 = \Theta_1(\eta) \in \NN\) and \(B_3 = B_3(\eta) \geq 0\) such that if \(Q' \leqslant Q\) is a relatively quasiconvex subgroup of $G$ and \(B \geq B_3\), \(C \geq C_3\) then there exists $E =E(\eta,Q',B) \ge 0$ such that the following holds.
    
    Suppose $Q'$ and some subgroup \(R'\leqslant R \) satisfy conditions \descref{C1}-\descref{C4}, \descref{C2-m} and \descref{C5-m}, with constants \(B\) and \(C\), and families \(\mathcal{P} \supseteq \mathcal{P}_1\) and \(\mathcal{U} \supseteq \mathcal{U}_1\).
    Let \(p\) be a minimal type path representative for an element  \(g \in Q \langle Q', R' \rangle R T_1 \dots T_m \setminus Q R T_1 \dots T_m\).
    Assume that for any path representative \(p'\) for \(g\) we have $th_X(p') \ge E$.
    Then \(p\) is \((B,c_0,\zeta,\Theta_1)\)-tamable.
  
    Let  \(\Sigma(p,\Theta_1) = f_0 e_1 f_1 \dots e_l f_l\) denote the \(\Theta_1\)-short\-cut\-ting  of \(p\), obtained by applying Procedure~\ref{proc:shortcutting}, and let $e_j'$ be the \(\mathcal{H}\)-component of \(\Sigma(p,\Theta_1)\) containing \(e_j\), $j=1,\dots,l$. Then \(\Sigma(p,\Theta_1)\) is a \((\lambda,c)\)-quasigeodesic without backtracking and  \(\abs{e'_j}_X \geq \eta\), for each \(j = 1, \dots, l\).
\end{proposition}

\begin{proof}
    The proof is similar to the argument in Lemma~\ref{lem:pathreps_have_qgd_shortcutting}.
    Let us define the necessary constants:
    \begin{itemize}
        \item \(\zeta = \zeta(\eta,c_0)\) is the constant from Proposition~\ref{prop:shortcutting_quasigeodesic};
        \item \(\Theta_1 = \max\{\Theta_0(\zeta),\zeta\}\), where \(\Theta_0\) is the constant from Lemma~\ref{lem:multicoset_adj_backtracking};
        \item $B_2(\zeta)$ and \(C_3 = C_2(\zeta)\) are the constants provided by Proposition~\ref{prop:multicoset_multitracking};
        \item \(B_3 = \max\{B_0(\Theta_1,c_0), B_2(\zeta)\} \), where \(B_0(\Theta_1,c_0)\) is the constant from Proposition~\ref{prop:shortcutting_quasigeodesic};
    \end{itemize}
    and, finally, for any given \(B \geq B_3, C \geq C_3\), we set
    \begin{itemize}
        \item \(E = \max \{B, \xi_0(\eta,Q') + 1\}\), where \(\xi_0(\eta,Q')\) is the constant from Proposition~\ref{prop:multicoset_multitracking}.
    \end{itemize}

    Suppose that $Q'$, $R'$, $g$ and \(p = q p_1 \dots p_n r t_1 \dots t_m\) are as in the statement of the proposition. We will now show that $p$ is \((B,c_0,\zeta,\Theta_1)\)-tamable.

    Since \(Q'\) and \(R'\) satisfy \descref{C2}, Lemma~\ref{lem:C2_implies_old_C2} together with Remark~\ref{rem:multicoset_path_reps} imply that \(\abs{p_i}_X \geq B\), for each \(i = 1, \dots, n\).
    Moreover, by assumption, \(\abs{r}_X, \abs{t_1}_X, \dots, \abs{t_{m-1}}_X \geq E \geq B\), so condition \ref{cond:tam_1} of Definition~\ref{def:tamable} is satisfied. On the other hand, condition \ref{cond:tam_2} is satisfied by Lemma~\ref{lem:multicoset_bdd_inn_prod}.

    If condition \ref{cond:tam_3} of Definition~\ref{def:tamable} is not satisfied then \(p\) must have consecutive backtracking along \(\mathcal{H}\)-components \(h_1, \dots, h_k\) of its segments, such that  
    \[
        \max{\Big\{ \abs{h_i}_X \, | \, i = 1, \dots, k \Big\}} \geq \Theta_1~\text{ and } d_X((h_1)_-,(h_k)_+) < \zeta.
    \]

    Lemma~\ref{lem:multicoset_adj_backtracking} rules out the case of adjacent backtracking (\(k = 2\)), so it must be that $k \ge 3$.
    That is, \(h_1, \dots, h_k\) is an instance of multiple backtracking in \(p\).
    Proposition~\ref{prop:multicoset_multitracking} now applies, giving a path representative \(p'\) for $g$ with \(th_X(p') \leq \xi_0(\eta,Q') < E\).
    This contradicts a hypothesis of the proposition, so $p$ must also satisfy condition \ref{cond:tam_3}.

    Therefore \(p\) is \((B,c_0,\zeta,\Theta_1)\)-tamable, and we can apply Proposition~\ref{prop:shortcutting_quasigeodesic} to achieve the desired conclusion.
\end{proof}


\section{Using separability to establish conditions \texorpdfstring{\descref{C2-m}}{(C2-m)} and \texorpdfstring{\descref{C5-m}}{(C5-m)}}
\label{sec:mcs_sep->metric}

In this section we exhibit, under suitable assumptions on \(G\), the existence of finite index subgroups \(Q' \leqslant_f Q\) and \(R' \leqslant_f R\) satisfying conditions \descref{C1}-\descref{C4}, \descref{C2-m} and \descref{C5-m}.

\begin{lemma}
\label{lem:sep->C2-n}
    Let \(G\) be a group generated by finite set \(X\),  let \(Q, R, T_1, \dots, T_m \leqslant G\) be some subgroups, and let $S=Q \cap R$. Suppose that \(R T_1 \dots T_l\) is separable in \(G\), for each \(l = 0, \dots, m\).
    Then for any \(B \geq 0\) there is a finite index subgroup \(N \leqslant_f G\), with \(S \subseteq N\), such that arbitrary subgroups \(Q' \leqslant  Q \cap N\) and \(R'\leqslant R \cap N\) satisfy condition \descref{C2-m} with constant \(B\).
\end{lemma}

\begin{proof}
    For each \(l \in \{ 0, \dots, m\}\) the product \(R T_1 \dots T_l\) is separable, so, by Lemma~\ref{lem:sep->large_minx}(b), there is a finite index normal subgroup \(M_l \lhd_f G\)  such that
    \begin{equation}
    \label{eq:ineq_for_C2-m}
        \minx (R T_1 \dots T_lM_l  \setminus R T_1 \dots T_l) \geq B, \text{ for all } l=0,\dots,m.
    \end{equation}
    Define the subgroup
    \(        M = \bigcap_{l=0}^{m} M_l \lhd_f G    \),
    and take \(N = SM \leqslant_f G\).
    Observe that
    \begin{equation}\label{eq:2nd_ineq_for_C2-m}
    R N R T_1 \dots T_l=RSMRT_1 \dots T_l=RSR T_1 \dots T_lM=RT_1 \dots T_lM, \text{ for all } l=0,\dots,m.
    \end{equation}
    Now choose arbitrary subgroups \(Q' \leqslant  Q \cap N\) and \(R'\leqslant R \cap N\), so that $\langle Q', R' \rangle \subseteq N$. Since \(M \subseteq M_l\) for all $l$, we can  combine \eqref{eq:ineq_for_C2-m} with \eqref{eq:2nd_ineq_for_C2-m} to draw the desired conclusion.
\end{proof}

The next statement is similar to Theorem~\ref{thm:sep->qc_comb}.

\begin{lemma}
\label{lem:sep->C5-n}
    Suppose that \(G\) is a group generated by finite set \(X\) and $m \in \NN_0$. Let \(Q, R \leqslant G\) be some subgroups, and let \(\mathcal{P}, \mathcal{U}\) be finite collections of subgroups of $G$ such that
    \begin{enumerate}[label={\normalfont (\arabic*)}]
        \item \label{cond:C5m-1}  each \(P \in \mathcal{P}\) has property RZ$_{m+2}$;
        \item \label{cond:C5m-2} the subgroups \(Q \cap P\), \(R \cap P\) and \(U \cap P\) are finitely generated, for all \(P \in \mathcal{P}\) and   all \(U \in \mathcal{U}\);
        \item \label{cond:C5m-3} if \(P \in \mathcal{P}\), \(K \leqslant_f P\) and \(L \leqslant_f Q\) then \(KL\) is separable in \(G\).
    \end{enumerate}
    Then for any \(C \geq 0\) and any finite index subgroup \(Q' \leqslant_f Q\), there is a finite index subgroup \(O \leqslant_f G\), with \(Q' \subseteq O\), such for any  \(R' \leqslant R \cap O\) the subgroups \(Q'\) and \(R'\) satisfy \descref{C5-m} with constant \(C\) and collections \(\mathcal{P}\) and \(\mathcal{U}\).
\end{lemma}

\begin{proof}
    As usual, for subgroups \(H \leqslant G\) and \(P \in \mathcal{P}\) we denote \(H \cap P\) by \(H_P\).

    Fix an enumeration \(\mathcal{P} = \{P_1, \dots, P_k\}\) and
    let \(Q' \leqslant_f Q\) be a finite index subgroup of \(Q\). Given any \(i \in \{ 1, \dots, k\}\), we choose some coset representatives \(a_{i1}, \dots, a_{in_i} \in Q_{P_i}\) of \(Q'_{P_i}\), so that \(Q_{P_i} = \bigsqcup_{j=1}^{n_i} a_{ij} Q'_{P_i}\).
    Let $\mathbb{U}$ be the finite set consisting of all $l$-tuples $(U_1,\dots,U_l)$, where $l \in \{0,\dots, m\}$ and $U_1,\dots, U_l \in \mathcal{U}$.

    Consider any \(i \in \{ 1, \dots, k\}\) and $\underline{u}=(U_1,\dots,U_l) \in \mathbb{U}$, where $l \in \{0 \dots,m\}$.
    Note that \(Q'_{P_i} \leqslant_f Q_{P_i}\) is finitely generated, for each \(i = 1, \dots, k\), since \(Q_{P_i}\) is itself finitely generated by assumption \ref{cond:C5m-2}.   Combining assumptions \ref{cond:C5m-1} and \ref{cond:C5m-2}, the subset \(Q'_{P_i} R_{P_i} (U_1)_{P_i} \dots (U_l)_{P_i}\) is separable in \(P_i\). 
    Therefore, 
    by Lemma~\ref{lem:sep->large_minx}(c), for any $C \ge 0$  there is  \(F_{i,\underline{u}} \lhd_f P_i\) such that
    \begin{equation}
    \label{eq:mcs_sep_for_Fiu}
        \minx \Big( a_{ij} Q'_{P_i} F_{i,\underline{u}} R_{P_i} (U_1)_{P_i} \dots (U_l)_{P_i} \setminus a_{ij} Q'_{P_i} R_{P_i} (U_1)_{P_i} \dots (U_l)_{P_i} \Big) \geq C,
    \end{equation}
    for all  $j=1,\dots,n_i$.

    Define \(K_{i,\underline{u}} = Q'_{P_i} F_{i,\underline{u}} \leqslant_f P_i\).
    Then (\ref{eq:mcs_sep_for_Fiu}) implies that for every $j=1,\dots,n_i$ we have
    \begin{equation}
    \label{eq:mcs_sep_for_Kiu}
        \minx \Big( a_{ij} K_{i,\underline{u}} R_{P_i} (U_1)_{P_i} \dots (U_l)_{P_i} \setminus a_{ij} Q'_{P_i} R_{P_i} (U_1)_{P_i} \dots (U_l)_{P_i} \Big) \geq C.
    \end{equation}

    Assumption \ref{cond:C5m-3} tells us that the double coset \(K_{i,\underline{u}} Q'\) is separable in \(G\), and since \(Q' \cap P_i = Q'_{P_i}  \subseteq  K_{i,\underline{u}}\), we can apply Lemma~\ref{lem:Lemma_1} to find a finite index subgroup \(O_{i,\underline{u}} \leqslant_f G\) such that \(Q' \subseteq O_{i,\underline{u}}\) and \(O_{i,\underline{u}} \cap P_i \subseteq K_{i,\underline{u}}\).

    We can now define a finite index subgroup $O$ of $G$ by 
    \[
        O = \bigcap_{i=1}^k \bigcap_{\underline{u} \in \mathbb{U}} O_{i,\underline{u}} \leqslant_f G.
    \]
    Observe that \(Q' \subseteq O\) and  \(O \cap P_i \subseteq K_{i,\underline{u}}\), for each \(i = 1, \dots, k\) and all $\underline{u} \in \mathbb{U}$. Consider any subgroup \(R' \leqslant R \cap O\). 
    Then \(Q_{P_i}' \cup R'_{P_i} \subseteq O \cap P_i \), so (\ref{eq:mcs_sep_for_Kiu}) yields that
    \begin{equation}
    \label{eq:mcs_sep_for_Q'R'}
        \minx \Big( a_{ij} \langle Q'_{P_i}, R'_{P_i} \rangle R_{P_i} (U_1)_{P_i} \dots (U_l)_{P_i} \setminus a_{ij} Q'_{P_i} R_{P_i} (U_1)_{P_i} \dots (U_l)_{P_i} \Big) \geq C,
    \end{equation}
    for arbitrary \(i = 1, \dots, k\),  \(l = 0, \dots, m\), $U_1,\dots,U_l \in \mathcal{U}$ and any $j=1,\dots,n_i$.

    Given any \(i \in \{ 1, \dots, k\}\) and any \(q \in Q_{P_i}\), there is $j \in \{1,\dots,n_i\}$ such that $qQ_{P_i}'= a_{ij}Q_{P_i}'$.
    It follows that $q \langle Q'_{P_i}, R'_{P_i} \rangle = a_{ij} \langle Q'_{P_i}, R'_{P_i} \rangle$, which, combined with \eqref{eq:mcs_sep_for_Q'R'}, shows that \(Q'\) and \(R'\) satisfy condition \descref{C5-m}, as required.
\end{proof}

For the next result we will follow the notation of Convention~\ref{conv:main_multicoset}.

\begin{proposition}
\label{prop:msc_sep->metric} 
    Suppose that $G$ is QCERF, the product $RT_1 \dots T_l$ is separable in $G$, for every $l=0,\dots,m$, and the peripheral subgroup \(H_\nu\) has property RZ$_{m+2}$, for each \(\nu \in \Nu\).
    Let \(\mathcal{P}_1\) be a finite collection of maximal parabolic subgroups and let \(\mathcal{U}_1\) be a finite collection of finitely generated relatively quasiconvex subgroups in $G$.

    Then for any \(B, C \geq 0\) there exist finite index subgroups \(Q' \leqslant_f Q\) and \(R' \leqslant_f R\) such that: 
    \begin{itemize}
        \item $\langle Q',R' \rangle $ is relatively quasiconvex in $G$;
        \item $Q'$, $R'$ satisfy conditions \descref{C1}-\descref{C4}, \descref{C2-m} and \descref{C5-m} with constants \(B\) and \(C\) and collections \(\mathcal{P}_1\) and \(\mathcal{U}_1\).
    \end{itemize}

    More precisely, there is \(L_1 \leqslant_f G\), with \(S \subseteq L_1\), such that for any \(L' \leqslant_f L_1\), satisfying \(S \subseteq L'\), we can take \(Q' = Q \cap L'\leqslant_f Q\), and there is \(M_1 \leqslant_f L'\), with \(Q' \subseteq M_1\), such that for any \(M' \leqslant_f M_1\), satisfying \(Q' \subseteq M'\), the subgroups \(Q'\) and \(R' = R \cap M'\leqslant_f R\) enjoy the above properties.
\end{proposition}

\begin{proof} 
    Fix some constants $B,C \ge 0$. Let $\mathcal{P}_0$ be the finite collection of maximal parabolic subgroups of $G$ provided by Theorem~\ref{thm:metric_qc} and set $\mathcal{P}=\mathcal{P}_0 \cup \mathcal{P}_1$.

    Note that maximal parabolic subgroups of $G$ are double coset separable by the assumptions, as $m+2 \ge 2$. 
    Therefore the argument from the proof of Theorem~\ref{thm:sep->qc_intro-detailed} shows that $G$, its subgroups $Q,R$ and $S= Q \cap R$, and the finite collection $\mathcal P$ satisfy assumptions \ref{cond:1}--\ref{cond:4} of Theorem~\ref{thm:sep->qc_comb}.
    Let $L \leqslant_f G$, with $S \subseteq L$, be the finite index subgroup provided by this theorem.

    By the hypothesis on \(G\), the subsets \(R T_1 \dots T_l\) are separable in \(G\), for each \(l=0,\dots, m\).
    We can therefore apply Lemma~\ref{lem:sep->C2-n} to obtain a finite index subgroup \(N \leqslant_f G\) from its statement (in particular, $S \subseteq N$). 
    Now we define the finite index subgroup $L_1 \leqslant_f G$, from the statement of the proposition, by setting $L_1=L \cap N$.
    Clearly  $L_1$ contains $S$.
    Take any $L' \leqslant_f L_1$, with $S \subseteq L'$, and set $Q'=Q \cap L' \leqslant_f Q$. Let $M \leqslant_f L'$ be the subgroup provided by Theorem~\ref{thm:sep->qc_comb}, with $Q' \subseteq M$.

    Lemma~\ref{lem:fg_qc_int_parab_is_fg} and Corollary~\ref{cor:fi_in_parab_times_fgqc_is_sep} imply that all the assumptions of Lemma~\ref{lem:sep->C5-n} are satisfied, so let $O \leqslant_f G$ be the subgroup given by this lemma, with $Q' \subseteq O$. 
    We now define the finite index subgroup $M_1 \leqslant_f L'$, from the statement of the proposition, by $M_1=M \cap O$.

    Evidently, $M_1$ contains $Q'$.   Choose an arbitrary finite index subgroup $M' \leqslant_f M_1$, with $Q' \subseteq M'$, and set $R'=R \cap M'$. 
    Observe that $M' \leqslant_f G$, by construction, hence $R' \leqslant_f R$.

    The combined statements of Theorem~\ref{thm:sep->qc_comb}, Lemma~\ref{lem:props_of_qc_sbgps},  Theorem~\ref{thm:metric_qc}, Lemma~\ref{lem:sep->C2-n} and Lemma~\ref{lem:sep->C5-n} now imply that the subgroups $Q' \leqslant_f Q$ and $R' \leqslant_f R$, obtained above, satisfy all of the required properties. 
    Thus the proposition is proved.
\end{proof}


\section{Separability of quasiconvex products in QCERF relatively hyperbolic groups}
\label{sec:RZs_proof}

In this section we prove Theorem~\ref{thm:RZs} from the Introduction.

\begin{remark}
\label{rem:qc_prod_sep} 
    Let $G$ be a relatively hyperbolic group. 
    Suppose that $s \in \NN$ and the product of any \(s\) finitely generated relatively quasiconvex subgroups is separable in \(G\).
    If \(Q_1, \dots, Q_s\) are finitely generated quasiconvex subgroups of $G$ and  \(a_0, \dots, a_s \in G\) are arbitrary elements, then the subset \(a_0 Q_1 a_1 \dots  Q_s a_s\) is separable in $G$.
\end{remark}    

Indeed, observe that the subset
\[
    a_0 Q_1 a_1 \dots  Q_s a_s =  Q_1^{a_0} Q_2^{a_0 a_1} \dots Q_s^{a_0 \dots a_{s-1}} a_0 \dots a_s
\]
is a translate of a product of conjugates of the subgroups \(Q_1, \dots, Q_s\).
Combining Lemma~\ref{lem:props_of_qc_sbgps} with Remark~\ref{rem:sep_props} and the assumption on \(G\) yields the desired conclusion.

\begin{proof}[Proof of Theorem~\ref{thm:RZs}] 
    We induct on $s$. The case $s = 1$ is equivalent to the  QCERF property of $G$, while the case  $s = 2$  follows from Corollary \ref{cor:double_cosets_sep}.
    Thus we can assume that \(s > 2\) and the product of any $s-1$ finitely generated relatively quasiconvex subgroups is separable in $G$.
    
    Let $F_1,\dots,F_s$ be finitely generated relatively quasiconvex subgroups of $G$. 
    For ease of notation we write \(m = s - 2\), \(Q = F_1, R = F_2\) and \(T_i = F_{i+2}\), for \(i \in \{1, \dots, m\}\). 
    Choose a finite generating set $X$ for $G$ and let $\delta \in \NN$ be a hyperbolicity constant for the Cayley graph $\ga$, where $\mathcal{H}=\bigsqcup_{\nu \in \Nu} (H_\nu\setminus\{1\})$. 
    Denote by $\varepsilon \ge 0$ a common quasiconvexity constant for $Q,R,T_1, \dots, T_m$.

    Arguing by contradiction, suppose that the subset $QRT_1 \dots T_m=F_1 \dots F_s$ is not separable in $G$. 
    Then there exists \(g \in G \setminus QRT_1 \dots T_m\) such that \(g\) belongs to the profinite closure of \(QRT_1 \dots T_m\) in $G$.
    Let us fix the following notation for the remainder of the proof:
    \begin{itemize}
        \item \(c_0 = \max\{C_0,14\delta\} \geq 0\), where \(C_0\) is the constant obtained from Lemma~\ref{lem:multicoset_bdd_inn_prod};
        \item \(c_3 = c_3(c_0) \geq 0\) is the constant obtained from Lemma~\ref{lem:concat};
        \item \(\lambda = \lambda(c_0) \geq 1\) and \(c = c(c_0) \geq 0\) are  obtained from Proposition~\ref{prop:shortcutting_quasigeodesic}, applied with the constant \(c_0\);
        \item $\mathcal{P}_1$ is the finite family of maximal parabolic subgroups of $G$ from Notation~\ref{not:C_1-P_1};
        \item $\mathcal{U}_1$ is the finite collection of finitely generated relatively quasiconvex subgroups of $G$ from Notation~\ref{not:Li_def};
        \item $A=\abs{g}_X + 1$ and \(\eta = \eta(\lambda,c,A) \geq 0\) is obtained from Lemma~\ref{lem:qgds_with_long_comps};
        \item \(\zeta = \zeta(\eta) \geq 0\), \(\Theta_1 = \Theta_1(\eta) \geq 0\), \(C_3 = C_3(\eta) \geq 0\)  and  \(B_3 = B_3(\eta) \geq 0\) are the constants obtained from Proposition~\ref{prop:mcs_shortcutting_quasigeodesic};
        \item $B=\max\{B_3(\eta),(4A+c_3)\Theta_1\} $ and $C=C_3(\eta)$.
    \end{itemize}

    Observe that, by the induction hypothesis, the product $RT_1 \dots T_l$ is separable in $G$, for every $l=0,\dots,m$.
    Let \(L_1 \leqslant_f G\) be the finite index subgroup obtained from Proposition~\ref{prop:msc_sep->metric}, applied with finite families $\mathcal{P}_1$, $\mathcal{U}_1$ and constants $B$, $C$, given above. 
    Note that $S \subseteq L_1$, and define \(Q' = Q \cap L_1 \leqslant_f Q\).
    Again, by Proposition~\ref{prop:msc_sep->metric}, there is a finite index subgroup \(M_1 \leqslant_f L_1\) such that $Q' \subseteq M_1$ and for any \(M'\leqslant_f M_1\), with \(Q' \subseteq M'\), the subgroups \(Q'\) and \(R' = R \cap M' \leqslant_f R\) satisfy the conclusion of  Proposition~\ref{prop:msc_sep->metric}. 
    
    Let \(E = E(\eta,Q',B) \geq 0\) be the constant provided by Proposition~\ref{prop:mcs_shortcutting_quasigeodesic}.
    Let \(\{N_j \, | \, j \in \NN\}\) be an enumeration of the finite index subgroups of \(M_1\) containing \(Q'\), and define the subgroups
    \begin{equation}
    \label{eq:Ri_def}
       M'_i= \bigcap_{j=1}^i N_j \leqslant_f L'~\text{ and }~ R'_i =  M'_i \cap R\leqslant_f R,~i \in \NN.
    \end{equation}
    Note that for every $i \in \NN$,  $Q' \subseteq M'_i$, so the subgroups $Q'$ and $R'_i$ satisfy the conclusion of Proposition~\ref{prop:msc_sep->metric}. 
    In particular, the subgroup \(\langle Q', R_i' \rangle \) is relatively quasiconvex (and finitely generated) in $G$, and \(Q'\), \(R'_i\) satisfy conditions \descref{C1}-\descref{C4}, \descref{C2-m} and \descref{C5-m} with constants \(B\), \(C\) and families \(\mathcal{P}_1\), \(\mathcal{U}_1\), defined above. 
    For each \(i \in \NN\), consider the subset
    \[
        K_i = Q \langle Q', R'_i \rangle R T_1 \dots T_m.
    \]    
    Choose coset representatives \(x_1, \dots, x_a \in Q\) and \(y_{i,1}, \dots, y_{i,b_i} \in R\) such that \(Q=\bigcup_{j=1}^a x_jQ'\) and $R=\bigcup_{k=1}^{b_i} R'_i\, y_{i,k}$. 
    Then
    \begin{equation*}
    \label{eq:QQ'RiR}
        Q \langle Q', R'_i \rangle R = \bigcup_{j = 1}^a \bigcup_{k=1}^{b_i} x_j \langle Q', R'_i \rangle y_{i,k},
    \end{equation*}
    hence \(K_i\) may be written as the finite union  
    \begin{equation*}
           K_i = \bigcup_{j = 1}^a \bigcup_{k=1}^{b_i} x_j \langle Q', R'_i \rangle y_{i,k} T_1 \dots T_m.
    \end{equation*}
    Therefore, for every \(i \in \NN\), $K_i$ is separable in $G$ by Remark~\ref{rem:qc_prod_sep} and the induction hypothesis.
    Since each \(K_i\) contains \(Q R T_1 \dots T_m\) and \(g\) is in the profinite closure of \(Q R T_1 \dots T_m\), it must be the case that \(g \in K_i\), for every \(i \in \NN\).
    We will show that considering sufficiently large values of \(i\) leads to a contradiction.
    
    For each \(i \in \NN\), let \(\mathcal{S}_i\) be the set of path representatives of \(g\) in \(K_i = Q \langle Q', R'_i \rangle R T_1 \dots T_m\) (see Definition~\ref{def:multicoset_path_reps}, where $R'$ is replaced by $R'_i$).
    We will now consider two  cases.

    \medskip
    \underline{\emph{Case 1:}} 
        there exists $i \in \NN$ such that $\displaystyle \inf_{p' \in \mathcal{S}_i} th_X(p') \ge E$.

        Choose a path representative of minimal type \(p = q p_1 \dots p_n r t_1 \dots t_m \) for \(g\) in $K_i$.  
        Note that $n \ge 1$  and $\elem{p_1} \in R'_i \setminus S$ because $g \notin QRT_1 \dots T_m$ (see Remark~\ref{rem:multicoset_path_reps}). 
        By the assumptions of Case~1 and the above construction, we can apply Proposition~\ref{prop:mcs_shortcutting_quasigeodesic} to conclude that \(p\) is \((B,c_0,\zeta,\Theta_1)\)-tamable and the shortcutting \(\Sigma(p,\Theta_1) = f_0 e_1 f_1 \dots f_{l-1} e_l f_l\), obtained from Procedure~\ref{proc:shortcutting}, is \((\lambda,c)\)-quasigeodesic without backtracking, with \(\abs{e'_k}_X \geq \eta\) for each \(k = 1, \dots, l\) (where \(e'_k\) denotes the \(\mathcal{H}\)-component of \(\Sigma(p,\Theta_1)\) containing \(e_k\)).

        If \(l > 0\), then applying Lemma~\ref{lem:qgds_with_long_comps} to the path \(\Sigma(p,\Theta_1)\) gives
        \[
            \abs{g}_X = \abs{p}_X = \abs{\Sigma(p,\Theta_1)}_X \geq A > \abs{g}_X,
        \]
        by the choice of \(\eta\), which gives a contradiction. 
    
        Therefore it must be that \(l = 0\).  
        Then \(p\) is \((4,c_3)\)-quasigeodesic by Lemma~\ref{lem:beta_quasigeodesic} and, according to Remark~\ref{rem:shortcutting}(c), no segment of \(p\) contains an \(\mathcal{H}\)-component \(h\) with \(\abs{h}_X \geq \Theta_1\).
        By the quasigeodesicity of \(p\) and the fact that \(p_1\) is a subpath of \(p\), we have
        \begin{equation}
        \label{eq:p_qgd}
            \abs{g}_{X\cup\mathcal{H}} = \abs{p}_{X\cup\mathcal{H}} \geq \frac{1}{4}(\ell(p) - c_3) \geq \frac{1}{4}(\ell(p_1) - c_3).
        \end{equation}
        Applying Lemma~\ref{lem:rel_geods_with_short_comps} to the geodesic \(p_1\) in $\ga$ we obtain
        \begin{equation}
        \label{eq:r_len_bd}
             \ell(p_1) \geq \frac{1}{\Theta_1}\abs{p_1}_X \geq \frac{B}{\Theta_1} \ge 4A+c_3,
        \end{equation}
        where the second inequality follows from the fact that $\elem{p_1} \in R'_i \setminus S$ and Lemma~\ref{lem:C2_implies_old_C2}.
        Combining (\ref{eq:p_qgd}) and (\ref{eq:r_len_bd}), we get
        \[
            \abs{g}_X \ge \abs{g}_{X\cup\mathcal{H}} \geq \frac{1}{4}(4A+c_3-c_3)= A> \abs{g}_{X},
        \]
        which is a contradiction.

    \medskip
    \underline{\emph{Case 2:}} 
        for all $i \in \NN$ we have $\displaystyle \inf_{p' \in \mathcal{S}_i} th_X(p') < E$.

        Then for each \(i \in \NN\) there is a path representative \(p_i = q_i p_{1,i} \dots p_{n_i,i} r_i t_{1,i} \dots t_{m,i} \in \mathcal{S}_i\) for $g$ such that $th(p_i) \le E$. 
        It must either be the case that \(\displaystyle \liminf_{i \to \infty}\abs{r_i}_X \le E\) or \(\displaystyle \liminf_{i \to \infty}\abs{t_{j,i}}_X \le E\), for some \(j \in \{ 1, \dots, m\}\).
        We will consider the former case, as the latter is very similar.

        Since there are only finitely many elements \(x \in G\) with \(\abs{x}_X \le E \), we may pass to a subsequence $(p_{i_k})_{k \in \NN}$ such that \(\elem{r_{i_k}} = y \in R\) is some fixed element, for all \(k \in \NN\).
        It follows that
        \begin{equation}
        \label{eq:g_in_prod_case_1}
            g = \elem{p_{i_k}} \in Q \langle Q', R'_{i_k} \rangle y T_1 \dots T_m,~ \text{ for each } k \in \NN.
        \end{equation}

        Now, \(g \notin Q y T_1 \dots T_m\) (as $y \in R$), and the subset \(Q y T_1 \dots T_m\) is separable in \(G\) by
        the induction hypothesis and Remark~\ref{rem:qc_prod_sep}.
        By Lemma~\ref{lem:sep->large_minx}(a), there is a finite index normal subgroup \(O \lhd_f G\) such that \(g \notin Q O y T_1 \dots T_m\).
        The subgroup \(M_1 \cap QO\) has finite index in $M_1$ and contains $Q'$, therefore $M_1 \cap QO=N_{j_0}$, for some $j_0 \in \NN$. 
        
        Choose $k \in \NN$ such that $i_k \ge j_0$, so that $M'_{i_k} \subseteq N_{j_0} \subseteq QO$ (see \eqref{eq:Ri_def}). Then $R_{i_k}'=M'_{i_k} \cap R\subseteq QO $, hence
        \begin{equation}
        \label{eq:containment_for_case_2}
            Q \langle Q',R'_{i_k} \rangle yT_1 \dots T_m \subseteq QOyT_1 \dots T_m.    
        \end{equation}
        Since  \(g \notin Q O y T_1 \dots T_m\), inclusions \eqref{eq:g_in_prod_case_1} and \eqref{eq:containment_for_case_2} contradict each other.

    \medskip
    We have arrived to a contradiction at each of the two cases, hence the proof is complete.
\end{proof}


\section{New examples of product separable groups}
\label{sec:ex_prod_sep}
In this section we prove Theorem~\ref{thm:prod_sep}, which will follow from the three propositions below.

\begin{proposition}
\label{prop:limit_gps_are_prod_sep} 
    Limit groups are product separable.
\end{proposition}

\begin{proof} 
    Dahmani \cite{Dahmani_Conv} and, independently, Alibegovi\'c \cite{Alib} proved that every limit group is hyperbolic relative to a collection of conjugacy class representatives of its maximal non-cyclic finitely generated abelian subgroups. 

    Moreover, Wilton \cite{WiltonLimitGps} showed that limit groups are LERF and Dahmani \cite{Dahmani_Conv} showed they are \emph{locally quasiconvex} (that is, each of their finitely generated subgroups is relatively quasiconvex with respect to the given peripheral structure). 
    Therefore our Theorem~\ref{thm:RZs} yields that limit groups are product separable.
\end{proof}

Finitely generated Kleinian groups are not always locally quasiconvex, and we require the following two lemmas to deal with the case when one of the factors is not relatively quasiconvex.

\begin{lemma} \label{lem:Kl-1}
Let $N$ be a group and $n \ge 2$ be an integer. Suppose that  $H_1,\dots,H_n$ are subgroups of $N$ such that $H_i \lhd N$, for some $i\in \{1,\dots,n\}$, and the image of the product $H_1 \dots H_{i-1}H_{i+1} \dots H_n$ is separable in $N/H_i$. Then $H_1\dots H_n$ is separable in $N$.
\end{lemma}

\begin{proof} Let $\varphi:N \to N/H_i$ denote the natural epimorphism. By the assumptions, the subset $S=\varphi(H_1 \dots H_{i-1}H_{i+1} \dots H_n)$ is separable in $N/H_i$. Observe that \[H_1\dots H_n=(H_1 \dots H_{i-1}H_{i+1} \dots H_n) H_i=\varphi^{-1}(S),\] as $H_i \lhd N$, whence 
$H_1\dots H_n$ is closed in the profinite topology on $N$ because group homomorphisms are continuous with respect to profinite topologies.
\end{proof}

\begin{lemma}
\label{lem:Kl-2}
    Let \(G\) be a group with finitely generated subgroups \(F_1, \dots, F_n \leqslant G\), $n \ge 2$.
    Suppose that there exists a finite index subgroup \(G' \leqslant_f G\) and an index \(i \in \{ 1, \dots, n\}\) such that \(F_i'=F_i \cap G' \lhd G'\) and \(G'/F_i'\) has property \(RZ_{n-1}\).
    Then the product \(F_1 \dots F_n\) is separable in $G$.
\end{lemma}

\begin{proof} Let $N \lhd_f G$ be a finite index normal subgroup contained in $G'$, and set $H_j=F_j \cap N$, for $j=1,\dots,n$. 

Since $|F_j:H_j|<\infty$, for each $j=1,\dots,n$, the product $F_1\dots F_n$ can be written as a finite union of subsets of the form $h_1H_1h_2H_2 \dots h_n H_n$, where $h_1,\dots,h_n \in G$. Observe that 
\begin{equation*}\label{eq:prod_of_cosets}
h_1H_1h_2H_2 \dots h_n H_n= H_1^{g_1} H_2^{g_2} \dots H_n^{g_n} g_n, 
\end{equation*}
where $g_j=h_1 \dots h_j \in G$, $j=1,\dots,n$. Thus, in view of Remark~\ref{rem:sep_props}, in order to prove the separability of $F_1\dots F_n$ in $G$ it is enough to show that the product $H_1^{g_1} H_2^{g_2} \dots H_n^{g_n}$ is separable, for arbitrary $g_1,\dots,g_n \in G$.

Given any elements $g_1,\dots,g_n \in G$, the subgroups $H_1^{g_1}, H_2^{g_2}, \dots , H_n^{g_n} \leqslant G$ are finitely generated and are contained in $N$. Moreover, since the subgroup $H_i=F_i \cap N=F_i' \cap N$ is normal in $N$ and $N \leqslant G'$ is normal in $G$, we see that $H_i^{g_i} \lhd N$ and \[ N/H_i^{g_i}=N^{g_i}/H_i^{g_i} \cong N/H_i \leqslant G'/F_i'.\]

Therefore the group $N/H_i^{g_i}$ has RZ$_{n-1}$, as a subgroup of $G'/F_i'$, so the image of the product $H_1^{g_1} \dots H_{i-1}^{g_{i-1}}H_{i+1}^{g_{i+1}} \dots H_n^{g_n}$ is separable in $N/H_i^{g_i}$. Lemma~\ref{lem:Kl-1} now implies that $H_1^{g_1} H_2^{g_2} \dots H_n^{g_n}$ is separable in $N$, hence it is also separable in $G$ by Lemma~\ref{lem:induced_top}(b). As we observed above, the latter yields the separability of $F_1 \dots F_n$ in $G$, as required.
\end{proof}

\begin{proposition}
\label{prop:Kleinian_gps_are_prod_sep} Finitely generated Kleinian groups are product separable.
\end{proposition}

\begin{proof} Let $G$ be a finitely generated discrete subgroup of $\mathrm{Isom}(\mathbb{H}^3)$. We will first reduce the proof to the case when $G\backslash \mathbb{H}^3$ is a finite volume manifold. This idea is inspired by the argument of Manning and Mart\'{i}nez-Pedroza used in the proof of \cite[Corollary~1.5]{MMPSep}.

Using Selberg's lemma, we can find a torsion-free finite index subgroup $K \leqslant G$. Since product separability of $K$ implies that of $G$ (\cite[Lemma~11.3.5]{Ribes-book}),  without loss of generality we can assume that $G$ is torsion-free. It follows that $G$ acts freely and properly discontinuously on $\mathbb{H}^3$, so that $M=G \backslash \mathbb{H}^3$ is a complete hyperbolic $3$-manifold. 

If $M$ has infinite volume then, by \cite[Theorem~4.10]{Matsuzki-Taniguchi}, $G$ is isomorphic to a geometrically finite Kleinian group. Thus we can further assume that $G$ is geometrically finite, which allows us to apply a theorem of Brooks \cite[Theorem~2]{Brooks_Extensions} to find an embedding of $G$ into a torsion-free Kleinian group $G^*$ such that $G^* \backslash \mathbb{H}^3$ is a finite volume manifold. If $G^*$ is product separable, then so is any subgroup of it, hence we have made the promised reduction. 

Thus we can suppose that $G =\pi_1(M)$, for a hyperbolic $3$-manifold $M$ of finite volume. The tameness conjecture, proved by Agol \cite{Agol-tame} and Calegari-Gabai \cite{Cal-Gab}, combined with a result of Canary \cite[Corollary~8.3]{Canary}, imply that any finitely generated subgroup $F \leqslant G$ is either geometrically finite or  is a virtual fibre subgroup. The latter means that there is a finite index subgroup $G' \leqslant_f G$ such that $F'=F \cap G' \lhd G'$ and $G'/F' \cong \mathbb{Z}$.

By \cite[Theorem~3.7]{Matsuzki-Taniguchi}, $G$ is a geometrically finite subgroup of $\mathrm{Isom}(\mathbb{H}^3)$, hence it is finitely generated and hyperbolic relative to a finite collection of finitely generated virtually abelian subgroups, each of which is product separable by \cite[Lemma~11.3.5]{Ribes-book}. Moreover, by \cite[Corollary~1.6]{HruskaRHCG}, a subgroup of $G$ is relatively quasiconvex if and only if it is geometrically finite. Finally, $G$ is LERF (and, hence, QCERF) by \cite[Corollary~9.4]{Agol}.

Let $F_1, \dots, F_n$ be finitely generated subgroups of $G$, $n \ge 2$. If $F_j$ is geometrically finite, for all $j=1,\dots,n$, then the product $F_1 \dots F_n$ is separable in $G$ by Theorem~\ref{thm:RZs}. Thus we can suppose that $F_i$ is not geometrically finite, for some $i \in \{1,\dots,n\}$. By the above discussion, in this case $F_i$ must be a virtual fibre subgroup of $G$. Since $\mathbb{Z}$ is product separable, we can apply Lemma~\ref{lem:Kl-2} to conclude that $F_1 \dots F_n$ is separable in $G$, completing the proof. 
\end{proof}

\begin{proposition}
\label{prop:balanced_gps_are_prod_sep} 
    Let $G$ be the fundamental group of a finite graph of free groups with cyclic edge groups. 
    If $G$ is balanced then it is product separable.
\end{proposition}

Limit groups and Kleinian groups are hyperbolic relative to virtually abelian subgroups. The peripheral subgroups from relatively hyperbolic structures on groups in Proposition~\ref{prop:balanced_gps_are_prod_sep} will be fundamental groups of graphs of cyclic groups, which motivates the next auxiliary lemma.

\begin{lemma}
\label{lem:graph_of_cyclic_gps} 
    Suppose that $G$ is the fundamental group of a finite graph of infinite cyclic groups. If $G$ is balanced then it is product separable. 
\end{lemma}

\begin{proof} 
    Suppose that $G=\pi_1(G_{-},\Gamma)$, where $(G_{-},\Gamma)$ is a graph of groups, associated to a finite connected graph $\Gamma$ with vertex set $V \Gamma$ and edge set $E\Gamma$. 
    According to the assumptions, each vertex group $G_v$, $v \in V\Gamma$, is infinite cyclic. 
    As usual, we use $G_e$ to denote the edge group corresponding to an edge $e \in E\Gamma$ (see Dicks and Dunwoody \cite[Section I.3]{Dicks-Dunw} for the definition and general theory of graphs of groups).

    If $|E\Gamma|=0$ then $G$ is cyclic and, thus, product separable. Let us proceed by induction on $|E\Gamma|$.

    Assume first that one of the edge groups $G_e$ is trivial. 
    If removing $e$ disconnects $\Gamma$ then $G$ splits as a free product $G_1*G_2$, where $G_1$, $G_2$ are the fundamental groups of finite graphs of infinite cyclic groups corresponding to the two connected components of $\Gamma\setminus\{e\}$. 
    Otherwise, $G \cong G_1*{G_2}$, where $G_1$ the fundamental group of a finite graph of infinite cyclic groups corresponding to the graph $\Gamma\setminus\{e\}$ and $G_2$ is infinite cyclic. 
    Moreover, $G_1$ and $G_2$ will be balanced as subgroups of a balanced group $G$. Hence $G_1$ and $G_2$ will be product separable by induction, so $G \cong G_1*G_2$ will be product separable by Coulbois' theorem \cite[Theorem~1]{Coulb}.

    Therefore we can assume that every edge group $G_e$ is infinite cyclic. This means that $G$ is a \emph{generalised Baumslag-Solitar group}. 
    The assumption that $G$ is balanced now translates into the assumption that $G$ is \emph{unimodular}, using Levitt's terminology from \cite{Levitt-GBS}. 
    We can now apply \cite[Proposition~2.6]{Levitt-GBS} to deduce that $G$ has a finite index subgroup $K$ isomorphic to the direct product $F \times \mathbb{Z}$, where $F$ is a free group.

    Now, $K \cong F\times \mathbb{Z}$ is product separable by You's result \cite[Theorem~5.1]{You}, hence $G$ is product separable as a finite index supergroup of $K$ (see \cite[Lemma~11.3.5]{Ribes-book}). Thus the lemma is proved.
\end{proof}

\begin{proof}[Proof of Proposition~\ref{prop:balanced_gps_are_prod_sep}] 
    Suppose that $G$ splits as the fundamental group of a finite graph of free groups $(G_-,\Gamma)$ with cyclic edge groups.

    Without loss of generality we can assume that each vertex group is a finitely generated free group (in particular, $G$ is finitely generated). 
    Indeed, otherwise $G \cong G_1*F$, where $G_1$ is the fundamental group of a finite graph of finitely generated free groups with cyclic edge groups and $F$ is free (this follows from the fact that any element of a free group is the product of only finitely many free generators). 
    In this case we can deduce the product separability of $G$ from the product separability of $G_1$ and $F$ by \cite[Theorem~1]{Coulb} (recall that $F$ is product separable by Ribes and Zalesskii \cite[Theorem ~2.1]{RibesZal}).

    Now, for each vertex group $G_v$, choose and fix a finite family of maximal infinite cyclic subgroups $\mathbb{P}_v$ such that

    \begin{itemize}
        \item[(a)] no two subgroups from $\mathbb{P}_v$ are conjugate in $G_v$;
        \item[(b)] for every edge $e$ incident to $v$ in $\Gamma$, the image of the cyclic group $G_e$ in $G_v$ is conjugate into one of the subgroups from $\mathbb{P}_v$.
    \end{itemize}

    Condition (a) means that each $G_v$ is hyperbolic relative to the finite family $\mathbb{P}_v$ (for example, by \cite[Theorem~7.11]{BowditchRHG}), and condition (b) means that each edge group of the given splitting of $G$ is parabolic in the corresponding vertex groups. 
    Therefore we can apply the work of Bigdely and Wise \cite[Theorem~1.4]{Big-Wise} to conclude that $G$ is hyperbolic relative to a finite collection of subgroups $\mathbb{Q}$, where each $Q \in \mathbb{Q}$ acts cocompactly on a \emph{parabolic tree} (see \cite[Definition~1.3]{Big-Wise}) with vertex stabilisers conjugate to elements of $\bigcup_{v \in V\Gamma} \mathbb{P}_v$ and edge stabilisers conjugate to elements of $\{ G_e \mid e \in \Gamma \}$. 
    The structure theorem for groups acting on trees (\cite[Section~I.4.1]{Dicks-Dunw}) implies that every  $Q \in \mathbb{Q}$ is isomorphic to the fundamental group of a finite graph of infinite cyclic groups. 
    Since $Q$ is balanced, being a subgroup of $G$, we can apply Lemma~\ref{lem:graph_of_cyclic_gps} to conclude that each  $Q \in \mathbb{Q}$ is product separable.
    By Wise's result \cite[Theorem~5.1]{Wise-balanced} $G$ is LERF, hence we can apply our Theorem~\ref{thm:RZs} to deduce that the product of a finite number of finitely generated relatively quasiconvex subgroups is separable in $G$. 

    To establish the product separability of $G$ it remains to show that it is locally quasiconvex. To achieve this we will again use the results of Bigdely and Wise. 
    More precisely, according to \cite[Theorem~2.6]{Big-Wise}, a subgroup of $G$ is relatively quasiconvex if it is \emph{tamely generated}.  

    Let $H \leqslant G$ be a finitely generated subgroup. 
    The splitting of $G$ as the fundamental group of the graph of groups $(G_-,\Gamma)$ induces a splitting of $H$ as the fundamental group of a graph of groups $(H_-,\Delta)$, where for each vertex $u \in V\Delta$ the stabiliser $H_u$ is equal to $H \cap {G_v}^g$, for some $v \in V\Gamma$ and some $g \in G$. 
    Moreover, the graph $\Delta$ is finite, because $H$ is finitely generated (see \cite[Proposition~I.4.13]{Dicks-Dunw}). 
    Note that every edge group from $(H_-,\Delta)$ is cyclic, hence each vertex group $H_u$, $u \in V\Delta$, must be finitely generated as $H$ is finitely generated (see \cite[Lemma~2.5]{Big-Wise}).

    According to \cite[Definition~0.1]{Big-Wise}, $H$ is tamely generated if for every $u \in V\Delta$ the subgroup $H_u=H \cap {G_v}^g$ is relatively quasiconvex in ${G_v}^g$, equipped with the peripheral structure ${\mathbb{P}_v}^g$. 
    But the latter is true because ${G_v}^g$ is a finitely generated free group, so any finitely generated subgroup is undistorted, and hence it is relatively quasiconvex with respect to any peripheral structure on ${G_v}^g$, by \cite[Theorem~1.5]{HruskaRHCG}. 
    Thus every finitely generated subgroup $H \leqslant G$ is tamely generated, and so it is relatively quasiconvex in $G$ by \cite[Theorem~2.6]{Big-Wise}.
\end{proof}

\begin{remark}
    In the case when the graph of groups has two vertices and one edge (so that $G$ is a free amalgamated product of two free groups over a cyclic subgroup), Proposition~\ref{prop:balanced_gps_are_prod_sep} was originally proved by Coulbois in his thesis: see \cite[Theorem~5.18]{Coulbois-thesis}. 
    We can use similar methods to recover another result of Coulbois: if $G=H*_{C} F$, where $H$ is product separable, $F$ is free and $C$ is a maximal cyclic subgroup in $F$ then $G$ is product separable \cite[Theorem~5.4]{Coulbois-thesis}.
    Indeed, in this case $G$ will be hyperbolic relative to $\mathbb{Q}=\{H\}$ and will be LERF by Gitik's theorem \cite[Theorem~4.4]{Gitik-LERF}. 
    As in the proof of Proposition~\ref{prop:balanced_gps_are_prod_sep}, the results from \cite{Big-Wise} imply that $G$ is locally quasiconvex. Therefore $G$ is product separable by Theorem~\ref{thm:RZs}.
\end{remark}

\begin{remark} Using recent work of Shepherd and Woodhouse \cite[Theorem~1.2]{Shep-Wood}, Proposition~\ref{prop:balanced_gps_are_prod_sep} can be immediately extended to balanced groups $G$ that split as fundamental groups of finite graphs of groups with virtually free vertex groups and virtually cyclic edge groups. In fact, by \cite[Proposition~3.13]{Shep-Wood}, $G$ has a torsion-free finite index subgroup $K$. Then $K$ is balanced and is isomorphic to the fundamental group of a finite graph of free groups with cyclic edge groups. So the product separability of $G$ follows by combining  Proposition~\ref{prop:balanced_gps_are_prod_sep} with \cite[Lemma~11.3.5]{Ribes-book}.
\end{remark}

\printbibliography

\end{document}